%% file: main.tex
\documentclass[11pt]{article}
\usepackage[a4paper,margin=1in]{geometry}

\usepackage{amsmath,amssymb,bbm,amsthm}
\usepackage{mathtools}
\usepackage{graphicx}
\usepackage{array}
\usepackage{placeins}

\graphicspath{{./}{figures/}{texcheck/}}
\DeclareGraphicsExtensions{.pdf,.png,.jpg,.jpeg}

\usepackage{algorithm}
\usepackage{algpseudocode}
\usepackage{enumerate}
\usepackage{multirow}
\usepackage{soul}
\usepackage{subcaption}
\usepackage[numbers]{natbib}
\usepackage{xcolor}

\usepackage{authblk}

\usepackage{hyperref}
\usepackage[capitalize,noabbrev]{cleveref}
\makeatletter
\providecommand{\theHALG@line}{}
\renewcommand{\theHALG@line}{\thealgorithm.\arabic{ALG@line}}
\makeatother

\numberwithin{equation}{section}

\newcommand{\missingfigure}[1]{%
  \fbox{\parbox[c][0.25\textheight][c]{0.8\linewidth}{%
    \centering Missing figure file: \texttt{\detokenize{#1}}%
  }}%
}

\newcommand{\safeincludegraphics}[2][]{%
  \IfFileExists{#2}{\includegraphics[#1]{#2}}{%
    \IfFileExists{#2.pdf}{\includegraphics[#1]{#2.pdf}}{%
      \IfFileExists{#2.png}{\includegraphics[#1]{#2.png}}{%
        \missingfigure{#2}%
      }%
    }%
  }%
}

\newcommand{\safeinput}[1]{%
  \IfFileExists{#1}{\input{#1}}{%
    \fbox{\parbox[c][3em][c]{0.88\linewidth}{%
      \centering Missing generated file: \texttt{\detokenize{#1}}%
    }}%
  }%
}

\date{}

\DeclareMathOperator*{\argmin}{arg\,min}
\DeclareMathOperator*{\Lip}{\rm{Lip}}
\DeclareMathOperator*{\diam}{\rm{diam}}
\DeclareMathOperator*{\essinf}{ess\,inf}
\DeclareMathOperator*{\esssup}{ess\,sup}
\DeclareMathOperator*{\essosc}{ess\,osc}

\newcommand{\HK}{\mathsf{H\kern-3pt K}}
\newcommand{\LET}{\mathsf{L\kern-2pt E\kern-2pt T}}

\newtheorem{theorem}{Theorem}[section]
\newtheorem{lemma}[theorem]{Lemma}
\newtheorem{corollary}[theorem]{Corollary}
\newtheorem{proposition}[theorem]{Proposition}

\newtheorem{remark}[theorem]{Remark}
\newtheorem{assumption}[theorem]{Assumption}

\usepackage{booktabs}
\newtheorem{introresult}{Theorem}

\crefname{introresult}{Theorem}{Theorems}
\Crefname{introresult}{Theorem}{Theorems}
\input{sections/experiment_setup}
\begin{document}

% Compact display spacing, applied document-wide.
\setlength{\abovedisplayskip}{6pt plus 2pt minus 2pt}
\setlength{\belowdisplayskip}{6pt plus 2pt minus 2pt}
\setlength{\abovedisplayshortskip}{3pt plus 2pt minus 1pt}
\setlength{\belowdisplayshortskip}{3pt plus 2pt minus 1pt}

\title{A Neural JKO Scheme for Hellinger–Kantorovich Gradient Flows via Monge-Growth Pairs}

% with Fully Implicit Interactions}
% \author{Geuntaek Seo, Cheolhyeong Kim, Hwijae Son, Hyung Ju Hwang}

\author{
Geuntaek Seo\textsuperscript{1,*}\quad
Cheolhyeong Kim\textsuperscript{2,*}\quad
Hwijae Son\textsuperscript{3,\(\dagger\)}\quad
Hyung Ju Hwang\textsuperscript{1,\(\dagger\)}
\\
{\small\textsuperscript{1}Department of Mathematics, Pohang University of Science and Technology}
\\
{\small\textsuperscript{2}Samsung Electronics}
\\
{\small\textsuperscript{3}Department of Mathematics, Konkuk University}
\\
{\small\rule{0pt}{3ex}
\texttt{gtseo@postech.ac.kr}\quad
\texttt{hspst440@gmail.com}\quad
\texttt{hwijaeson@konkuk.ac.kr}\quad
\texttt{hjhwang@postech.ac.kr}
}
\\
{\small
\rule{0pt}{3ex}\textsuperscript{*}These authors contributed equally to this work.
\\
\textsuperscript{\(\dagger\)}Corresponding authors.
}
}

\maketitle

\begin{abstract}
We develop a mesh-free neural JKO scheme for advection--reaction--diffusion equations with a gradient-flow structure in the Hellinger--Kantorovich ($\HK$) geometry of unbalanced optimal transport.
Each update is parametrized by a spatial map and a mass-changing factor,
allowing spatial redistribution and local mass creation or loss to be
treated jointly within a single variational step.
Their cone action bounds the squared $\HK$ distance from above, yielding a sufficient condition for discrete energy dissipation through comparison with the identity pair.
Minimizing the pair objective over all admissible pairs recovers the exact JKO minimum when the source and a minimizer have positive densities. We establish existence and mass bounds for JKO minimizers and, under additional assumptions, obtain positivity and regularity together with a discrete Euler--Lagrange equation and a metric--dissipation identity. The self-consistent chemical potential is then nonincreasing along an optimal map. There exist parametric pairs whose endpoint densities and objective values converge to those of an exact JKO minimizer, provided a regular-pair approximation hypothesis holds. Finally, we show that a primal--dual gap controls objective suboptimality and, for Boltzmann entropy, the $L^1$ density error, assuming exact-step regularity, positive-semidefinite interactions, and global dual feasibility.
Numerical experiments examine pointwise agreement with the PDE, energy dissipation, and the roles of transport, reaction, and fully implicit interactions. 
\end{abstract}

\vspace{1em}
\phantom{\textbf{Keywords:}} advection--reaction--diffusion equations, neural networks, minimizing movements

% \tableofcontents

\section{Introduction}\label{sec:introduction}

\subsection{Hellinger--Kantorovich gradient flows and the computational problem}
Variational time discretization provides a natural way to construct gradient-flow evolutions.  In Wasserstein space, the Jordan--Kinderlehrer--Otto (JKO) scheme identifies the Fokker--Planck equation as a gradient flow of free energy \cite{JKO98,AGS08}.  Classical Wasserstein geometry preserves total mass, whereas advection--reaction--diffusion systems generally combine spatial transport with local creation or destruction of mass.

The Hellinger--Kantorovich distance, denoted by $\HK$ and also called
the Wasserstein--Fisher--Rao distance, is an unbalanced optimal transport metric that couples transport with reaction on finite nonnegative measures
\cite{KMV16,CPSV2018,CPSV18,LMS16,LMS18}.  
The distance admits static, cone, dynamic, and dual formulations. This geometry has been used to develop minimizing-movement schemes for advection--reaction--diffusion equations and to study hard-congestion limits \cite{GM17,Fleissner21,DiMarinoChizat20}. Subsequent work has established Monge representations, regularity of dual potentials, and local geometric properties \cite{LMS23,GGV25,LM26}.

We study the $\HK$ gradient flow of a free energy $\mathcal F$ consisting of internal, external, and nonlocal interaction terms. For positive constants $a$ and $b$, the formal flow takes the form
\begin{equation}\label{pde-intro}
\partial_t\rho=a\,\operatorname{div}(\rho\nabla g_\rho)-b\,\rho g_\rho,
\qquad g_\rho=U'(\rho)+V+W[\rho].
\end{equation}
Here $a$ and $b$ weight spatial transport and local mass change, respectively. Throughout this paper, we take $a=1$ and $b=4$, corresponding to the normalization $\HK=\HK_{1,4}$. Both terms are driven by the \emph{same} chemical potential $g_\rho$. The associated variational time discretization is the $\HK$-JKO scheme
\begin{equation*}
\mu_{k+1}^\tau\in\argmin_{\sigma\in\mathcal M_+(\Omega)}
\left\{\mathcal F(\sigma)+\frac1{2\tau}\HK^2(\sigma,\mu_k^\tau)\right\}.
\end{equation*}
Here $\mathcal M_+(\Omega)$ denotes the space of finite nonnegative Borel measures on the spatial domain $\Omega$.
Whenever a minimizer exists and the current energy is finite, it satisfies the one-step dissipation inequality.  Computing a step nevertheless requires minimizing an objective that combines the energy of the next measure with the unbalanced optimal transport cost between the current and next measures.  Grid-based discretizations are effective in low dimensions but scale poorly with dimension, motivating mesh-free parametric methods \cite{Mokrov21,PKSH23}.

A parametric method must both represent the next measure and approximate the metric term. Whether the metric approximation is an upper or a lower bound matters for energy descent. Restricting the maximization in the $\HK$ dual problem to a prescribed class of feasible potentials gives a lower bound on the squared distance, so descent of the resulting objective need not imply descent of the exact JKO objective. We instead use a feasible primal cone action associated with a spatial map and a growth factor. Under the positive-density hypotheses stated below, optimization over all such pairs is equivalent to the exact JKO problem. Every admissible parametric pair still provides an upper bound on the exact JKO objective at its induced measure.

\subsection{The primal formulation}\label{sec:design-choices}
Our formulation combines two choices. Each update acts on the current measure, so the learned displacement describes a single time step. A feasible primal cone action then allows the same pair to define the candidate measure and bound its exact proximal cost from above (Table~\ref{tab:design-choices}).

\emph{Restarted update.} We call $(T,q)$ a \emph{Monge-growth pair}: $T$ transports positions and $q^2$ changes the mass before transport, giving 
\[T_\#(q^2\mu).
\]
Here $T_\#\nu$ denotes the pushforward of a measure $\nu$ under $T$,
defined by $(T_\#\nu)(A)=\nu(T^{-1}(A))$ for every Borel set
$A\subset\Omega$.
The term ``growth'' includes both mass creation ($q>1$) and mass loss ($q<1$). We call updates with $T=\mathrm{id}$ \emph{pure-reaction} updates; they are labeled ``growth-only'' in the experiments. The scheme is \emph{restarted} in the sense that a fresh pair maps the current measure to the next candidate at each step. Density values and the sampling law must therefore be propagated between steps; Section~\ref{sec:neural-realization-algorithm} gives the exact formulas and the density fitting used in recursive runs.

% {
% \emph{Restarted state.}
% An update is represented by a spatial map $T$ and a nonnegative growth
% factor $q$ through $T_\#(q^2\mu)$.  We train a fresh current-to-next pair
% at every step rather than a cumulative map from the initial law.  The price
% is loss of direct density access from a fixed base law: density values
% propagate by change of variables and the normalized sampling law by the
% growth weights.  For the recursive experiments, the resulting density
% values are fitted in a fixed positive density family, adding a per-step
% projection error; Section~\ref{sec:neural-realization-algorithm} gives the
% exact formulas and implementation.
% }

\emph{Primal metric representation.}  A pair $(T,q)$ is \emph{admissible} for $\mu$ if $T:\Omega\to\Omega$ and $q:\Omega\to[0,\infty)$ are Borel and $\int_\Omega q^2\,d\mu<\infty$.  Its induced measure is
$T_\#(q^2\mu),$
and its Monge-growth action is
\begin{equation}\label{eq:incremental-HK-action}
\mathcal A_\mu(T,q)
:=\int_\Omega
\left[
1+q(x)^2-2q(x)\cos\!\left({\sf d}_\Omega(x,T(x))\wedge\frac\pi2\right)
\right]d\mu(x).
\end{equation}
In the pair formulation, \emph{unrestricted} optimization ranges over all admissible Monge-growth pairs, whereas \emph{restricted} optimization ranges only over a prescribed subclass, such as the admissible pairs representable by the chosen neural network class.
The cone formulation gives the one-sided bound
\begin{equation}\label{eq:incremental-HK-upper-bound}
\HK^2\bigl(\mu,T_\#(q^2\mu)\bigr)\le\mathcal A_\mu(T,q),
\end{equation}
which is tight at an optimal unrestricted pair when the two measures have positive densities.  We therefore optimize
\begin{equation}\label{eq:incremental-pair-objective-abstract}
\mathcal I_\tau(T,q\mid\mu)
:=
\mathcal F\bigl(T_\#(q^2\mu)\bigr)
+\frac1{2\tau}\mathcal A_\mu(T,q).
\end{equation}
Every term of this objective can be written as an integral with respect to the current measure; see \eqref{eq:incremental-internal-expectation}--\eqref{eq:incremental-interaction-expectation}.

\paragraph{Equivalence to the exact JKO problem.}
Writing
\[
\mathcal J_\tau(\sigma\mid\mu)
:=
\mathcal F(\sigma)+\frac1{2\tau}\HK^2(\sigma,\mu),
\]
the cone-action upper bound gives
\begin{equation}\label{eq:intro-primal-majorization}
\mathcal J_\tau\bigl(T_\#(q^2\mu)\mid\mu\bigr)
\le
\mathcal I_\tau(T,q\mid\mu)
\end{equation}
for every admissible pair.  Proposition~\ref{prop:incremental-HK-JKO-reparam} shows that, under the positive-density hypotheses used below, optimization over all admissible pairs has the same optimal value as the exact JKO problem:
\begin{equation*}
\min_{\sigma\in\mathcal M_+(\Omega)}
\mathcal J_\tau(\sigma\mid\mu)
=
\inf_{T,\,q\ge0}
\mathcal I_\tau(T,q\mid\mu).
\end{equation*}
A restricted pair class may overestimate the metric term, but the unrestricted formulation is exactly equivalent to the JKO problem. Since the identity pair has zero action, comparison with that pair gives the descent test of Proposition~\ref{prop:one-step-certificate}. Section~\ref{sec:restarted-formulation} proves both statements. Theorem~\ref{thm:a-posteriori} then combines the primal objective with a globally feasible dual lower bound to estimate the density error. Numerical optimization, quadrature, and density fitting introduce additional errors, discussed in Section~\ref{sec:neural-realization-algorithm}.

\begin{table}[htbp]
\centering\small
\renewcommand{\arraystretch}{1.2}
\begin{tabular}{@{}>{\raggedright\arraybackslash}p{0.17\linewidth}>{\raggedright\arraybackslash}p{0.36\linewidth}>{\raggedright\arraybackslash}p{0.36\linewidth}@{}}
\toprule
 & \textbf{Restricted dual value} \quad($\le\HK^2$)
 & \textbf{Primal cone action} \quad($\ge\HK^2$)\\
\midrule
\textbf{Cumulative state}\newline $\widehat\mu_k=(S_k)_\#(r_k^2\mu_0)$
 & Density can be evaluated from a fixed reference law. A dual potential or $c$-transform is required, and the lower bound alone does not certify exact JKO descent.
 & A feasible pair can certify descent, but cumulative displacement and representation complexity may grow.\\
\midrule
\textbf{Restarted state}\newline $\widehat\mu_{k+1}=T_\#(q^2\widehat\mu_k)$
 & Each update starts from the current measure. Dual optimization is still required, and the lower bound alone does not certify exact JKO descent.
 & \emph{This paper:} one pair defines the next measure and an upper bound on its JKO objective. The unrestricted problem is exact; density and sampling information must be propagated between steps.\\
\bottomrule
\end{tabular}
\caption{Here $(S_k,r_k)$ is the cumulative Monge-growth pair representing the approximate state $\widehat\mu_k$, with $r_k\ge0$.
Restarting parametrizes each update from the current state;
it does not reset accumulated errors.
The one-step descent criterion for the pair endpoint relies on the primal upper bound and comparison with the identity pair.}
\label{tab:design-choices}
\end{table}

\subsection{Main results and contributions}\label{sec:main-results}
We first establish the exact JKO step that the neural update is meant to approximate, then examine the effect of parametric restriction and the accuracy of a computed candidate.

\paragraph{Contributions and analytical ingredients.}
First, Theorem~A establishes existence and mass control, with positivity under an additional entropy condition. Theorem~B identifies the structure of this exact step. For fully implicit interactions, it adapts the comparison argument of \cite[Sec.~4]{LM26} by establishing monotonicity of the self-consistent chemical potential $\Phi_1$. This yields oscillation-dependent density bounds. The metric--dissipation identity follows by combining the discrete Euler equation with the $\HK$ potential identity of \cite{LMS23}, using the manifold Monge formula of \cite{GGV25} for the torus case. Appendix~\ref{sec:full-density-estimates} details the modified comparison argument.

Second, we determine when the pair formulation retains this variational target. The Monge representation of \cite{LMS23,GGV25} gives exact equivalence before parametric restriction. A neural restriction may introduce both a density error and a cone-action overestimate. Under an additional regular-pair approximation hypothesis, Theorem~C transfers pair approximation to convergence of the endpoint density and objective values, and shows that the metric overestimate vanishes.

Third, we assess an individual computed pair. Comparison with the identity gives the energy-descent criterion of Proposition~\ref{prop:one-step-certificate}, but descent alone does not quantify accuracy. Under positive-semidefinite interactions and global dual feasibility, Theorem~D combines a primal upper bound with a dual lower bound for the JKO minimum, using $\HK$ duality \cite{CPSV2018,LMS18} and convexity. Their gap controls objective suboptimality and, for Boltzmann entropy, the $L^1$ density error, without a parametric approximation hypothesis.

\begin{introresult}[Theorem~\ref{thm:main-fully-implicit-HK-JKO-step}]\label{intro:existence}
Let $U$ be convex, lower semicontinuous, finite at $0$, and superlinear; let $V$ be continuous; let $W$ be continuous and symmetric; and assume that the scalar function $\Gamma_{\tau,\bar m}$ in \eqref{eq:scalar-full-coercivity} is coercive.  Then the exact step exists and every minimizer satisfies a quantitative mass bound.  If $U'(0+)=-\infty$, its density is strictly positive and has full support.
\end{introresult}

\begin{introresult}[Theorem~\ref{thm:full-step-regularity}]\label{intro:regularity}
Under the standing entropy assumptions and two-sided bounds on the source density, the self-consistent chemical potential is nonincreasing along an optimal Monge map.  The new density has explicit two-sided bounds, satisfies the exact discrete Euler--Lagrange equation, and obeys $\HK^2(\mu_1,\bar\mu)=\tau^2I_g(\rho_1)$.  If, in addition, $V$ and $W$ are Lipschitz, then the new density is Lipschitz.  The result covers bounded convex domains and flat tori, including $W=0$.
\end{introresult}

\begin{introresult}[Theorem~\ref{thm:surrogate-one-step-consistency}]\label{intro:representability}
If a regular optimal Monge-growth pair is approximable in the chosen admissible class, there are parametric pairs whose endpoint densities, JKO objective values, and pair-objective values converge to those of the exact step. Their cone actions converge to the squared $\HK$ distance. At fixed $\tau$, this conclusion extends to any fixed number of steps under the corresponding approximation hypothesis at each step; no rate uniform in the number of steps is proved.
\end{introresult}

\begin{introresult}[Theorem~\ref{thm:a-posteriori}]\label{intro:a-posteriori}
Under the exact-step regularity assumptions and positive semidefinite interaction, the gap between a globally feasible dual lower bound and the Monge-growth primal upper bound controls both the JKO objective gap and the internal-energy convexity remainder relative to the unique exact step.  For Boltzmann entropy, it controls the squared $L^1$ error by $2(\nu(\Omega)+\mu_1(\Omega))\mathcal R$, where $\nu$ is the computed measure and $\mu_1$ the exact step; \eqref{eq:full-mass-sublevel-bound} bounds the latter mass.  No approximation assumption on the parametric class is used, but global dual feasibility must be verified.
\end{introresult}

\subsection{Relation to previous work}
\paragraph{Hellinger--Kantorovich geometry and minimizing movements.}
The $\HK$ distance admits complementary dynamic, cone, entropy--transport, and dual formulations \cite{KMV16,CPSV2018,CPSV18,LMS16,LMS18}.  Alternating Wasserstein--Fisher--Rao JKO steps were proved to converge under suitable compactness assumptions in \cite{GM17}; direct $\HK$ minimizing movements for scalar reaction--diffusion equations were constructed in \cite{Fleissner21}; and $\HK$ geometry was used in a Hele--Shaw-type tumor-growth limit in \cite{DiMarinoChizat20}.  More recent work develops fine geodesic and Monge properties, regularity of unbalanced transport, and an EVI theory for geodesically semiconvex energies \cite{LMS23,GGV25,LM26}.  Our exact-step analysis concerns a fully implicit proximal problem in which the interaction field depends on the candidate measure.  As explained above, the comparison scheme is rooted in \cite[Sec.~4]{LM26}, but the self-consistent chemical potential replaces the density as the monotone quantity and necessitates the modified oscillation-sensitive argument of Appendix~\ref{sec:full-density-estimates}.

\paragraph{Numerical optimal transport and neural JKO solvers.}
Entropic scaling made discrete unbalanced transport, barycenters, and unbalanced gradient-flow steps accessible through Sinkhorn-type iterations \cite{ChizatScaling18}, while Eulerian Benamou--Brenier discretizations with operator splitting provide structure-preserving Wasserstein JKO solvers in low dimensions \cite{CarrilloCraig22}.  Minimax estimation of transport-growth pairs in Gaussian--Hellinger unbalanced transport is studied
in \cite{PII26}.
Neural methods instead seek mesh-free parametrizations of Wasserstein or related minimizing movements: convex transport potentials are used in \cite{Mokrov21}, variational representations of density terms in \cite{Fan22}, direct state parametrization in \cite{PKSH23}, blockwise invertible flows in \cite{Xu23}, and Lagrangian neural ODEs with density updates in \cite{Lee24}.  These are primarily balanced or mass-preserving precedents.  Neural unbalanced transport methods combine learned maps with mass reweighting between prescribed distributions~\cite{YangUhler19,Lubeck22,Eyring24}.
TIGON learns velocity and growth fields from temporal snapshots using a dynamic WFR formulation~\cite{TIGON24}.
For a prescribed free energy, we use a primal cone action for fully implicit $\HK$-JKO steps, obtaining energy-descent criteria and, together with feasible dual lower bounds, a posteriori error estimates.

\paragraph{Fixed-reference, dual, and resampling-based formulations.}
A semi-dual approach to Wasserstein JKO steps uses a fixed reference measure and an unbalanced-transport reformulation \cite{CCK24}. A broader dual Lagrangian framework for balanced and unbalanced path problems appears in the literature \cite{Neklyudov24}. Local continuous normalizing flows have also been combined with importance correction and resampling \cite{Hertrich25}.  In our restarted formulation, each pair acts on the current measure and represents only one time step.  The trade-off is the loss of direct density evaluation from a fixed base law. We address this with change-of-variables formulas for density propagation and growth-weighted particle updates.

\subsection{Scope and organization}\label{sec:scope}
We distinguish the formal flow \eqref{pde-intro}, the exact iterates $\mu_{k+1}^\tau=\mathcal S_\tau(\mu_k^\tau)$, and the idealized parametric iterates with exact density propagation
\begin{equation}\label{eq:intro-incremental-neural-state}
\widehat\mu_{k+1}^\tau
=(T_{\theta,k+1})_\#\bigl(q_{\theta,k+1}^2\widehat\mu_k^\tau\bigr)
=u_{\theta,k+1}^\tau dx.
\end{equation}
Recursive numerical runs additionally apply a mass-corrected density fit after each pair update (Algorithm~\ref{alg:composite-HK-JKO}).

The analysis is time-discrete and covers the exact $\HK$-JKO step, its Monge-growth pair formulation, and guarantees for admissible computed pairs. We do not prove convergence of the stochastic optimizer or a $\tau\downarrow0$ limit for the trained scheme. Remarks~\ref{rem:one-step-neural-scope} and~\ref{rem:ap-scope} detail the assumptions and limitations of the approximation and a posteriori results.

Section~\ref{sec:setting-main-results} develops the exact-step theory, and Section~\ref{sec:restarted-formulation} presents the Monge-growth formulation. Section~\ref{sec:neural-realization-algorithm} describes the neural algorithm, and Section~\ref{sec:analytical-results} gives the approximation and a posteriori results. Section~\ref{sec:experiments} presents numerical experiments.\footnote{Code for the numerical experiments is available at \url{https://github.com/tyty4/hkflow}.} Proofs are collected in Appendices~\ref{sec:proof-first-theorem}--\ref{sec:proof-a-posteriori}, and implementation details in Appendix~\ref{sec:design-notes}.

\section{\texorpdfstring{{Setting and the exact $\HK$-JKO step}}{Setting and the exact HK-JKO step}}\label{sec:setting-main-results}\label{sec:preliminaries}\label{sec:exact-step}

\subsection{\texorpdfstring{{The minimal setting}}{The minimal setting}}\label{sec:minimal-setting}\label{sec:neural-proximal-solver}
Throughout, $\Omega$ denotes either a compact convex subset of $\mathbb R^d$
with nonempty interior or the flat torus $\mathbb T^d$.
In the Euclidean case, we write $D=\operatorname{int}\Omega$ when the
open domain must be distinguished from its closure.
We denote by $\mathcal M_+(\Omega)$ the space of finite nonnegative Borel
measures on $\Omega$.  The symbol $dx$ denotes Lebesgue measure in the Euclidean case and the Riemannian volume measure on $\mathbb T^d$; ${\sf d}_\Omega$ denotes the corresponding Euclidean or flat-geodesic distance.  For a real-valued function $f$, we write $f_-:=\max\{-f,0\}$.

We reserve $W(x,y)$ for the interaction kernel.  For $\mu\in\mathcal M_+(\Omega)$, define the associated interaction potential
\[
W[\mu](x):=\int_\Omega W(x,y)\,d\mu(y),
\]
and, for a density $\rho$, write $W[\rho]:=W[\rho\,dx]$.
The free energy is
\begin{equation*}
\mathcal F(\mu)=
\begin{cases}
\displaystyle\int_\Omega U(\rho)\,dx+\int_\Omega V\rho\,dx
+\frac12\iint_{\Omega\times\Omega}W(x,y)\rho(x)\rho(y)\,dx\,dy,
&\mu=\rho\,dx,\\[1mm]
+\infty,&\text{otherwise}.
\end{cases}
\end{equation*}

For later use, denote the scalar interaction-energy functional separately by
\begin{equation*}
\mathcal W(\mu):=\frac12\iint_{\Omega\times\Omega}W(x,y)\,d\mu(x)d\mu(y).
\end{equation*}
Thus $W$ denotes the two-point kernel, $W[\mu]$ its induced one-point potential, and $\mathcal W(\mu)$ the scalar interaction energy.

\begin{assumption}[Interaction kernel]\label{ass:interaction}
We use the following regularity levels:
\[
\textbf{\emph{(W0)}}\quad
W\in C_{\mathrm{sym}}(\Omega\times\Omega),
\qquad
\textbf{\emph{(W1)}}\quad
W\in W^{1,\infty}_{\mathrm{sym}}(\Omega\times\Omega),
\]
where
\(
C_{\mathrm{sym}}(\Omega\times\Omega)
:=
\{F\in C(\Omega\times\Omega):F(x,y)=F(y,x)\},
\)
and similarly for $W^{1,\infty}_{\mathrm{sym}}(\Omega\times\Omega)$.
\end{assumption}

{\noindent The internal energy is subject to the following standing assumption.}

\begin{assumption}[Standing assumptions on the internal energy]\label{ass:standing-entropy}
Whenever the pointwise regularity theorem is invoked,
\[
U\in C([0,\infty))\cap C^2((0,\infty)),\qquad U''>0,
\qquad U'(0+)=-\infty,\qquad U'(+\infty)=+\infty.
\]
\end{assumption}

Thus $U':(0,\infty)\to\mathbb R$ is a strictly increasing bijection.  Assumptions on $V,W$ and coercivity of the full JKO objective are stated in each theorem.

{
The principal example is the Boltzmann entropy $U(r)=r\log r-r$; a negative contribution from the quadratic interaction energy requires the separate coercivity hypothesis in Theorem~\ref{thm:main-fully-implicit-HK-JKO-step}.
\par}

For a positive density $\rho$, set
\begin{equation*}
P[\rho\,dx]=V+W[\rho],
\qquad g_\rho=U'(\rho)+P[\rho\,dx].
\end{equation*}
On a bounded domain the formal no-flux condition, with $\sf n$ the outward unit normal, is
\begin{equation*}
\rho\nabla g_\rho\cdot{\sf n}=0\qquad\mathcal H^{d-1}\text{-a.e. on }\partial\Omega.
\end{equation*}

\subsection{\texorpdfstring{$\HK$ formulations and normalization}{HK formulations and normalization}}\label{sec:HK-formulations}
{\paragraph{The Hellinger--Kantorovich metric.}
Throughout we use the normalization $\HK=\HK_{1,4}$, characterized by the dynamic formulation \cite{CPSV2018,LMS16}}

\begin{equation*}
\HK^2(\mu_0,\mu_1)=\inf\int_0^1\int_\Omega(|v_t|^2+4|\alpha_t|^2)\,d\mu_t\,dt,
\end{equation*}
where the infimum is taken over narrowly continuous curves
$(\mu_t)_{t\in[0,1]}$ in $\mathcal M_+(\Omega)$ and Borel
velocity and reaction fields $v_t$ and $\alpha_t$ with finite
action, satisfying
\begin{equation*}
\partial_t\mu_t+\operatorname{div}(\mu_t v_t)
=4\alpha_t\mu_t,
\qquad
\mu|_{t=0}=\mu_0,\quad \mu|_{t=1}=\mu_1.
\end{equation*}
The continuity equation is understood weakly, with no boundary
flux in the Euclidean case and periodicity on the torus.
This dynamic formulation corresponds to the choice $a=1$, $b=4$ in \eqref{pde-intro}.  The equivalent entropy--transport, dual, and one-potential formulations, together with the elementary bounds used in the proofs, are recorded below.

Set $F(s)=s\log s-s+1$, $F^*(a)=e^a-1$, and
\begin{equation*}
c_{\HK}(x,y)=-\log\!\left(\cos^2({\sf d}_\Omega(x,y)\wedge\tfrac\pi2)\right).
\end{equation*}
The entropy--transport formulation is \cite{CPSV2018,LMS18}
\begin{equation}\label{eq:HK-entropy-transport}
\HK^2(\mu_0,\mu_1)=\inf_{\gamma}
\left\{\int F\!\left(\frac{d\gamma_0}{d\mu_0}\right)d\mu_0
+\int F\!\left(\frac{d\gamma_1}{d\mu_1}\right)d\mu_1
+\int c_{\HK}\,d\gamma\right\},
\end{equation}
where $\gamma_i\ll\mu_i$.  Its dual is
\begin{equation}\label{eq:HK-dual}
\HK^2(\mu,\nu)=
\sup_{z_0+z_1\le c_{\HK}}
\left[-\int F^*(-z_0)\,d\mu-\int F^*(-z_1)\,d\nu\right],
\end{equation}
and the one-potential form is
\begin{equation}\label{eq:HK-onepot}
\HK^2(\mu,\nu)=\sup_v
\left[-\int F^*(-v^c)\,d\mu-\int F^*(-v)\,d\nu\right],
\qquad v^c(x)=\inf_y\{c_{\HK}(x,y)-v(y)\}.
\end{equation}
The formulation actually used in this paper is the equivalent \emph{cone}, or
semi-coupling, form \cite{LMS18,CPSV2018}: with
\(
\mathfrak c(x,y,a,b):=a^2+b^2-2ab\cos({\sf d}_\Omega(x,y)\wedge\tfrac\pi2),
\)
\begin{equation}\label{eq:HK-cone}
\HK^2(\mu_0,\mu_1)
=\min\left\{
\int_{\Omega\times\Omega}\mathfrak c\bigl(x,y,r_0(x,y),r_1(x,y)\bigr)\,d\gamma
\ :\
\pi^0_\#(r_0^2\gamma)=\mu_0,\quad \pi^1_\#(r_1^2\gamma)=\mu_1
\right\},
\end{equation}
the minimum being taken over $\gamma\in\mathcal M_+(\Omega\times\Omega)$ and Borel
$r_0,r_1\ge0$ with $r_i\in L^2(\gamma)$.  Only one consequence of
\eqref{eq:HK-cone} is used below: \emph{every} admissible triple
$(\gamma,r_0,r_1)$ yields an upper bound for $\HK^2$.  In particular,
$\gamma=(\mathrm{id},T)_\#\mu$ with $(r_0,r_1)=(1,q)$ gives
\eqref{eq:incremental-HK-upper-bound}, $\gamma=(T,T')_\#\mu$ with
$(r_0,r_1)=(q,q')$ gives the common-base bound used in
Lemma~\ref{lem:pair-to-HK-stability}, and the diagonal choice
$\gamma=(\mathrm{id},\mathrm{id})_\#(\rho_0\,dx)$ with
$(r_0,r_1)=(1,\sqrt{\rho_1/\rho_0})$ gives the pure-reaction competitor.  The latter reads
%The pure-reaction competitor gives
\begin{equation}\label{eq:HK2-FR2}
\HK^2(\rho_0dx,\rho_1dx)\le\int(\sqrt{\rho_0}-\sqrt{\rho_1})^2dx
\le\|\rho_0-\rho_1\|_{L^1},
\end{equation}
and
\begin{equation}\label{eq:HK-mass-difference}
\HK(\mu,\nu)\ge|\sqrt{\mu(\Omega)}-\sqrt{\nu(\Omega)}|,
\qquad \HK^2(\mu,0)=\mu(\Omega).
\end{equation}
For $\tau>0$, define
\begin{equation*}
\mathcal J_\tau(\sigma\mid\mu):=\mathcal F(\sigma)+\frac1{2\tau}\HK^2(\sigma,\mu),
\end{equation*}
and, whenever a minimizer exists, choose
\begin{equation*}
\mathcal S_\tau(\mu)\in\operatorname*{argmin}_{\sigma\in\mathcal M_+(\Omega)}\mathcal J_\tau(\sigma\mid\mu).
\end{equation*}
The exact recursion generated by this operator is
\begin{equation*}
\mu_{k+1}^\tau=\mathcal S_\tau(\mu_k^\tau),
\qquad k=0,1,\dots,
\end{equation*}
and is kept notationally distinct from the neural recursion
\eqref{eq:intro-incremental-neural-state} throughout.  In both recursions the step is
\emph{fully implicit} in the interaction: a candidate $r\,dx$ enters through
\begin{equation*}
g_r=U'(r)+V+W[r],
\end{equation*}
and not through a field frozen at the preceding state.

\subsection{\texorpdfstring{{Theorem~A: existence, mass control, positivity}}{Theorem A: existence, mass control, positivity}}\label{sec:exact-existence}
For $\bar m=\bar\mu(\Omega)$, $\underline V:=\min_{\Omega}V$, and $\underline W:=\min_{(x,y)\in\Omega\times\Omega}W(x,y)$, we define
\begin{equation}\label{eq:scalar-full-coercivity}
\Gamma_{\tau,\bar m}(m)=|\Omega|U\!\left(\frac m{|\Omega|}\right)+\underline V m+\frac{\underline W}{2}m^2
+\frac1{2\tau}(\sqrt m-\sqrt{\bar m})^2.
\end{equation}

\begin{theorem}[Existence, mass control, and positivity]\label{thm:main-fully-implicit-HK-JKO-step}
Let $\Omega$ be either a compact convex subset of $\mathbb R^d$ with nonempty interior or the flat torus $\mathbb T^d$.
Assume that $U:[0,\infty)\to\mathbb R$ is convex, lower semicontinuous, finite at $0$, and superlinear. Let $V\in C(\Omega)$ and suppose that \textbf{\emph{(W0)}} holds. Assume also that $\Gamma_{\tau,\bar m}(m)\to+\infty$ as $m\to\infty$.
Then $\mathcal J_\tau(\cdot\mid\bar\mu)$ has a minimizer $\mu_1=\rho_1dx$.  Every minimizer satisfies
\begin{equation}\label{eq:full-mass-sublevel-bound}
\mu_1(\Omega)\le M_{\tau,\bar\mu}:=
\sup\left\{m\ge0:\Gamma_{\tau,\bar m}(m)\le|\Omega|U(0)+\frac{\bar m}{2\tau}\right\}<\infty.
\end{equation}
If $U'(0+)=-\infty$, then $\rho_1>0$ almost everywhere and $\operatorname{supp}\mu_1=\Omega$.
Here $U'(0+)$ denotes the right derivative of the convex function $U$ at $0$, which exists in $[-\infty,+\infty)$; no differentiability of $U$ is assumed in this theorem.
\end{theorem}
\noindent
\textit{Idea of the proof (Appendix \ref{sec:proof-first-theorem}).}
The proof uses the direct method, with a scalar coercivity estimate to control the total mass.  Testing against
$\sigma=0$ bounds the infimum by $|\Omega|U(0)+\bar m/(2\tau)$, because
$\HK^2(0,\bar\mu)=\bar\mu(\Omega)$; this is where the unbalanced geometry helps, since the
zero measure is an admissible competitor at finite cost.  Jensen's inequality, the elementary
bounds $V\ge\underline V$ and $W\ge\underline W$, and the mass estimate
\eqref{eq:HK-mass-difference} then bound $\mathcal J_\tau(\sigma\mid\bar\mu)$ from below by
the scalar function $\Gamma_{\tau,\bar m}$ of the total mass \emph{alone}.  This reduction
is the only place where the negative contribution of the quadratic interaction energy must be controlled, and it is why
the coercivity hypothesis is stated as a condition on one scalar function rather than as a
sign condition on $W$.

Superlinearity of $U$ upgrades the resulting mass bound to uniform
integrability through the de la Vall\'ee-Poussin criterion, so no singular part can appear
in the limit, while narrow convergence of $\sigma_n\otimes\sigma_n$ on the compact product
space handles the interaction term.

Strict positivity is obtained by adding
$\varepsilon\mathbbm1_B\,dx$ on a putative vacuum set $B$: the extra proximal cost is at
most $\varepsilon|B|/(2\tau)$ because $F(0)=1$, whereas the internal-energy increment is
$|B|\bigl(U(\varepsilon)-U(0)\bigr)$, which dominates every $O(\varepsilon)$ term when
$U'(0+)=-\infty$.

\medskip

\begin{corollary}[Explicit mass bound]\label{cor:explicit-mass-bound}
If the interaction energy is nonnegative on nonnegative measures and $U(r)\ge ar-C_a$ with $a>\|V_-\|_\infty$, then
\begin{equation}\label{eq:explicit-full-mass-bound}
\mu_1(\Omega)\le\frac{C_a|\Omega|+|\Omega|U(0)+\bar m/(2\tau)}{a-\|V_-\|_\infty}.
\end{equation}
\end{corollary}

\subsection{\texorpdfstring{Monge representation on domains and tori}{Monge representation on domains and tori}}\label{sec:monge-input}
The next proposition supplies the optimal Monge-growth pairs used in both Theorem~B and the exact pair reformulation in Section~\ref{sec:incremental-monge-growth}. Its proof depends on whether the domain is Euclidean or a flat torus. The subsequent arguments use the representation \eqref{eq:injective-growth-representation}, positivity of $q$, and essential injectivity of $T$.
\begin{proposition}[Monge representation and essential injectivity]\label{prop:HK-Monge-representation}
Let $\Omega$ be either a compact convex subset of $\mathbb R^d$ with nonempty interior or the flat torus $\mathbb T^d$, and let $\mu_i=\rho_i dx$ with $\rho_i>0$ almost everywhere.  There is an optimal Monge-growth pair $(T,q)$ from $\mu_0$ to $\mu_1$ satisfying
\begin{equation}\label{eq:injective-growth-representation}
T_\#(q^2\mu_0)=\mu_1,
\qquad
\HK^2(\mu_0,\mu_1)=\int_\Omega(1+q^2-2q\cos({\sf d}_\Omega(x,T(x))\wedge\tfrac\pi2))\,d\mu_0.
\end{equation}
Moreover, $q>0$ $\mu_0$-almost everywhere and $T$ may be chosen essentially injective. The same statement holds with $\mu_0$ and $\mu_1$ interchanged.
\end{proposition}
\noindent
\textit{Idea of the proof (Appendix~\ref{sec:proof-monge-representation}).}
The Euclidean case is Theorem~3.3(4)--(5) and Corollary~3.5 of \cite{LMS23}.  Since both
densities are positive, their supports both equal $\Omega$, so each support lies in the
$\pi/2$-neighborhood of the other; this is precisely the strong-reduction condition of
\cite[Definition~2.6]{LMS23}.  The torus case follows by combining the Lipschitz-potential and
Monge results of \cite{GGV25}, with essential
injectivity obtained from uniqueness of the entropy--transport optimizer for the reversed
pair.  Appendix~\ref{sec:proof-monge-representation} also records the explicit form of
$(T,q)$ in terms of the logarithmic dual potential.

\paragraph{Orientation convention.}
The \emph{forward} pair defines the numerical update, $T_\#(q^2\mu_k)=\mu_{k+1}$, with cone action integrated against $\mu_k$. The \emph{reversed} pair maps $\mu_{k+1}$ back to $\mu_k$; its source-side dual potential appears in Theorem~\ref{thm:full-step-regularity}(iii), whose Euler equation holds at the new density. Proposition~\ref{prop:HK-Monge-representation} provides both orientations.

Accordingly, \eqref{eq:exact-step-dissipation-main} integrates against $\rho_{k+1}$, while $\mathcal A_{\mu_k}$ integrates against $\rho_k$. Their leading-order agreement in $\tau$ supplies the diagnostic reference values. The forward and reversed generating potentials also have opposite leading-order signs; see Remark~\ref{rem:tau-scaled-ansatz}.

\subsection[Theorem B: density bounds and discrete identities]{\texorpdfstring{{Theorem~B: density bounds, the discrete Euler equation, and the metric--dissipation identity}}{Theorem B: density bounds, the discrete Euler equation, and the metric--dissipation identity}}\label{sec:exact-regularity}
Let us define
\begin{equation*}
I_g(\rho)=\int_\Omega(|\nabla g_\rho|^2+4|g_\rho|^2)\rho\,dx.
\end{equation*}

\begin{theorem}[Density bounds and the exact discrete Euler equation]\label{thm:full-step-regularity}
Under the assumptions of Theorem~\ref{thm:main-fully-implicit-HK-JKO-step}, we further assume Assumption~\ref{ass:standing-entropy}.  Let $\bar\mu=\bar\rho\,dx$ with $\bar\rho\in L^\infty(\Omega)$ and $0<c_0\le\bar\rho\le C_0$ almost everywhere.  For a minimizer $\mu_1=\rho_1dx$, set $M_1=M_{\tau,\bar\mu}$ and
\begin{equation*}
B_1=\|V\|_\infty+M_1\|W\|_\infty,
\qquad \omega_1=\essosc_\Omega V+2M_1\|W\|_\infty,
\qquad c_*=(U')^{-1}(-B_1-1),
\end{equation*}
where $\essosc_\Omega f:=\esssup_\Omega f-\essinf_\Omega f$ denotes the essential oscillation.
Then:
\begin{enumerate}[(i)]
\item there exists an essentially injective optimal pair $(T,q)$ from $\bar\mu$ to $\mu_1$ such that
\begin{equation}\label{eq:chemical-potential-monotonicity-main}
\Phi_1(Tx)\le\Phi_1(x),\qquad \Phi_1:=U'(\rho_1)+V+W[\mu_1],
\end{equation}
for $\bar\mu$-almost every $x$;
\item the density satisfies
\begin{equation}\label{eq:canonical-two-sided-density-bound}
\begin{aligned}
0<\underline\rho_1&:=(U')^{-1}(U'(\min\{c_0,c_*\})-\omega_1)\le\rho_1,\\
\rho_1&\le(U')^{-1}(\max\{B_1,U'(C_0)\}+\omega_1)=:\overline\rho_1<\infty;
\end{aligned}
\end{equation}
\item for the reversed pair $(\mu_1,\bar\mu)$, there is a potential $\phi_1\in W^{1,\infty}(\Omega)$ such that $2\phi_1$ is the coefficient of $\mu_1$ in the transformed dual \eqref{eq:transformed-HK-dual}. With this normalization,
\begin{equation}\label{eq:full-Euler-Lagrange-main}
g_{\rho_1}:=U'(\rho_1)+V+W[\rho_1]=-\frac{\phi_1}{\tau}\qquad\text{a.e.},
\end{equation}
and
\begin{equation}\label{eq:exact-step-dissipation-main}
\HK^2(\mu_1,\bar\mu)=\tau^2I_g(\rho_1),
\qquad
\mathcal F(\mu_1)+\frac\tau2I_g(\rho_1)\le\mathcal F(\bar\mu).
\end{equation}
\end{enumerate}
Here $g_{\rho_1}=-\phi_1/\tau$ is Lipschitz by part~(iii), and $I_g(\rho_1)$ uses its weak gradient.
If, in addition, $V\in W^{1,\infty}(\Omega)$ and \textbf{\emph{(W1)}} holds, then $\rho_1\in W^{1,\infty}(\Omega)$.
\end{theorem}
\begin{remark}
The local case follows by setting $W=0$, $B_1=\|V\|_\infty$, and $\omega_1=\essosc_\Omega V$.    
\end{remark}

\noindent
\textit{Idea of the proof (Appendix \ref{sec:full-interaction-analysis}).}
Two independent mechanisms are combined.

The first is a comparison principle for the step.  On a set $A$, replacing the optimal
Monge-growth pair by the identity map with the constant growth factor
$\sqrt{\mu_1(T(A))/\bar\mu(A)}$ does not increase the $\HK$ cost, by Cauchy--Schwarz
(Lemma~\ref{lem:cut-and-glue-full}).  Consequently an ordering violation of the
self-consistent chemical potential $\Phi_1$ along $T$ would allow one to move mass from a
region of high $\Phi_1$ to a region of low $\Phi_1$ without increasing the metric cost and with strictly
lower energy, contradicting minimality.  This gives $\Phi_1\circ T\le\Phi_1$
(Lemma~\ref{lem:monotonicity-full}). Using this monotonicity and the restriction identity
for an essentially injective Monge pair (Lemma~\ref{lem:Monge-restriction-identity}),
we compare suitable local mass variations on regions of high or low chemical potential.
Comparing their effects on the energy and the metric cost yields the two-sided density
bounds \eqref{eq:canonical-two-sided-density-bound}.

The second is a dual computation, carried out for the \emph{reversed} pair from $\mu_1$ to
$\bar\mu$, since the Euler--Lagrange equation is an identity at the new density.  For the source component $\psi$ supplied by
Lemma~\ref{lem:transformed-potential-Lipschitz}, that pair satisfies
$1-\psi=q\cos r$ and $\nabla\psi=-2q\sin r\,e$, where $r={\sf d}_\Omega(x,Tx)$ and $e$ is the
initial unit tangent of the selected geodesic.  With $\phi=\psi/2$ this yields the pointwise
identity
\[
|\nabla\phi|^2+4|\phi|^2=1+q^2-2q\cos r,
\]
whose integral is exactly $\HK^2$ (Corollary~\ref{cor:HK-energy-identity-full}).
Taking the first variation of the JKO objective with respect to the endpoint density and using uniqueness of that source component
identifies $\phi_1=-\tau g_{\rho_1}$, which is the Euler--Lagrange equation
\eqref{eq:full-Euler-Lagrange-main}; substituting it back into the identity gives
$\HK^2(\mu_1,\bar\mu)=\tau^2I_g(\rho_1)$.  If $V$ and $W$ are additionally Lipschitz,
then $U'(\rho_1)=-\phi_1/\tau-V-W[\rho_1]$ is Lipschitz.  Since $(U')^{-1}$ is
Lipschitz on the interval supplied by part~(ii), the density is then Lipschitz as well.

The comparison argument uses the essentially injective optimal pair from Proposition~\ref{prop:HK-Monge-representation}, the only part of the proof that depends on the state space.

\medskip

\section{\texorpdfstring{{The Monge-growth formulation}}{The restarted Monge-growth formulation}}\label{sec:restarted-formulation}

\subsection{\texorpdfstring{{Exact reparametrization and pair stability}}{Exact reparametrization and pair stability}}\label{sec:incremental-monge-growth}
The first point is variational: before any finite-dimensional restriction, the restarted pair variables reproduce the exact JKO problem.

\begin{proposition}[{One-step} $\HK$-JKO reparametrization]\label{prop:incremental-HK-JKO-reparam}
Let $\mu$ have a positive density and assume that
$\mathcal J_\tau(\cdot\mid\mu)$ admits a positive-density minimizer.  Then
\begin{equation}\label{eq:incremental-HK-JKO-reparam}
\min_{\sigma\in\mathcal M_+(\Omega)}
\mathcal J_\tau(\sigma\mid\mu)
=
\inf_{T,\,q\ge0}
\mathcal I_\tau(T,q\mid\mu).
\end{equation}
If $\mu_+$ is a positive-density minimizer and $(T_*,q_*)$ is an optimal Monge-growth pair from $\mu$ to $\mu_+$, then
\begin{equation}\label{eq:incremental-optimal-pair}
(T_*)_\#(q_*^2\mu)=\mu_+,
\qquad
\mathcal A_\mu(T_*,q_*)=\HK^2(\mu,\mu_+),
\end{equation}
and $(T_*,q_*)$ attains the right-hand side of \eqref{eq:incremental-HK-JKO-reparam}.
\end{proposition}
\begin{proof}
For every pair $(T,q)$, \eqref{eq:incremental-HK-upper-bound} gives
\[
\mathcal J_\tau\bigl(T_\#(q^2\mu)\mid\mu\bigr)
\le
\mathcal I_\tau(T,q\mid\mu).
\]
Since the left-hand side is bounded below by the minimum of the JKO problem, the infimum over $(T,q)$ is no smaller than that minimum.  Conversely, let $\mu_+$ be a positive-density minimizer.  Proposition~\ref{prop:HK-Monge-representation} supplies an optimal pair $(T_*,q_*)$ satisfying \eqref{eq:incremental-optimal-pair}.  Substitution gives
\[
\mathcal I_\tau(T_*,q_*\mid\mu)
=
\mathcal J_\tau(\mu_+\mid\mu),
\]
which proves the reverse inequality and attainment.
\end{proof}

Applying Proposition~\ref{prop:HK-Monge-representation} successively, every finite exact $\HK$-JKO sequence with positive densities admits an essentially injective restarted Monge-growth representation whose cone action equals the squared $\HK$ distance at each step.

For the approximation arguments below, the same cone representation also gives direct stability of the induced endpoint with respect to the pair.  We will often work with a current measure
\begin{equation*}
\mu=\rho\,dx,
\qquad
\rho\in C(\Omega),
\qquad
0<\underline\rho\le\rho\le\overline\rho.
\end{equation*}

\begin{lemma}[Stability of the pushforward in the pair]\label{lem:pair-to-HK-stability}
Let $\mu\in\mathcal M_+(\Omega)$ and let $(T,q)$, $(T',q')$ be admissible pairs.  Then
\begin{equation}\label{eq:pair-to-HK-stability}
\HK^2\bigl(T_\#(q^2\mu),\,T'_\#(q'^2\mu)\bigr)
\le
\int_\Omega\Bigl[(q-q')^2+q\,q'\,{\sf d}_\Omega(T,T')^2\Bigr]d\mu .
\end{equation}
In particular, if $\max\{q,q'\}\le A$ on $\Omega$, then
\begin{equation}\label{eq:pair-to-HK-stability-simple}
\HK\bigl(T_\#(q^2\mu),\,T'_\#(q'^2\mu)\bigr)
\le
\mu(\Omega)^{1/2}\Bigl(\|q-q'\|_{L^\infty}+A\,\|{\sf d}_\Omega(T,T')\|_{L^\infty}\Bigr).
\end{equation}
\end{lemma}
\begin{proof}
{Set $S=(T,T')$ and $\lambda=S_\#\mu$, and define
\[
a=\frac{dS_\#(q^2\mu)}{d\lambda},\qquad
b=\frac{dS_\#(q'^2\mu)}{d\lambda},\qquad
c=\frac{dS_\#(qq'\mu)}{d\lambda}.
\]
Conditional Cauchy--Schwarz gives $c\le\sqrt{ab}$ $\lambda$-a.e.  Thus $\lambda$ with amplitudes $(\sqrt a,\sqrt b)$ is admissible in \eqref{eq:HK-cone}, even when $S$ is not injective. Since the truncated cosine is nonnegative, its cost is at most the following common-base integral:}
\[
\HK^2\bigl(T_\#(q^2\mu),T'_\#(q'^2\mu)\bigr)
\le
\int_\Omega\Bigl[q^2+q'^2-2q\,q'\cos\bigl({\sf d}_\Omega(T,T')\wedge\tfrac\pi2\bigr)\Bigr]d\mu,
\]
of which \eqref{eq:incremental-HK-upper-bound} is the case $(T',q')=(\mathrm{id},1)$.  Writing $s:={\sf d}_\Omega(T,T')\wedge\frac\pi2$ and $q^2+q'^2-2q\,q'\cos s=(q-q')^2+2q\,q'(1-\cos s)$, the bound $1-\cos s\le s^2/2\le{\sf d}_\Omega(T,T')^2/2$ gives \eqref{eq:pair-to-HK-stability}, and $\sqrt{x+y}\le\sqrt x+\sqrt y$ gives \eqref{eq:pair-to-HK-stability-simple}.
\end{proof}

Together with Proposition~\ref{prop:incremental-HK-JKO-reparam}, estimate \eqref{eq:pair-to-HK-stability-simple} converts pair approximation into a one-step state error whenever $\max\{q_\theta,q_*\}\le A$:
\begin{equation}\label{eq:eta-to-tau-squared}
\HK\bigl(\widehat\mu_{k+1},\mathcal S_\tau\widehat\mu_k\bigr)
\le
\widehat\mu_k(\Omega)^{1/2}\Bigl(\|q_\theta-q_*\|_{L^\infty}+A\|{\sf d}_\Omega(T_\theta,T_*)\|_{L^\infty}\Bigr).
\end{equation}

\subsection{\texorpdfstring{{Primal action versus restricted duality}}{Primal action versus restricted duality}}\label{sec:primal-vs-dual}
Even after parametric restriction, the cone action gives an upper bound on the squared $\HK$ distance, while a feasible dual potential gives a lower bound. The upper bound yields the following sufficient condition for energy descent.

Writing $r(x):={\sf d}_\Omega(x,T(x))\wedge\frac\pi2$, the cone integrand splits as
\begin{equation}\label{eq:cone-action-splitting}
1+q^2-2q\cos r=(1-q)^2+2q\,(1-\cos r).
\end{equation}
Accordingly, $\mathcal A_\mu=\mathcal A^{\rm re}_\mu+\mathcal A^{\rm tr}_\mu$, where
\[
\mathcal A^{\rm re}_\mu(T,q):=\int_\Omega(1-q)^2d\mu,
\qquad
\mathcal A^{\rm tr}_\mu(T,q):=\int_\Omega2q\,(1-\cos r)\,d\mu.
\]
These are the reaction and transport components of the action. Their ratios to the total action, when nonzero, are called the reaction and transport shares and are reported in Section~\ref{sec:experiments}.  The identity pair reproduces the source with zero action:
\begin{equation}\label{eq:identity-pair-zero-action}
(\mathrm{id})_\#(1^2\mu)=\mu,
\qquad
\mathcal A_\mu(\mathrm{id},1)=0,
\qquad
\mathcal I_\tau(\mathrm{id},1\mid\mu)=\mathcal F(\mu).
\end{equation}

\paragraph{An acceptance certificate for a computed state.}
A useful consequence of the one-sided formulation is a certificate for an arbitrary admissible pair, independently of how that pair was obtained.  Combining \eqref{eq:intro-primal-majorization} with the identity competitor gives the following criterion.

\begin{proposition}[One-step descent criterion with slack]\label{prop:one-step-certificate}
Let $\mu\in\mathcal M_+(\Omega)$ satisfy $\mathcal F(\mu)<\infty$, let $(T,q)$ be admissible for $\mu$, let $\delta\ge0$, and set $\mu_+:=T_\#(q^2\mu)$.  If
\begin{equation}\label{eq:descent-guard}
\mathcal I_\tau(T,q\mid\mu)
\le
\mathcal I_\tau(\mathrm{id},1\mid\mu)-\delta
=\mathcal F(\mu)-\delta,
\end{equation}
then
\begin{equation}\label{eq:unconditional-one-step-dissipation}
\mathcal F(\mu_+)+\frac1{2\tau}\HK^2(\mu_+,\mu)
\le\mathcal F(\mu)-\delta.
\end{equation}
\end{proposition}
Indeed, \eqref{eq:incremental-HK-upper-bound} gives
\[
\mathcal F(\mu_+)+\frac1{2\tau}\HK^2(\mu_+,\mu)
\le
\mathcal F(\mu_+)+\frac1{2\tau}\mathcal A_\mu(T,q)
=\mathcal I_\tau(T,q\mid\mu),
\]
and \eqref{eq:descent-guard} completes the argument.  The case $\delta=0$ is the acceptance criterion in Algorithm~\ref{alg:composite-HK-JKO}; summing the inequalities over steps accumulates the verified slacks in the dissipation bound. A strictly positive slack excludes the identity step.  In numerical use the criterion is exact only when the relevant functional values have controlled quadrature error.  Moreover, energy decay alone does not exclude a nontrivial pure-reaction update, which is why the transport and reaction shares are monitored in Remark~\ref{rem:fisher-rao-degeneration} and Section~\ref{sec:experiments}.

\paragraph{Primal upper bounds and restricted dual lower bounds.}

Primal and dual representations are equivalent before restriction. The primal upper bound and the dual lower bound remain valid after restriction.  For an admissible pair $(T_\theta,q_\theta)$ and the induced
candidate $\nu_\theta:=(T_\theta)_\#(q_\theta^2\mu)$, the cone action gives
\begin{equation}\label{eq:primal-upper-certificate}
\mathcal F(\nu_\theta)+\frac1{2\tau}\HK^2(\mu,\nu_\theta)
\le
\mathcal F(\nu_\theta)+\frac1{2\tau}\mathcal A_\mu(T_\theta,q_\theta)
=\mathcal I_\tau(T_\theta,q_\theta\mid\mu),
\end{equation}
which together with \eqref{eq:identity-pair-zero-action} yields the descent estimate
\eqref{eq:descent-guard}--\eqref{eq:unconditional-one-step-dissipation} for every pair satisfying the criterion.  Independently of acceptance, every admissible pair supplies an \emph{upper} bound for the exact proximal cost, and the same variables
generate the next state.  Writing $D_{\HK}(\varphi;\mu,\nu)$ for the objective of
\eqref{eq:HK-dual}--\eqref{eq:HK-onepot} at a feasible potential $\varphi$, a maximization
restricted to a finite class $\mathcal D_\Theta\subset\mathcal D$ has the opposite error,
$\sup_{\varphi\in\mathcal D_\Theta}D_{\HK}\le\HK^2(\mu,\nu)$, so descent of an
objective built from such a value does not by itself guarantee descent of the exact JKO
functional.  The price of the primal choice is that a restricted class may strictly
overestimate $\HK^2$; this is a representation error, and at the unrestricted level the bound
is attained by an optimal pair.  Duality remains available as the lower bound in the a
posteriori estimate
\[
D_{\HK}(\varphi;\mu,\nu_\theta)\le\HK^2(\mu,\nu_\theta)\le\mathcal A_\mu(T_\theta,q_\theta),
\]
which can be used to assess the tightness of a learned Monge-growth representation.
The same one-sided structure
underlies the reaction-only indicator of Remark~\ref{rem:fisher-rao-degeneration}; the
held-out-sample or enriched-quadrature acceptance test in Algorithm~\ref{alg:composite-HK-JKO} remains empirical unless its error is
controlled rigorously.

\section{Neural realization and algorithm}\label{sec:neural-realization-algorithm}

\subsection{Pair objective and neural parametrization}\label{sec:incremental-neural-param}
Assume that $T$ is an orientation-preserving $C^1$ diffeomorphism of $\Omega$ and write
\[
J_T(x):=|\det DT(x)|.
\]
If $T_\#(q^2\rho\,dx)=u\,dx$, then the change-of-variables formula gives
\begin{equation}\label{eq:incremental-induced-density}
u(T(x))=\frac{q(x)^2\rho(x)}{J_T(x)}.
\end{equation}
Thus \eqref{eq:incremental-pair-objective-abstract} becomes
\begin{equation}\label{eq:incremental-current-objective}
\begin{aligned}
\mathcal I_\tau(T,q\mid\rho\,dx)
={}&\int_\Omega
U\!\left(\frac{q(x)^2\rho(x)}{J_T(x)}\right)J_T(x)\,dx
+\int_\Omega V(T(x))q(x)^2\rho(x)\,dx\\
&+\frac12\iint_{\Omega\times\Omega}
W(T(x),T(y))q(x)^2q(y)^2\rho(x)\rho(y)\,dx\,dy\\
&+\frac1{2\tau}\int_\Omega
\left[
1+q(x)^2-2q(x)\cos\!\left({\sf d}_\Omega(x,T(x))\wedge\frac\pi2\right)
\right]\rho(x)\,dx.
\end{aligned}
\end{equation}

We use one parameter symbol $\theta$ for the map and growth networks:
\begin{equation*}
T_\theta=T_{\theta_T},
\qquad
q_\theta={q_{\rm ref}}\exp(Q_{\theta_q}),
\qquad
\theta=(\theta_T,\theta_q).
\end{equation*}
Here $q_{\rm ref}$ is a fixed positive reference profile. The experiments use $q_{\rm ref}=1$ and optimize the growth network jointly with the map. An optional pure-reaction reference can instead be defined by freezing the interaction at the current density and solving the following equation pointwise by bisection:
\begin{equation}\label{eq:growth-reference}
{q\Bigl(1+2\tau\,[U'(q^2\rho)+V+W[\rho]]\Bigr)=1,
\qquad
q=q_{\rm ref}.}
\end{equation}
For $W=0$, this is the exact Euler equation for the pure-reaction subproblem. For $W\ne0$, the fully implicit equation contains $W[q^2\rho]$ and couples spatial points, so \eqref{eq:growth-reference} supplies only a frozen-field reference. Any such competitor is evaluated with the same fully implicit objective as the trained pair.

The choice $Q_{\theta_q}=0$ reproduces the reference profile. At the level of continuous functions, the fixed positive continuous factor $q_{\rm ref}$ can be absorbed into the log-growth function by adding $\log q_{\rm ref}$. Approximation by a network with a bounded output remains subject to Assumption~\ref{ass:incremental-neural-recovery}. Appendix~\ref{sec:design-notes-experiments} describes the bounded neural output and initialization used in the experiments.
Given a current density $\rho$, the notation $u_\theta$ denotes the density determined by
\begin{equation*}
(T_\theta)_\#(q_\theta^2\rho\,dx)=u_\theta\,dx.
\end{equation*}

The map must be admissible by construction rather than made invertible only through a soft penalty. {The following constructions have this property; experiment-specific variants are detailed in Appendix~\ref{sec:design-notes-experiments}.}

{\emph{A constrained residual map.} The experiments in Section~\ref{sec:experiments} compose two blocks of this form. On a box $\Omega=\prod_j[l_j,h_j]$ of widths $w_j$ we set}
\begin{equation}\label{eq:residual-map}
{T_{\theta_T}(x)=x+s(x)\odot\delta_{\theta_T}(x),
\qquad
s_j(x)=\sin\Bigl(\pi\frac{x_j-l_j}{w_j}\Bigr),}
\end{equation}
{where $\odot$ is the componentwise product and $\delta_{\theta_T}$ is a bounded network output with $|\delta_j|\le A$ and $|\delta|\le\pi/2-\varepsilon$.  Only the factor of coordinate $j$ multiplies component $j$ of the displacement.  This enforces the domain constraint: since $\sin(\pi t)/\pi\le\min\{t,1-t\}$, the bound $A\le w_j/\pi$ gives $|s_j\delta_j|\le\operatorname{dist}(x_j,\{l_j,h_j\})$, so $T_{\theta_T}(\Omega)\subset\Omega$ with each boundary face preserved: the normal coordinate is fixed, while tangential motion along the face is allowed.  On the torus $s\equiv1$ and the network takes a periodic embedding of the spatial coordinates as input.}

The residual network uses continuously differentiable activations; ReLU is excluded from the map class. For ReLU, the location at which the map derivative jumps can depend on the parameters. Differentiating the integrated internal energy then produces an interface term that pointwise automatic differentiation does not capture. The radial displacement bound also uses a smooth contraction. For $c=\pi/2-\varepsilon$ and $r_0=c/2$, a raw displacement $z$ of norm $r$ is unchanged for $r\leq r_0$ and otherwise becomes $f(r)z/r$, where
\[
f(r)=r_0+(c-r_0)\tanh\!\left(\frac{r-r_0}{c-r_0}\right).
\]
This cap is $C^2$ at $r=r_0$, has radial derivative and tangential scaling at most one, and preserves each component bound. Thus it preserves the contraction estimate below while avoiding a moving derivative discontinuity.

{Orientation and injectivity are obtained from a contraction estimate rather than from a Jacobian penalty.  Writing $L$ for a Lipschitz bound of the network in the normalized input variable, one has the pointwise bound}
\begin{equation*}
{\bigl\|D\bigl(s\odot\delta_{\theta_T}\bigr)(x)\bigr\|_2
\;\le\;
\frac{A\,(\pi+2L)}{w_{\min}}
\;=:\;\kappa_\theta ,}
\end{equation*}
{because the diagonal part contributes $\max_j|s_j'||\delta_j|\le\pi A/w_{\min}$ and the remaining part contributes $\|s\|_\infty\operatorname{Lip}(\delta_{\theta_T})\le 2AL/w_{\min}$; on the torus the same argument gives $\kappa_\theta\le 2\pi AL/w_{\min}$.  If $\kappa_\theta<1$, then $DT_{\theta_T}=I+D(s\odot\delta_{\theta_T})$ is invertible with $\det DT_{\theta_T}>0$ and $T_{\theta_T}$ is injective because its displacement is a strict contraction.

On a box, write $G=\operatorname{int}\Omega$. The inverse function theorem makes $T_{\theta_T}(G)$ open in $\mathbb R^d$ and hence contained in $G$. Boundary preservation gives $T_{\theta_T}(G)=T_{\theta_T}(\Omega)\cap G$, which is relatively closed in $G$ by compactness. Connectedness gives $T_{\theta_T}(G)=G$, and taking closures yields $T_{\theta_T}(\Omega)=\Omega$. On the torus, the image is likewise open by local invertibility and closed by compactness, hence equals the connected torus. Thus $T_{\theta_T}$ is an orientation-preserving diffeomorphism of $\Omega$ onto itself, as is any finite composition of such blocks, and \eqref{eq:incremental-induced-density} is valid.

The implementation bounds each residual block using network coefficients and the spatial input derivatives, with tighter interval bounds when needed. It scales the displacement to enforce contraction, the domain constraint, and the displacement bound. Validation checks these bounds again; a candidate that fails is rejected.  The sampled minimum of $J_{T_\theta}$ remains an additional diagnostic.}

{\emph{A neural ODE flow.}}  {O}ne convenient {alternative} is the time-one flow of a neural ordinary differential equation,
\begin{equation}\label{eq:neural-ode-map}
\frac{d}{dt}\Phi_{\theta,t}(x)
=v_{\theta_T}\bigl(t,\Phi_{\theta,t}(x)\bigr),
\qquad
\Phi_{\theta,0}(x)=x,
\qquad
T_{\theta_T}(x)=\Phi_{\theta,1}(x).
\end{equation}
On a bounded domain, $v_{\theta_T}(t,\cdot)$ is taken $C^1$ in space, globally Lipschitz, and satisfies $v_{\theta_T}\cdot{\sf n}=0$ $\mathcal H^{d-1}$-almost everywhere on $\partial\Omega$; the stronger condition $v_{\theta_T}=0$ on $\partial\Omega$ is a simple sufficient choice.  The flow preserves the boundary and maps $\Omega$ diffeomorphically onto itself.  On the torus, the vector field is periodic.  In either case the exact flow is orientation preserving and Liouville's formula gives
\begin{equation}\label{eq:neural-ode-log-jacobian}
\log J_{T_{\theta_T}}(x)
=
\int_0^1
\operatorname{div}v_{\theta_T}\bigl(t,\Phi_{\theta,t}(x)\bigr)\,dt,
\qquad J_{T_{\theta_T}}(x)>0,
\end{equation}
as in continuous normalizing flows \cite{Chen18,Grathwohl19}.  Other invertible architectures may be used if they map $\Omega$ onto itself and provide a controlled positive Jacobian.  The ODE construction includes the identity map through $v_{\theta_T}=0$. The experiments use composed residual maps with Jacobians evaluated from the network derivatives; their common architecture and density model are described in Appendix~\ref{sec:design-notes-experiments}. For either map class, the approximation result below requires Assumption~\ref{ass:incremental-neural-recovery} for the regular optimal pair under consideration.

\subsection{\texorpdfstring{{Sampling, mass propagation, and the algorithm}}{Sampling, mass propagation, and the algorithm}}
Let $\mathfrak m=\mu(\Omega)$ and $\bar\mu=\mu/\mathfrak m$.  All terms in \eqref{eq:incremental-current-objective} can be evaluated from samples of the current measure.  For example,
\begin{equation}\label{eq:incremental-internal-expectation}
\int_\Omega
U\!\left(\frac{q^2\rho}{J_T}\right)J_T\,dx
=
\mathfrak m\,
\mathbb E_{X\sim\bar\mu}
\left[
\frac{J_T(X)}{\rho(X)}
U\!\left(\frac{q(X)^2\rho(X)}{J_T(X)}\right)
\right],
\end{equation}
whereas the interaction is
\begin{equation}\label{eq:incremental-interaction-expectation}
\frac{\mathfrak m^2}{2}
\mathbb E_{X,Y\sim\bar\mu}
\left[W(T(X),T(Y))q(X)^2q(Y)^2\right].
\end{equation}
For a single batch of independent samples, the interaction is estimated by the off-diagonal U-statistic
\begin{equation}\label{eq:incremental-interaction-U-statistic}
\frac{\mathfrak m^2}{2B(B-1)}
\sum_{i\ne j}
W(T(X_i),T(X_j))q(X_i)^2q(X_j)^2.
\end{equation}
The interaction estimator is unbiased when the $X_i$ are independent draws from $\bar\mu$, since each term with $i\ne j$ has the product sampling law. The proximal term is $\mathfrak m$ times the empirical average of the integrand in \eqref{eq:incremental-HK-action}.  If a uniformly selected subset of size $s\ge2$ is used, its mass weights are $w_i=\mathfrak m q(X_i)^2/s$ and the estimator is
\[
\frac{s}{2(s-1)}\sum_{i\ne j}w_iw_jW(T(X_i),T(X_j)).
\]
The correction $s/(s-1)$ is fixed by the sample count and does not depend on the learned growth weights.  When sampling from the current normalized measure, it draws independently, with replacement, from a density-weighted finite proposal pool.  Conditional on that pool the off-diagonal estimator is unbiased for the corresponding discrete sampling law; approximation of the current continuum measure by the pool remains a separate sampling error.

The interaction benchmark instead uses deterministic periodic quadrature.  If $h_i$ are spatial quadrature weights and $\omega_i=h_i\rho(X_i)$, its pair energy is
\begin{equation}\label{eq:incremental-interaction-quadrature}
\frac12\sum_{i,j}\omega_i\omega_j q(X_i)^2q(X_j)^2
W(T(X_i),T(X_j)).
\end{equation}
This full weighted double sum includes the diagonal and has no U-statistic correction.  A shifted grid with more quadrature nodes checks the trained pair; it is distinct from the training grid but is not a statistically independent random sample.  Quadrature refinement is an empirical discretization check, not a rigorous continuum error bound. The neural parametrization is the same for stochastic and deterministic integration.

The normalized sampling law must also be propagated with the growth weights.  Let $\pi=\mu/\mathfrak m$ and let $\mu_+=T_\#(q^2\mu)$.  Then
\begin{equation}\label{eq:normalized-monge-growth-update}
Z:=\int_\Omega q^2\,d\pi,
\qquad
\mathfrak m_+=\mathfrak m Z,
\qquad
\pi_+:=\frac{\mu_+}{\mathfrak m_+}
=T_\#\!\left(\frac{q^2}{Z}\,\pi\right).
\end{equation}
Consequently, if $X_i\sim\pi$ and $Y_i=T(X_i)$, the points $Y_i$ represent $\pi_+$ only after attaching normalized weights
\(
 w_i=q(X_i)^2/\sum_j q(X_j)^2
\)
or after resampling according to these weights.  The empirical measure of the unweighted points $T(X_i)$ represents the next normalized measure only when $q$ is constant.  At the transformed sample locations the density update is
\begin{equation}\label{eq:incremental-log-density-update}
\log u_+(T(x))
=
\log\rho(x)+2\log q(x)-\log J_T(x).
\end{equation}
The formulas in \eqref{eq:normalized-monge-growth-update} are exact; replacing $Z$ and $\pi_+$ by empirical estimates introduces a separate particle/quadrature error.  Importance resampling plays an analogous role in iterative neural JKO sampling \cite{Hertrich25}.

\begin{algorithm}[p]

\caption{Restarted Monge-growth neural $\HK$-JKO scheme with mass-corrected propagation}\label{alg:composite-HK-JKO}
\begin{algorithmic}[1]
\State Set $\widehat\mu_0^\tau=\rho_0\,dx$, $\mathfrak m_0=\widehat\mu_0^\tau(\Omega)$, and $\pi_0=\widehat\mu_0^\tau/\mathfrak m_0$.
\For{$k=0,1,\dots,K-1$}
    \State Write $\widehat\mu_k^\tau=u_{\theta,k}^\tau dx$ and choose current-measure sampling, importance sampling, or deterministic quadrature with source-mass weights $\omega_i$.
    \State Initialize a fresh admissible pair from the current-density predictor (Appendix~\ref{sec:design-notes-experiments}), or start from $(\mathrm{id},1)$ in the independent {pair-scaling and certificate} tests.
    \For{training iterations}
        \State Use the training points $X_1,\dots,X_B$ and weights from the chosen rule, refreshing samples according to the experiment protocol.
        \State Evaluate $T_\theta(X_i)$, $q_\theta(X_i)$, $J_{T_\theta}(X_i)$, and $u_{\theta,k}^\tau(X_i)$.
        \State Approximate \eqref{eq:incremental-current-objective}, using \eqref{eq:incremental-interaction-U-statistic} for iid interaction samples or \eqref{eq:incremental-interaction-quadrature} for deterministic quadrature.
        \State Update $\theta$ to decrease the objective. Compare saved parameter iterates (checkpoints) using one fixed training estimator throughout the step.
    \EndFor
    \State {\parbox[t]{0.88\linewidth}{Restore the selected parameters, fix the network normalization used for evaluation, enforce map contraction, and recompute the training gradient and growth residual. On a sample independent of training (held out) or a distinct finer quadrature grid, evaluate $\widehat{\mathcal I}_\tau(\theta)$ and $\widehat{\mathcal I}_\tau(\mathrm{id},1)=\widehat{\mathcal F}(\widehat\mu_k^\tau)$. {Reject the candidate if its objective exceeds the identity value beyond the recorded tolerance or another required validation check fails.} This is an \emph{empirical} acceptance test.}}
    \State {\parbox[t]{0.88\linewidth}{{The common neural solver uses strict rejection: if no candidate passes after the prescribed attempts, terminate with a failure; do not advance physical time or propagate an identity/growth fallback. The identity is the acceptance baseline, and the analytic pure-growth competitor is diagnostic only.}}}
    \State { Record $\mathcal A^{\rm tr}$, $\mathcal A^{\rm re}$ from \eqref{eq:cone-action-splitting}, $\min_iJ_{T_\theta}(X_i)$, {the applicable certified map-contraction bound (Appendix~\ref{sec:design-notes-experiments}),} and $\max_i{\sf d}_\Omega(X_i,T_\theta(X_i))$.}
    \State Freeze the accepted pair $(T_{\theta,k+1},q_{\theta,k+1})$ and choose transition points with source-mass weights $\omega_i$; current-measure samples use $\omega_i=\mathfrak m_k/B$.
    \State Set $Y_i=T_{\theta,k+1}(X_i)$, $\widetilde w_i=\omega_i q_{\theta,k+1}(X_i)^2$, and
    $\widehat Z_{k+1}=\mathfrak m_k^{-1}\sum_i\widetilde w_i$.
    \State Update $\mathfrak m_{k+1}=\mathfrak m_k\widehat Z_{k+1}$ and $w_i=\widetilde w_i/\sum_j\widetilde w_j$.
    \State Retain $\sum_iw_i\delta_{Y_i}$ as a weighted approximation of $\pi_{k+1}$, or resample the $Y_i$ according to $(w_i)_i$.
    \State {\parbox[t]{0.88\linewidth}{{\emph{(Density model.)}  The recursive runs of Section~\ref{sec:experiments} use density-model propagation: the log-density values \eqref{eq:incremental-log-density-update} at the mapped points $Y_i$ are fitted by a positive neural density model. Output coefficients are fitted by variable projection, and a scalar
    shift restores sampled mass. The fit must pass the recorded validation tolerances for density error, mass error, and energy increase; otherwise the step is rejected. This projection adds a per-step representation error.}}}
    \State Use the target values
    $\log u_{\rm pair}(Y_i)
    =\log u_{\theta,k}^\tau(X_i)+2\log q_{\theta,k+1}(X_i)-\log J_{T_{\theta,k+1}}(X_i)$.
    \State Define the ideal pair endpoint
    $\nu_{k+1}=(T_{\theta,k+1})_\#(q_{\theta,k+1}^2\widehat\mu_k^\tau)$.
    \State For the density-model runs set $\widehat\mu_{k+1}^\tau=\mathcal P_k\nu_{k+1}$, where $\mathcal P_k$ is the accepted mass-corrected fit; exact propagation would instead set $\widehat\mu_{k+1}^\tau=\nu_{k+1}$.
\EndFor
\end{algorithmic}
\end{algorithm}

If rigorous error bounds $|\widehat{\mathcal I}_\tau-\mathcal I_\tau|\le\varepsilon_{I,k}$ and $|\widehat{\mathcal F}-\mathcal F|\le\varepsilon_{F,k}$ are available, the stronger acceptance condition
\[
\widehat{\mathcal I}_\tau(\theta)+\varepsilon_{I,k}
\le \widehat{\mathcal F}(\widehat\mu_k^\tau)-\varepsilon_{F,k}
\]
implies the exact descent criterion \eqref{eq:descent-guard} with $\delta=0$.  The empirical guard used in the experiments does not provide these quadrature error bounds.

The inner loop needs density values at its evaluation points. Reusing a batch supplies values only at the stored points; a density model allows evaluation at new points. All recursive trajectory experiments fit a positive neural density after each step, with architectures listed in Table~\ref{tab:density-family} and fitting details in Appendix~\ref{sec:design-notes-experiments}.

The cone action and its upper bound apply to the pair-induced measure $\nu_{k+1}$ before density fitting. The fitted density is checked separately for representation error, mass error, and energy increase. These empirical checks do not certify its exact $\HK$-JKO objective or transfer the convergence result for exact pair propagation to the trained recursion. Weighted-particle propagation is an alternative, but repeated reweighting can reduce particle diversity; its effective sample size should therefore be monitored.

{
\begin{remark}[Cutoff and reaction-only degeneracy]\label{rem:selector-range-cutoff}\label{rem:fisher-rao-degeneration}
Where ${\sf d}_\Omega(x,T_\theta(x))>\pi/2$, the cone integrand is independent of the map and its $\theta_T$-gradient vanishes.  Although the exact optimal Monge map stays below the cutoff by Corollary~\ref{cor:HK-growth-couple-relations-full}, a neural iterate may enter the region where the map gradient of the cone cost vanishes.  If map learning then stalls, the update can collapse to the pure-reaction pair $T_\theta=\mathrm{id}$, producing the reaction-only limit $\partial_t\rho=-4\rho g_\rho$ while still satisfying the descent criterion with $\delta=0$.  Energy monotonicity is therefore insufficient: the maximal displacement and the transport share in \eqref{eq:cone-action-splitting} should be monitored.  This can be mitigated by constraining the displacement radius, by choosing the
transport--reaction scaling so that $\operatorname{diam}\Omega<\pi/2$, or by using
a parametrization for which the cutoff is inactive.  Penalties or smoothing,
however, define a modified numerical objective and do not by themselves guarantee
either $T_\theta(\Omega)\subset\Omega$ or $J_{T_\theta}>0$.

The transport share can also be used as an optional additional acceptance criterion.  Fix $\kappa\in(0,1)$ and accept a trained pair only if, in addition to the descent criterion \eqref{eq:descent-guard}, it satisfies
\begin{equation}\label{eq:transport-share-guard}
\mathcal A^{\rm tr}_\mu(T,q)\ge\kappa\,\mathcal A_\mu(T,q).
\end{equation}
Since \eqref{eq:transport-share-guard} only shrinks the accepted set, Proposition~\ref{prop:one-step-certificate} applies verbatim to every pair that passes both criteria, while $T=\mathrm{id}$ is now rejected whenever $\mathcal A_\mu(T,q)>0$.  Splitting the metric--dissipation identity \eqref{eq:exact-step-dissipation-main} supplies the theoretical reference share $\int|\nabla g_{\rho_1}|^2\rho_1\,dx/I_g(\rho_1)$ for interpreting choices of $\kappa$.  Like the descent criterion itself, this additional test is empirical unless its quadrature error is controlled rigorously.

\end{remark}
}

\FloatBarrier
\section{\texorpdfstring{{Guarantees for parametric and computed steps}}
{Guarantees for parametric and computed steps}}\label{sec:analytical-results}\label{sec:representability}

This section gives two complementary guarantees. Theorem~C is an a priori approximation result: if a regular optimal pair is approximable, there exist admissible parametric pairs whose endpoint densities and objective values converge to the exact ones. Theorem~D gives an a posteriori error bound for an admissible computed measure under additional convexity and dual-feasibility assumptions. It does not require convergence of the optimizer or an approximation assumption on the parametric class.

\subsection{\texorpdfstring{{Theorem~C: one-step representability}}{Theorem C: one-step representability}}\label{sec:one-step-representability}

\begin{assumption}[Approximation of a regular Monge-growth pair]\label{ass:incremental-neural-recovery}
Let $(T,q)$ be a Monge-growth pair that is regular in the following sense:
$T$ is an
orientation-preserving $C^1$ diffeomorphism \emph{of $\Omega$ onto itself} with
$\inf_\Omega J_T>0$, and $q\in C(\Omega)$ satisfies
$0<\inf q\le\sup q<\infty$.
The chosen network classes admit parameters $\theta_n$ such that
each $T_{\theta_n}$ is an orientation-preserving
$C^1$ diffeomorphism of $\Omega$ onto itself with $\inf_\Omega J_{T_{\theta_n}}>0$, and
\begin{equation}\label{eq:incremental-network-recovery}
\|T_{\theta_n}-T\|_{C^1(\Omega)}
+\|q_{\theta_n}-q\|_{L^\infty(\Omega)}\longrightarrow0.
\end{equation}
\end{assumption}

Surjectivity onto $\Omega$ is included because Appendix~\ref{sec:proof-third-theorem}
compares induced densities through globally defined inverses $T_{\theta_n}^{-1}$.
The constrained residual map \eqref{eq:residual-map}, its finite compositions, and the
neural ODE flow \eqref{eq:neural-ode-map} are onto under the respective conditions in
Section~\ref{sec:incremental-neural-param}.

\begingroup
\begin{theorem}[One-step variational approximation and asymptotic tightness of the cone action]\label{thm:surrogate-one-step-consistency}
Assume the hypotheses of Theorem~\ref{thm:main-fully-implicit-HK-JKO-step}.  Let $\bar\mu=\bar\rho\,dx$ with $\bar\rho\in C(\Omega)$ and $C^{-1}\le\bar\rho\le C$ on $\Omega$ for some $C\ge1$, and let
\[
\mu_*=\rho_*\,dx
\in
\operatorname*{argmin}_{\sigma\in\mathcal M_+(\Omega)}
\mathcal J_\tau(\sigma\mid\bar\mu).
\]
Suppose that an optimal Monge-growth pair $(T_*,q_*)$ from $\bar\mu$ to $\mu_*$ is regular and satisfies the approximation property of Assumption~\ref{ass:incremental-neural-recovery} in the chosen neural class.  Then, for every $\varepsilon>0$, there is a neural pair $(T_\theta,q_\theta)$ such that, with
\begin{equation*}
(T_\theta)_\#(q_\theta^2\bar\mu)=u_\theta\,dx=:\nu_\theta,
\end{equation*}
we have
\begin{equation}\label{eq:one-step-uniform-density-recovery}
\|u_\theta-\rho_*\|_{L^\infty(\Omega)}\le\varepsilon,
\end{equation}
\begin{equation}\label{eq:one-step-incremental-objective-recovery}
0\le
\mathcal I_\tau(T_\theta,q_\theta\mid\bar\mu)
-\mathcal J_\tau(\mu_*\mid\bar\mu)
\le\varepsilon,
\end{equation}
and
\begin{equation}\label{eq:one-step-original-jko-recovery}
0\le
\mathcal J_\tau(\nu_\theta\mid\bar\mu)
-\mathcal J_\tau(\mu_*\mid\bar\mu)
\le\varepsilon.
\end{equation}
Moreover, the primal cone majorization is quantitatively tight:
\begin{equation}\label{eq:one-step-cone-tightness}
0\le
\mathcal A_{\bar\mu}(T_\theta,q_\theta)
-\HK^2(\bar\mu,\nu_\theta)
\le 2\tau\varepsilon.
\end{equation}
\end{theorem}
\begin{remark}[Role and scope of the regular-pair assumption]
\label{rem:one-step-neural-scope}\label{rem:regular-pair-sufficient}
Assumption~\ref{ass:incremental-neural-recovery} imposes regularity
and approximation conditions, as in classical convergence analyses. 
It separates the regularity of the exact step from the approximation capability of the chosen neural class. 
Theorem~\ref{thm:surrogate-one-step-consistency} transfers the
assumed pair approximation to density and objective convergence and
vanishing cone-action overestimation. The $C^1\times L^\infty$ topology
is a natural sufficient choice because the induced density
\eqref{eq:incremental-induced-density} depends on the Jacobian of the
transport map: approximation only in $L^2(\bar\mu)$, with bounded growth
factors, controls the $\HK$ distance between the endpoints
(Lemma~\ref{lem:pair-to-HK-stability}) but does not by itself control
their internal energy.

The exact theory provides essential injectivity but does not establish
the stronger regularity assumed here. For the approximating pairs, the
two-block residual architecture satisfies the structural requirements:
under the contraction and domain constraints of
Section~\ref{sec:incremental-neural-param}, its maps are
orientation-preserving $C^1$ diffeomorphisms of $\Omega$ onto itself
with Jacobian bounded away from zero, and its growth factors are
positive and continuous. Every accepted map satisfies its contraction
certificate, with positive sampled Jacobians
(Tables~\ref{tab:diagnostics} and~\ref{tab:tau-scaling}).
The experiments also provide evidence of small-time consistency:
the rescaled-field errors of independently trained pairs decrease under
time-step refinement (Table~\ref{tab:tau-scaling}), and the relative $L^2$ residual of the PDE decreases from $1.8\%$ to $0.07\%$
(Figure~\ref{fig:hk-euler-check}). 
These observations do not establish
the fixed-$\tau$ approximation property
\eqref{eq:incremental-network-recovery}.
Proposition~\ref{prop:one-step-certificate} and
Theorem~\ref{thm:a-posteriori} provide complementary guarantees
without this assumption.
For the potential-generated parametrization, sufficient conditions
and the remaining regularity issue are discussed in
Appendix~\ref{sec:design-notes-Zadm}.
\end{remark}

\noindent
\textit{Idea of the proof (Appendix~\ref{sec:proof-third-theorem}).}
The argument transfers convergence through the change-of-variables formula.
Assumption~\ref{ass:incremental-neural-recovery} gives $T_{\theta_n}\to T_*$ in $C^1$ and
$q_{\theta_n}\to q_*$ uniformly; since $J_{T_*}$ is bounded away from zero, so are the
$J_{T_{\theta_n}}$ for large $n$, and the inverse maps converge uniformly.  Writing the
induced densities through \eqref{eq:incremental-induced-density} at the same target point then
gives $u_{\theta_n}\to\rho_*$ uniformly, with all densities eventually bounded above and away from zero, so the three energy terms and the cone integrand converge and the pair
objective converges to $\mathcal I_\tau(T_*,q_*\mid\bar\mu)$.  By
Proposition~\ref{prop:incremental-HK-JKO-reparam} the latter equals the JKO minimum, and
every competitor is bounded below by that minimum, which makes the gap nonnegative and
vanishing.  The tightness statement \eqref{eq:one-step-cone-tightness} follows algebraically from the definitions:
\[
\begin{aligned}
&\mathcal A_{\bar\mu}(T_\theta,q_\theta)-\HK^2(\bar\mu,\nu_\theta)=2\tau\bigl[\mathcal I_\tau(T_\theta,q_\theta\mid\bar\mu)-\mathcal J_\tau(\nu_\theta\mid\bar\mu)\bigr],
\end{aligned}
\]
and the right-hand side is dominated by the pair-objective gap.  This last identity also shows directly how the metric overestimate is measured: it is exactly the excess of the pair objective over the exact JKO objective at the induced state.

\medskip
\endgroup

For fixed $\tau$, Corollary~\ref{cor:recursive-finite-time-tracking}
in Appendix~\ref{sec:proof-third-theorem} extends this recovery result
to any fixed finite number of JKO steps, provided the regular-pair
approximation hypothesis holds at each step.

{
\subsection{\texorpdfstring{{Theorem~D: an a posteriori error estimate from a feasible dual pair}}{Theorem D: an a posteriori error estimate from a feasible dual pair}}\label{sec:a-posteriori}
% The sandwich above can be applied to the optimal value of the step, not only to the metric
% cost of a fixed candidate.  Written with $F^*(a)=e^a-1$, the dual
% \eqref{eq:HK-dual} is a supremum of functionals affine in each measure argument.  Hence a
% globally feasible dual pair gives an affine lower bound for
% $\sigma\mapsto\HK^2(\sigma,\bar\mu)$, while the returned Monge-growth pair gives the primal
% upper bound through the cone action.  The resulting gap is an a posteriori certificate once
% global dual feasibility has itself been verified.  The proof and the algebraic decomposition of the gap are in Appendix~\ref{sec:proof-a-posteriori}.
Theorem~C establishes approximation by a parametric class, conditional on a regular optimal pair. We now bound the error of an individual computed measure. Its pair objective gives an upper bound on the exact JKO objective by \eqref{eq:primal-upper-certificate}. A globally feasible $\HK$ dual pair gives a lower bound on the squared distance. When the interaction is positive semidefinite, this dual bound also leads to a lower bound for the exact JKO minimum. The difference between the primal and dual bounds therefore controls the candidate's JKO objective gap.

Theorem~\ref{thm:a-posteriori} also bounds the internal-energy convexity remainder $U(u)-U(\rho_1)-U'(\rho_1)(u-\rho_1)$ integrated over $\Omega$. For Boltzmann entropy, this gives an $L^1$ density error bound. The proof and the explicit gap decomposition appear in Appendix~\ref{sec:proof-a-posteriori}.

\begin{theorem}[Theorem~D: an a posteriori error estimate for an admissible pair]\label{thm:a-posteriori}
Assume the hypotheses of Theorem~\ref{thm:full-step-regularity} and suppose that $W$ is
positive semidefinite, in the sense that
\[
\iint_{\Omega\times\Omega}W(x,y)\,d\eta(x)d\eta(y)\ge0
\qquad\text{for every finite signed measure }\eta.
\]
Then the exact step $\mu_1=\rho_1dx=\mathcal S_\tau(\bar\mu)$ is unique.  Let $(T,q)$ be
any admissible pair for $\bar\mu$ such that
$\nu:=T_\#(q^2\bar\mu)=u\,dx$ and
$\mathcal I_\tau(T,q\mid\bar\mu)<\infty$.  Let $\zeta,\xi$ be bounded Borel functions
satisfying the global dual-feasibility condition
\begin{equation}\label{eq:ap-feasible}
\zeta(x)+\xi(y)\le c_{\HK}(x,y)
\qquad\text{for every }x,y\in\Omega.
\end{equation}
Put $U^*(p):=\sup_{s\ge0}\{ps-U(s)\}$ and
\begin{equation}\label{eq:ap-field}
a:=V+W[u]+\frac1{2\tau}\bigl(1-e^{-\zeta}\bigr),
\qquad
\mathcal R:=\mathcal I_\tau(T,q\mid\bar\mu)
+\int_\Omega U^*(-a)\,dx+\mathcal W(\nu)
-\frac1{2\tau}\int_\Omega\bigl(1-e^{-\xi}\bigr)d\bar\mu .
\end{equation}
Then $\mathcal R\ge0$ and
\begin{equation}\label{eq:ap-objective-gap}
0\le
\mathcal J_\tau(\nu\mid\bar\mu)-\mathcal J_\tau(\mu_1\mid\bar\mu)
\le \mathcal R.
\end{equation}
Moreover,
\begin{equation}\label{eq:ap-estimate}
\int_\Omega
\bigl[U(u)-U(\rho_1)-U'(\rho_1)(u-\rho_1)\bigr]dx
\le \mathcal R.
\end{equation}
\end{theorem}

\begin{corollary}[Boltzmann $L^1$ error bound]\label{cor:ap-boltzmann-L1}
Under the hypotheses of Theorem~\ref{thm:a-posteriori}, assume in addition that
$U(s)=s\log s-s$.  Then
\begin{equation}\label{eq:ap-boltzmann-L1}
\|u-\rho_1\|_{L^1(\Omega)}^2
\le2\bigl(\nu(\Omega)+\mu_1(\Omega)\bigr)\mathcal R
\le2\bigl(\nu(\Omega)+M_{\tau,\bar\mu}\bigr)\mathcal R.
\end{equation}
\end{corollary}

The gap decomposes into a Fenchel--Young residual for the discrete Euler relation and a primal--dual metric gap; the exact decomposition is Proposition~\ref{prop:ap-gap-decomposition} in Appendix~\ref{sec:proof-a-posteriori}.

\begin{remark}[Constructing a feasible pair; monotonicity of the error bound in $\zeta$]\label{rem:ap-dual-construction}
Global feasibility \eqref{eq:ap-feasible} need not be verified after the fact; it can be arranged by construction.  Fix any bounded Borel $\xi$ and set
\begin{equation*}
\zeta(x):=\inf_{y\in\Omega}\bigl\{c_{\HK}(x,y)-\xi(y)\bigr\}.
\end{equation*}
Then \eqref{eq:ap-feasible} holds for every pair $(x,y)$ by definition, and
$-\sup_\Omega\xi\le\zeta\le-\inf_\Omega\xi$ because $c_{\HK}\ge0$ and $c_{\HK}(x,x)=0$, so $\zeta$ is admissible in Theorem~\ref{thm:a-posteriori}.  If moreover the transport--reaction weighting is normalized so that $\diam\Omega<\pi/2$,  then $c_{\HK}$ is Lipschitz on $\Omega\times\Omega$ and $\zeta$ inherits its Lipschitz constant.

A pointwise lower bound for the exact $c$-transform preserves feasibility and gives a possibly larger error bound.  Suppose $\zeta_-\le\zeta$ pointwise.  Then $\zeta_-$ is again feasible.  Moreover, $\zeta\mapsto a$ is nondecreasing by \eqref{eq:ap-field}, and $U^*$ is nondecreasing on $\mathbb R$, so replacing $\zeta$ by $\zeta_-$ can only increase $\mathcal R$.  Hence \emph{any} rigorous pointwise lower bound for the exact $c$-transform yields a valid, possibly larger, error bound.  For Lipschitz $\xi$ and $\diam\Omega<\pi/2$, one such lower bound is available on a finite grid $\{y_j\}\subset\Omega$ with fill radius $h:=\sup_{y\in\Omega}\min_j{\sf d}_\Omega(y,y_j)$:
\begin{equation}\label{eq:ap-grid-lower-bound}
\zeta_-(x):=\min_j\bigl\{c_{\HK}(x,y_j)-\xi(y_j)\bigr\}-Lh\le\zeta(x),
\qquad
L:=\Lip(c_{\HK})+\Lip(\xi),
\end{equation}
uniformly in $x$.  This correction is what distinguishes a rigorous evaluation from the batchwise minimum excluded in Remark~\ref{rem:ap-scope}.  An unadjusted minimum over sampled points is an \emph{upper} bound for $\zeta$ and does not guarantee feasibility, whereas subtracting $Lh$ restores it.  Global feasibility can therefore be enforced by construction.  A useful numerical error bound still requires deterministic control of the three quadratures in \eqref{eq:ap-field} and a sufficiently tight dual choice.
\end{remark}

{
\begin{remark}[Scope]\label{rem:ap-scope}
Theorem~\ref{thm:a-posteriori} avoids approximation and regular-pair assumptions, but it applies only to $W=0$ or, more generally, a positive-semidefinite interaction and requires global dual feasibility.  Numerical use additionally requires a rigorous evaluation of the gap.  In particular, replacing the exact $c$-transform by a batchwise minimum does not verify \eqref{eq:ap-feasible}, whereas the corrected variant \eqref{eq:ap-grid-lower-bound} of Remark~\ref{rem:ap-dual-construction} does.  The estimate controls only the induced density, not the map-growth pair or its cone action.  Finite-horizon propagation would require an additional stability estimate for $\mathcal S_\tau$ and rigorously controlled per-step gaps; no such propagation is proved here.
\end{remark}
}
}

\input{sections/experiments}

\section{Discussion and conclusion}\label{sec:conclusion}
We developed a Monge-growth neural minimizing-movement scheme
for $\HK$ dynamics with fully implicit pair interactions.  At each step, a fresh spatial map and growth factor act on the current measure through a single-level objective.  The key design choice is the one-sided primal approximation.  Since the cone action majorizes $\HK^2$ and vanishes at the identity pair, every admissible pair satisfying the descent criterion {has an exact, unprojected endpoint obeying the discrete energy inequality; the computed density fit is checked separately}.  The criterion with $\delta=0$ is compatible with $T=\mathrm{id}$, so the transport--reaction split is needed to detect reaction-only degeneration.

For the exact step, Theorems~\ref{thm:main-fully-implicit-HK-JKO-step} and~\ref{thm:full-step-regularity} establish existence, mass control, positivity, density bounds, the discrete Euler--Lagrange equation, and the metric--dissipation identity on bounded convex domains and flat tori.  With Lipschitz $V$ and $W$, the density is Lipschitz as well.  At the parametric level, Theorem~\ref{thm:surrogate-one-step-consistency} gives one-step recovery under a regular-pair approximation hypothesis; Corollary~\ref{cor:recursive-finite-time-tracking} extends this recovery to each fixed finite horizon by finite induction.  These are conditional representability results and do not imply convergence of the trained stochastic solver. {In particular, approximation by the fixed two-block architecture used in the experiments is not proved.}

{Under a positive-semidefinite interaction, Theorem~\ref{thm:a-posteriori} complements the descent criterion with a one-step state-error estimate.  A globally feasible dual pair supplies a rigorous lower bound for the exact JKO minimum.  Combined with the primal cone-action upper bound, it controls the JKO-objective gap and the convexity remainder of $U$, with an explicit $L^1$ bound for Boltzmann entropy.  Remark~\ref{rem:ap-scope} records the limitations of this estimate and the additional ingredients needed for finite-horizon propagation.}  Section~\ref{sec:experiments} checks the small-step pair structure, evaluates a deterministic-grid version of the primal--dual certificate, and tests the transport--reaction and fully implicit design choices by ablation.  The 8D example reproduces Gaussian dynamics; projection errors, optimizer work, and interaction residual checks are reported in Appendix~\ref{sec:design-notes-experiments}. These experiments do not replace the missing continuum quadrature control or step-uniform estimates.  A general $\tau\downarrow0$ theory would additionally require stepwise estimates uniform in the step index.

\section*{Acknowledgements}
Geuntaek Seo was supported by the National Research Foundation of Korea (NRF) grant funded by the Korea government(MSIT) (RS-2023-00219980).
Hwijae Son was supported by National Research Foundation of Korea (NRF) grants funded by the Korean government (Ministry of Science and ICT, MSIT) (No. RS-2026-25491905 and RS-2026-25506577).
Hyung Ju Hwang was supported by the National Research Foundation of Korea(NRF) grant funded by the Korea government(MSIT) (RS-2023-00219980 and RS-2022-00165268).

\appendix

\section{Proof of Theorem~\ref{thm:main-fully-implicit-HK-JKO-step} (Theorem~A)}\label{sec:proof-first-theorem}
\begin{proof}[Proof of Theorem~\ref{thm:main-fully-implicit-HK-JKO-step}]
We prove existence, the quantitative mass estimate, strict positivity in this order.

\paragraph{Step 1: a coercive lower bound depending only on the mass.}
The zero measure is admissible and
\begin{equation}\label{eq:zero-objective-value}
\mathcal J_\tau(0\mid\bar\mu)
=|\Omega|U(0)+\frac{\bar m}{2\tau},
\qquad \bar m:=\bar\mu(\Omega),
\end{equation}
because $\HK^2(0,\bar\mu)=\bar m$.  Hence the infimum of the one-step
functional is finite.  Moreover, every competitor with finite energy is of the
form $\sigma=r\,dx$.  Writing $m=\sigma(\Omega)=\int_\Omega r\,dx$, Jensen's
inequality gives
\[
\int_\Omega U(r)\,dx
\ge |\Omega|U\!\left(\frac{m}{|\Omega|}\right).
\]
Since $V\ge\underline V$ on $\Omega$ and $W\ge\underline W$ on $\Omega\times\Omega$,
\[
\int_\Omega V\,d\sigma\ge \underline V m,
\qquad
\frac12\iint_{\Omega\times\Omega}W(x,y)\,d\sigma(x)d\sigma(y)
\ge \frac{\underline W}{2}m^2.
\]
Finally, the mass estimate \eqref{eq:HK-mass-difference} yields
\[
\HK^2(\sigma,\bar\mu)
\ge (\sqrt m-\sqrt{\bar m})^2.
\]
Combining the preceding inequalities, we obtain the scalar lower bound
\begin{equation}\label{eq:full-objective-scalar-lower-bound}
\mathcal J_\tau(\sigma\mid\bar\mu)
\ge \Gamma_{\tau,\bar m}(m).
\end{equation}
This estimate is the point at which the possibly negative quadratic part of
the interaction is taken into account.  The assumed coercivity of
$\Gamma_{\tau,\bar m}$ therefore controls the total mass even when the
interaction energy itself is not nonnegative.

\paragraph{Step 2: compactness of a minimizing sequence.}
Let $\sigma_n=r_n\,dx$ be a minimizing sequence.  By discarding finitely many
terms, we may assume
\[
\mathcal J_\tau(\sigma_n\mid\bar\mu)
\le \mathcal J_\tau(0\mid\bar\mu)+1.
\]
If $m_n=\sigma_n(\Omega)$, then \eqref{eq:full-objective-scalar-lower-bound}
and the coercivity of $\Gamma_{\tau,\bar m}$ imply
\begin{equation}\label{eq:minimizing-sequence-mass-bound}
\sup_n m_n<\infty.
\end{equation}
The potential and interaction terms are bounded from below on this
mass-bounded family:
\[
\int_\Omega V r_n\,dx\ge \underline V m_n,
\qquad
\frac12\iint W(x,y)r_n(x)r_n(y)\,dx\,dy
\ge \frac{\underline W}{2}m_n^2.
\]
Since the proximal term is nonnegative and the objective values are bounded
from above, it follows that
\begin{equation}\label{eq:internal-energy-uniform-bound}
\sup_n\int_\Omega U(r_n)\,dx<\infty.
\end{equation}

We spell out why this gives compactness of the densities.  Convexity and
finiteness of $U$ provide an affine lower bound $U(s)\ge-a s-b$ on
$[0,\infty)$ for suitable $a,b\ge0$.  Thus
\[
\Psi(s):=U(s)+as+b
\]
is nonnegative and remains superlinear.  From
\eqref{eq:minimizing-sequence-mass-bound} and
\eqref{eq:internal-energy-uniform-bound},
\[
\sup_n\int_\Omega\Psi(r_n)\,dx<\infty.
\]
The de la Vall\'ee-Poussin criterion shows that $(r_n)_n$ is uniformly
integrable.  By the Dunford--Pettis theorem, after passing to a subsequence,
there is $r\in L^1(\Omega)$ such that
\begin{equation}\label{eq:weak-L1-minimizing-sequence}
r_n\rightharpoonup r
\qquad\text{weakly in }L^1(\Omega).
\end{equation}
The positive cone of $L^1$ is weakly closed, so $r\ge0$.  Setting
$\sigma=r\,dx$, \eqref{eq:weak-L1-minimizing-sequence} also implies the narrow
convergence $\sigma_n\rightharpoonup\sigma$ in $\mathcal M_+(\Omega)$.
In particular, the weak $L^1$ argument rules out the creation of a singular
part in the limit; this is the compactness role played by the superlinearity of
$U$.

\paragraph{Step 3: passage to the limit.}
The convex integral functional is weakly lower semicontinuous on $L^1$, hence
\begin{equation}\label{eq:internal-energy-lsc-existence}
\int_\Omega U(r)\,dx
\le \liminf_{n\to\infty}\int_\Omega U(r_n)\,dx.
\end{equation}
Because $V\in C(\Omega)\subset L^\infty(\Omega)$,
\[
\int_\Omega Vr_n\,dx\longrightarrow\int_\Omega Vr\,dx.
\]
For the interaction term, narrow convergence together with the uniform mass
bound implies
\[
\sigma_n\otimes\sigma_n\rightharpoonup\sigma\otimes\sigma
\qquad\text{narrowly in }\mathcal M_+(\Omega\times\Omega).
\]
Indeed, the convergence is immediate for functions of the form
$f(x)g(y)$, and finite sums of such functions are uniformly dense in
$C(\Omega\times\Omega)$.  Since $W$ is continuous and bounded on $\Omega\times\Omega$,
\begin{equation}\label{eq:interaction-continuity-existence}
\frac12\iint W(x,y)\,d\sigma_n(x)d\sigma_n(y)
\longrightarrow
\frac12\iint W(x,y)\,d\sigma(x)d\sigma(y).
\end{equation}
The squared Hellinger--Kantorovich distance is lower semicontinuous under
narrow convergence of finite measures; this follows, for instance, from its
entropy--transport formulation \eqref{eq:HK-entropy-transport}; see
\cite{LMS18}.  Consequently,
\[
\HK^2(\sigma,\bar\mu)
\le\liminf_{n\to\infty}\HK^2(\sigma_n,\bar\mu).
\]
Together with \eqref{eq:internal-energy-lsc-existence} and
\eqref{eq:interaction-continuity-existence}, this gives
\[
\mathcal J_\tau(\sigma\mid\bar\mu)
\le\liminf_{n\to\infty}\mathcal J_\tau(\sigma_n\mid\bar\mu).
\]
Thus $\sigma=\rho_1dx$ is a minimizer.

\paragraph{Step 4: the quantitative mass bound.}
Let $\mu_1=\rho_1dx$ be any minimizer and put $m_1=\mu_1(\Omega)$.  Minimality
against the zero measure, \eqref{eq:zero-objective-value}, and
\eqref{eq:full-objective-scalar-lower-bound} give
\[
\Gamma_{\tau,\bar m}(m_1)
\le \mathcal J_\tau(\mu_1\mid\bar\mu)
\le |\Omega|U(0)+\frac{\bar m}{2\tau}.
\]
The scalar sublevel appearing in \eqref{eq:full-mass-sublevel-bound} is
nonempty because it contains $m=0$, and it is bounded by the coercivity of
$\Gamma_{\tau,\bar m}$.  Hence $M_{\tau,\bar\mu}<\infty$ and
$m_1\le M_{\tau,\bar\mu}$.

\paragraph{Step 5: strict positivity when $U'(0+)=-\infty$.}
Assume now that $U'(0+)=-\infty$ and suppose, by contradiction, that
$B:=\{x\in\Omega:\rho_1(x)=0\}$ has positive Lebesgue measure.

We first estimate the proximal term.  For \(\delta>0\), choose an
entropy--transport plan \(\gamma\) for \((\mu_1,\bar\mu)\), with marginals
\(\gamma_0,\gamma_1\), such that
\begin{equation}\label{eq:proximal--1}
\begin{aligned}
&\int_\Omega
F\!\left(\frac{d\gamma_0}{d\mu_1}\right)d\mu_1
+
\int_\Omega
F\!\left(\frac{d\gamma_1}{d\bar\mu}\right)d\bar\mu
+
\int_{\Omega\times\Omega}c_{\HK}\,d\gamma
\le
\HK^2(\mu_1,\bar\mu)+\delta .
\end{aligned}    
\end{equation}
Since \(\gamma_0\ll\mu_1\) and \(\mu_1(B)=0\), we have
\(\gamma_0(B)=0\).   
For $\varepsilon>0$, if we define
\[
\mu_\varepsilon
:=\mu_1+\varepsilon\mathbbm1_B\,dx,
\]
then \(\mu_1\ll\mu_\varepsilon\), and hence
\(\gamma_0\ll\mu_\varepsilon\).  
Thus the same plan \(\gamma\) is
admissible for \((\mu_\varepsilon,\bar\mu)\), and
$$
\frac{d\gamma_0}{d\mu_\varepsilon}
=
\frac{d\gamma_0}{d\mu_1}
\quad\text{on }\Omega\setminus B,
\qquad
\frac{d\gamma_0}{d\mu_\varepsilon}=0
\quad\text{on }B.
$$
Therefore
\begin{equation}\label{eq:proximal--2}
\begin{aligned}
\int_\Omega
F\!\left(\frac{d\gamma_0}{d\mu_\varepsilon}\right)d\mu_\varepsilon
&=
\int_{\Omega\setminus B}
F\!\left(\frac{d\gamma_0}{d\mu_1}\right)d\mu_1
+
\int_BF(0)\,d\mu_\varepsilon
=
\int_\Omega
F\!\left(\frac{d\gamma_0}{d\mu_1}\right)d\mu_1
+\varepsilon|B|,
\end{aligned}    
\end{equation}
where we used $F(0)=1$.  
The second entropy term and the transport-cost term are
unchanged.  Since \(\HK^2(\mu_\varepsilon,\bar\mu)\) is the infimum over
all admissible entropy--transport plans, using this same \(\gamma\) as a
competitor gives
$$
\begin{aligned}
\HK^2(\mu_\varepsilon,\bar\mu)
&\le
\int_\Omega
F\!\left(\frac{d\gamma_0}{d\mu_\varepsilon}\right)d\mu_\varepsilon
+
\int_\Omega
F\!\left(\frac{d\gamma_1}{d\bar\mu}\right)d\bar\mu
+
\int_{\Omega\times\Omega}c_{\HK}\,d\gamma.
\end{aligned}
$$
By combining \eqref{eq:proximal--1} and \eqref{eq:proximal--2}, and
letting \(\delta\downarrow0\), we obtain
\begin{equation}\label{eq:vacuum-mass-addition-HK}
\HK^2(\mu_\varepsilon,\bar\mu)
\le
\HK^2(\mu_1,\bar\mu)+\varepsilon|B|.
\end{equation}

Since $W$ is symmetric, the interaction increment has the exact expansion
\[
\begin{aligned}
&\frac12\iint W(x,y)\,d\mu_\varepsilon(x)d\mu_\varepsilon(y)
-\frac12\iint W(x,y)\,d\mu_1(x)d\mu_1(y)\\
&\qquad=
\varepsilon\int_B W[\mu_1](x)\,dx
+\frac{\varepsilon^2}{2}
\iint_{B\times B}W(x,y)\,dx\,dy.
\end{aligned}
\]
Using minimality of $\mu_1$ and \eqref{eq:vacuum-mass-addition-HK}, we obtain
\begin{equation}\label{eq:positivity-competitor-estimate}
\begin{aligned}
0\le{}&\mathcal J_\tau(\mu_\varepsilon\mid\bar\mu)
-\mathcal J_\tau(\mu_1\mid\bar\mu)\\
\le{}&|B|\bigl(U(\varepsilon)-U(0)\bigr)
+\varepsilon\int_B\bigl(V+W[\mu_1]\bigr)\,dx
+\frac{\varepsilon^2}{2}
\iint_{B\times B}W(x,y)\,dx\,dy
+\frac{\varepsilon|B|}{2\tau}.
\end{aligned}
\end{equation}
Because $V,W$ are bounded and $\mu_1$ has finite mass, all terms on the
right-hand side except the internal-energy increment are bounded from above by
$C\varepsilon|B|+C\varepsilon^2|B|^2$, with $C$ independent of
$\varepsilon$.  Dividing \eqref{eq:positivity-competitor-estimate} by
$\varepsilon|B|$ gives
\[
0\le \frac{U(\varepsilon)-U(0)}{\varepsilon}
+C+C\varepsilon|B|.
\]
The right-hand side tends to $-\infty$ as $\varepsilon\downarrow0$ because
$U'(0+)=-\infty$, which is impossible.  Thus $\rho_1>0$ almost everywhere.  Every nonempty
relatively open subset of $\Omega$ has positive Lebesgue measure, so it has
positive $\mu_1$-mass; hence $\operatorname{supp}\mu_1=\Omega$.
\end{proof}

\begin{proof}[Proof of Corollary~\ref{cor:explicit-mass-bound}]
We only need a coercive estimate in the total mass.  We obtain it by retaining the affine lower bound for the internal energy and the negative part of the external potential, and then compare the minimizer with the zero measure.

Let $\sigma=r\,dx$ have mass $m$.  Under the assumptions of the corollary,
\[
\int_\Omega U(r)\,dx\ge am-C_a|\Omega|,
\qquad
\int_\Omega V\,d\sigma\ge-\|V_-\|_\infty m,
\]
while the interaction and proximal terms are nonnegative.  Therefore
\[
\mathcal J_\tau(\sigma\mid\bar\mu)
\ge (a-\|V_-\|_\infty)m-C_a|\Omega|.
\]
For a minimizer $\mu_1$, comparison with the zero measure gives
\[
(a-\|V_-\|_\infty)\mu_1(\Omega)-C_a|\Omega|
\le |\Omega|U(0)+\frac{\bar m}{2\tau}.
\]
Since $a>\|V_-\|_\infty$, rearranging proves
\eqref{eq:explicit-full-mass-bound}.
\end{proof}

\section{Proof of Proposition~\ref{prop:HK-Monge-representation}}\label{sec:proof-monge-representation}
\begin{proof}[Proof of Proposition~\ref{prop:HK-Monge-representation}]
We treat the two state spaces separately.  In the Euclidean case the cited Monge theorem applies directly after extension by zero; on the torus we use the logarithmic-potential formulation and then verify essential injectivity from the reversed optimizer.

In the Euclidean case, extend the measures by zero outside $\Omega$.  Since both densities are positive almost everywhere, both supports equal $\Omega$.  Thus each support lies in the $\pi/2$-neighborhood of the other, so the pair is \emph{strongly reduced} in the sense of \cite[Definition~2.6]{LMS23}; equivalently, there is no non-reduced marginal part to discard.  Theorem~3.3(4)--(5) and Corollary~3.5 of \cite{LMS23} then give the Monge-growth representation, positivity of the growth factor, and essential injectivity.  Convexity keeps every selected segment in $\Omega$.

We next turn to the flat torus, where the same conclusion is obtained through the Riemannian entropy--transport theory.

Now let $\Omega=\mathbb T^d$ and regard it as a compact flat Riemannian manifold.  The two measures have full support and form an admissible pair for the $\HK$ cost, since $c_{\HK}(x,x)=0$.  The Lipschitz-potential and Monge results of \cite[Sections~2.2 and~3.4.1]{GGV25} give finite Lipschitz logarithmic potentials $(z_0,z_1)$ and a unique entropy--transport optimizer
\[
\gamma=(\mathrm{id},T)_\#\gamma_0,
\qquad
\gamma_0=e^{-z_0}\mu_0,
\qquad
\gamma_1=e^{-z_1}\mu_1,
\]
concentrated on the graph of the Riemannian $c$-exponential map
\[
T(x)=\exp_x\!\left(
-\arctan\!\left(\frac{|\nabla z_0(x)|}{2}\right)
\frac{\nabla z_0(x)}{|\nabla z_0(x)|}
\right)
\]
for $\mu_0$-almost every differentiability point, with the usual convention when $\nabla z_0=0$.  The corresponding growth factor is
\[
q(x)=e^{-z_0(x)}\left(1+\frac14|\nabla z_0(x)|^2\right)^{1/2}>0,
\]
and the Monge formulation gives both identities in \eqref{eq:injective-growth-representation}.

It remains to verify essential injectivity of the torus map.  For this we compare the forward optimizer with the optimizer of the reversed problem.

Apply the same result to the reversed pair $(\mu_1,\mu_0)$ and denote its Monge map by $S$, so that the reversed optimizer is $(\mathrm{id},S)_\#\gamma_1$.  The entropy--transport cost is symmetric under transposition, so $\gamma^{\mathsf T}$ is optimal for the reversed pair; by its uniqueness,
$\gamma^{\mathsf T}=(\mathrm{id},S)_\#\gamma_1$.  A plan that is concentrated both on the graph of $T$ read forwards and on the graph of $S$ read backwards satisfies $S(T(x))=x$ for $\gamma_0$-almost every $x$, so $T$ is injective off a $\gamma_0$-null set.  Since $z_0$ is finite on the compact torus, $e^{-z_0}$ is bounded away from zero and $\gamma_0$ and $\mu_0$ are equivalent; therefore $T$ is essentially injective with respect to $\mu_0$ as well.  Interchanging the two measures gives the assertion for either ordering.  All distances, gradients, and exponential maps in this case are understood in the flat Riemannian sense.
\end{proof}

\section{Proof of Theorem~\ref{thm:full-step-regularity} (Theorem~B)}\label{sec:proof-second-theorem}\label{sec:full-interaction-analysis}

\paragraph{Notation used in this appendix.}
We write the source as $\mu_0=\rho_0\,dx$ with $c_{\min}=c_0$ and $c_{\max}=C_0$.
The symbol $\Phi_1$ denotes the chemical potential \eqref{eq:Phi-full}, while $\phi_1$
denotes the dual potential of Theorem~\ref{thm:full-step-regularity}(iii); the relation
$\phi_1=-\tau\Phi_1$ is proved below, not assumed in the density-bound argument.  The first
subsection uses the forward pair $\mu_0\to\mu_1$, whereas Subsection~\ref{sec:dual-euler}
uses the reversed pair $\mu_1\to\mu_0$; the symbol $T$ is reused for these different maps.
Subsection~\ref{sec:dual-euler} is moreover stated for a generic ordered pair $(\mu,\nu)$.
In its application to a JKO step, $\mu$ is the minimizer (the new state) and $\nu$ is the
\emph{reference} measure (the preceding state).

\subsection{Monotonicity of the discrete chemical potential along the optimal map and density bounds}\label{sec:full-density-estimates}
Throughout this subsection the pair $(T,q)$ is oriented \emph{forward}, from the source to
the new state.
Let $\Omega$ be either of the two state spaces in Theorem~\ref{thm:full-step-regularity}, and let
\[
\mu_0=\rho_0dx,
\qquad 0<c_{\min}\le\rho_0\le c_{\max}<\infty,
\]
and let $\mu_1=\rho_1dx=\mathcal S_\tau(\mu_0)$.  By Theorem~\ref{thm:main-fully-implicit-HK-JKO-step}, $\rho_1>0$ almost everywhere.  By Proposition~\ref{prop:HK-Monge-representation}, let $(T,q)$ be a positive, essentially injective optimal Monge-growth pair satisfying \eqref{eq:injective-growth-representation}.  Choose Borel representatives of the self-consistent field and density and define the chemical potential by
\begin{equation}\label{eq:Phi-full}
P_1=P[\mu_1]=V+W[\mu_1],
\qquad
\Phi_1(x)=U'(\rho_1(x))+P_1(x).
\end{equation}
All pointwise relations in this subsection are understood for these representatives after null-set modifications.  Since $q>0$ almost everywhere and $T_\#(q^2\mu_0)=\mu_1$, these modifications do not affect the asserted $\mu_0$-almost-everywhere statements.

\begin{lemma}[Restriction identity for an essentially injective Monge pair]\label{lem:Monge-restriction-identity}
After modifying $T$ on a $\mu_0$-null set, there is a Borel set $G\subset\Omega$ of full $\mu_0$-measure on which $T$ is injective.  For every Borel set $A\subset G$,
\begin{equation*}
T_\#(q^2\mu_0|_A)=\mu_1|_{T(A)},
\qquad
\mu_1(T(A))=\int_Aq^2\,d\mu_0.
\end{equation*}
The same identities hold after completing the measures.
\end{lemma}
\begin{proof}
We show that restricting the source of an essentially injective Monge pair restricts the target to the corresponding image set, with no additional mass coming from outside.

Choose a Borel full-measure set on which the Borel representative of $T$ is injective.  By the Lusin--Souslin theorem, $T(A)$ is Borel for every Borel $A$ in that set.  For a Borel set $B$,
\[
T_\#(q^2\mu_0|_A)(B)
=\int_{A\cap T^{-1}(B)}q^2\,d\mu_0
=\int_{G\cap T^{-1}(B\cap T(A))}q^2\,d\mu_0
=\mu_1(B\cap T(A)),
\]
where injectivity on $G$ and $\mu_0(\Omega\setminus G)=0$ are used in the middle equality.  Completion gives the final assertion.
\end{proof}

The comparison argument below adapts the density-estimate scheme of
\cite[Sec.~4]{LM26} to the present fully implicit JKO step.
The essential difference is that the localized first variation involves the
self-consistent chemical potential $\Phi_1$ of \eqref{eq:Phi-full}, rather than
$U'(\rho_1)$ alone.  Consequently, the relevant invariant sets are superlevel
sets of $\Phi_1$, and converting this invariance into density bounds introduces
the field-oscillation loss $\omega_1$.  In this correspondence,
Lemma~\ref{lem:monotonicity-full} and
Propositions~\ref{prop:density-upper-full}--\ref{prop:density-lower-full}
play the roles of \cite[Lem.~4.1 and Props.~4.2--4.3]{LM26}, respectively.

For the remainder of this subsection, fix a full-$\mu_0$-measure Borel set
$G\subset\Omega$ on which $T$ is injective.

\begin{lemma}[Localized first variation of the full energy]\label{lem:localized-first-variation}
Assume $V\in L^\infty(\Omega)$, and $W\in L^\infty(\Omega\times\Omega)$ is symmetric, $\rho>0$ almost everywhere, and $h\in L^\infty(\Omega)$ is supported in a set $E$ on which $0<c\le\rho\le C$.  Then
\begin{equation*}
\mathcal F((\rho+sh)dx)-\mathcal F(\rho dx)
=s\int_\Omega\bigl(U'(\rho)+V+W[\rho]\bigr)h\,dx+o(s)
\end{equation*}
as $s\downarrow0$.
\end{lemma}
\begin{proof}
We compute the first variation term by term.  The internal and external parts are immediate under the localized two-sided density bound, while the interaction contributes the linear field $W[\rho]$ plus a quadratic remainder.

For sufficiently small $s$, the values of $\rho+sh$ on $E$ remain in a compact subset of $(0,\infty)$.  The mean-value theorem and dominated convergence give the internal-energy derivative.  Symmetry of $W$ gives the exact interaction expansion
\[
\mathcal W((\rho+sh)dx)-\mathcal W(\rho dx)
=s\int_\Omega W[\rho]h\,dx
+\frac{s^2}{2}\iint W(x,y)h(x)h(y)\,dx\,dy.
\]
The last term is $O(s^2)$, which proves the claim.
\end{proof}

\begin{lemma}[Localized target-replacement competitor]\label{lem:cut-and-glue-full}
Let \((T,q)\) satisfy \eqref{eq:injective-growth-representation}, and let \(A\subset G\) be Borel with \(0<\mu_0(A)<\infty\).  Put
\[
m_0=\mu_0(A),
\qquad m_1=\mu_1(T(A))=\int_Aq^2\,d\mu_0,
\qquad \alpha=\frac{m_1}{m_0},
\]
and
\[
\widetilde\mu_1=\mu_1-\mu_1|_{T(A)}+\alpha\mu_0|_A.
\]
Then \(\HK^2(\mu_0,\widetilde\mu_1)\le\HK^2(\mu_0,\mu_1)\).  Consequently, for
\[
\mu_1^s=(1-s)\mu_1+s\widetilde\mu_1
=\mu_1+s\bigl(\alpha\mu_0|_A-\mu_1|_{T(A)}\bigr),
\qquad s\in[0,1],
\]
it holds \(\HK^2(\mu_0,\mu_1^s)\le\HK^2(\mu_0,\mu_1)\).
\end{lemma}
\begin{proof}
We first construct a competitor that replaces the target mass generated from $A$ by a diagonal growth branch supported on $A$.  The choice of $\alpha$ preserves the amount of target mass.

By Lemma~\ref{lem:Monge-restriction-identity}, the target generated by the restriction to $A$ is exactly $\mu_1|_{T(A)}$.  Define a new admissible pair $(\widetilde T,\widetilde q)$ by keeping $(T,q)$ outside \(A\) and setting $\widetilde T=\mathrm{id}$, $\widetilde q=\sqrt\alpha$ on \(A\).  Its pushforward is $\widetilde T_\#(\widetilde q^2\mu_0)=\widetilde\mu_1$, again by Lemma~\ref{lem:Monge-restriction-identity}.

We next compare the cone action of this replacement with the original optimal branch.  The calculation shows that the replacement cannot increase the metric cost.

The two pairs agree outside $A$, so only the contributions on $A$ have to be compared.  The new one is
\[
\int_A\bigl(1+\alpha-2\sqrt\alpha\bigr)d\mu_0=m_0+m_1-2\sqrt{m_0m_1},
\]
since $\alpha m_0=m_1$ and $\sqrt\alpha\,m_0=\sqrt{m_0m_1}$, while the original one is
\[
m_0+m_1-2\int_Aq\cos\left({\sf d}_\Omega(x,T(x))\wedge\frac\pi2\right)\,d\mu_0 .
\]
By Cauchy--Schwarz,
\(
\int_Aq\cos\left({\sf d}_\Omega(x,T(x))\wedge\frac\pi2\right)\,d\mu_0
\le\int_Aq\,d\mu_0\le\mu_0(A)^{1/2}\bigl(\int_Aq^2d\mu_0\bigr)^{1/2}=\sqrt{m_0m_1},
\)
so $\mathcal A_{\mu_0}(\widetilde T,\widetilde q)\le\mathcal A_{\mu_0}(T,q)$.  Since the
new pair is admissible whereas $(T,q)$ is optimal,
\[
\HK^2(\mu_0,\widetilde\mu_1)
\le\mathcal A_{\mu_0}(\widetilde T,\widetilde q)
\le\mathcal A_{\mu_0}(T,q)
=\HK^2(\mu_0,\mu_1).
\]
Having obtained the endpoint competitor, we pass to the interpolated measures used in the variation argument.  Convexity of the squared $\HK$ distance is exactly what is needed here.

Finally, $\nu\mapsto\HK^2(\mu_0,\nu)$ is convex by the entropy--transport formulation, so
\[
\HK^2(\mu_0,\mu_1^s)
\le(1-s)\HK^2(\mu_0,\mu_1)+s\HK^2(\mu_0,\widetilde\mu_1)
\le\HK^2(\mu_0,\mu_1).
\]
\end{proof}

\begin{lemma}[Monotonicity of the discrete chemical potential along the optimal map]\label{lem:monotonicity-full}
Under the assumptions of this subsection,
\begin{equation}\label{eq:chemical-potential-monotonicity}
\Phi_1(Tx)\le\Phi_1(x)
\qquad\text{for }\mu_0\text{-almost every }x.
\end{equation}
\end{lemma}
\begin{proof}
We argue by contradiction.  We first localize a strict violation of the desired monotonicity to a set on which the density and chemical potential are uniformly controlled.

Suppose the conclusion fails, so that $S:=\{x\in G:\Phi_1(Tx)>\Phi_1(x)\}$ has positive
$\mu_0$-measure.  For rationals $a<b$ put
$S_{a,b}=\{x\in G:\Phi_1(x)<a<b<\Phi_1(Tx)\}$.  Then $S=\bigcup_{a<b\text{ rational}}S_{a,b}$
is a countable union, so $\mu_0(S_{a,b})>0$ for at least one such pair; fix it.  Intersecting
$S_{a,b}$ with the sets
\begin{equation}\label{eq:upper--lower--restriction}
\{n^{-1}\le q\le n\},\qquad
\{\rho_1\ge n^{-1}\}\cap T^{-1}\{\rho_1\ge n^{-1}\},\qquad
\{|\Phi_1|\le n\}\cap T^{-1}\{|\Phi_1|\le n\}    
\end{equation}
and letting $n\to\infty$, we may in addition assume that on the resulting set $A\subset G$ the
factor $q$ is bounded above and away from zero, and that $\rho_0$, $\rho_1$ and $\Phi_1$ are
bounded on $A\cup T(A)$ with $\rho_1$ bounded away from zero there.  This is possible because
$\rho_0$ is bounded and strictly positive, $\rho_1>0$ almost everywhere, and $q>0$ for
$\mu_0$-almost every point; the sets $T(A)$ are Borel by Lemma~\ref{lem:Monge-restriction-identity}.

We evaluate the variation of the full energy in the direction given by Lemma~\ref{lem:cut-and-glue-full}. The strict ordering of $\Phi_1$ on $A$ and $T(A)$ makes it negative.

Let \(\eta=\alpha\mu_0|_A-\mu_1|_{T(A)}\) be the direction in Lemma~\ref{lem:cut-and-glue-full}, that is, $\eta=h\,dx$ with
$h=\alpha\rho_0\mathbbm1_A-\rho_1\mathbbm1_{T(A)}$.  Then $h\in L^\infty$, it is supported in
$E=A\cup T(A)$, on which $\rho_1$ is bounded above and away from zero, and
$h_-\le\rho_1\mathbbm1_{T(A)}\le\rho_1$ pointwise, so
Lemma~\ref{lem:localized-first-variation} applies with $\rho=\rho_1$ and yields
\[
\mathcal F(\mu_1+s\eta)-\mathcal F(\mu_1)
=s\int_\Omega\Phi_1\,d\eta+o(s).
\]
Note that  $\Phi_1<a$ on $A$ and  $\Phi_1>b$ on $T(A)$, the latter because every
$y\in T(A)$ is $y=Tx$ for some $x\in A$.  Hence, using
$\alpha\mu_0(A)=m_1=\mu_1(T(A))$ from Lemma~\ref{lem:Monge-restriction-identity} and the
definition of $\alpha$,
\[
\int_\Omega\Phi_1\,d\eta
=\alpha\int_A\Phi_1\,d\mu_0-\int_{T(A)}\Phi_1\,d\mu_1
\le a\alpha\mu_0(A)-b\mu_1(T(A))
=(a-b)\mu_1(T(A))<0.
\]
This gives the required contradiction: the energy decreases to first order, whereas the proximal term does not increase.

The energy therefore decreases for small \(s>0\), while Lemma~\ref{lem:cut-and-glue-full} shows that the $\HK$ term does not increase.  This contradicts the minimality of \(\mu_1\).
\end{proof}

Replacing $G$ by its intersection with the full-measure Borel set on which \eqref{eq:chemical-potential-monotonicity} holds, we henceforth assume that $T$ is injective on $G$ and
\[
\Phi_1(Tx)\le\Phi_1(x)
\qquad\text{for every }x\in G.
\]
Recall the standing bounds
\[
0<c_{\min}\le \rho_0\le c_{\max}<\infty.
\]

\begin{proposition}[Upper density bound]\label{prop:density-upper-full}
Let \(c_{\rm upp}>0\) satisfy
\[
1+2\tau\min\{U'(c_{\rm upp})+\essinf_\Omega P_1,0\}>0,
\]
and set
\begin{equation}\label{eq:upper-threshold-constant}
\Lambda=\frac1{1+2\tau\min\{U'(c_{\rm upp})+\essinf_\Omega P_1,0\}}\ \ (\ge1).
\end{equation}
Then
\begin{equation}\label{eq:upper-density-bound-final}
\rho_1(x)
\le
(U')^{-1}\!\left(
\max\{U'(c_{\rm upp}),U'(\Lambda^2c_{\max})\}
+\essosc_\Omega P_1
\right)
\end{equation}
for almost every \(x\in\Omega\).
\end{proposition}
\begin{proof}
We prove the upper bound by excluding a sufficiently high superlevel set of the chemical potential.  The monotonicity lemma makes this set backward-invariant under the optimal map, and the rest of the argument turns that invariance into a contradiction with minimality.

Fix \(\varepsilon>0\), put \(\Lambda_\varepsilon=\Lambda+\varepsilon\), and define
\[
\ell_\varepsilon
=\max\left\{
U'((1+\varepsilon)^2c_{\rm upp}),
U'(\Lambda_\varepsilon^2c_{\max})
\right\},
\qquad
B=\{x:\Phi_1(x)\ge \ell_\varepsilon+\esssup_\Omega P_1\}.
\]
On \(B\) one has $U'(\rho_1)=\Phi_1-P_1\ge\ell_\varepsilon$, hence, $U'$ being increasing,
\begin{equation}\label{eq:upper-B-density}
\rho_1\ge(1+\varepsilon)^2c_{\rm upp}
\qquad\text{and}\qquad
\rho_1\ge \Lambda_\varepsilon^2c_{\max}
\qquad\text{on }B .
\end{equation}
Moreover, \eqref{eq:chemical-potential-monotonicity} gives \(T^{-1}(B)\cap G\subset B\) because if
$x\in G$ and $Tx\in B$, then $\Phi_1(x)\ge\Phi_1(Tx)\ge\ell_\varepsilon+\esssup_\Omega P_1$. We write  $B':=T^{-1}(B)\cap G$. 

Suppose that this high-potential set carries positive target mass, that is, assume \(\mu_1(B)>0\).  Since $G$ has full $\mu_0$-measure,
$T_\#(q^2\mu_0)=\mu_1$ gives $\int_{B'}q^2d\mu_0=\mu_1(B)$.  Combining this with
\eqref{eq:upper-B-density}, with $B'\subset B$, and with $\rho_0\le c_{\max}$,
\begin{equation}\label{eq:upper-growth-mass}
\int_{B'}q^2\,d\mu_0
=\mu_1(B)=\int_B \rho_1 dx
\ge \Lambda_\varepsilon^2c_{\max}|B|
\ge \Lambda_\varepsilon^2c_{\max}|B'|
\ge \Lambda_\varepsilon^2\mu_0(B').
\end{equation}
In particular $\mu_0(B')>0$; If not, $\mu_1(B)=0$ due to the first equality in  \eqref{eq:upper-growth-mass}.

If $q<\Lambda_\varepsilon$ held $\mu_0$-almost everywhere on
$B'$, then $\int_{B'}q^2d\mu_0<\Lambda_\varepsilon^2\mu_0(B')$, contradicting
\eqref{eq:upper-growth-mass}; hence $A_0:=B'\cap\{q\ge\Lambda_\varepsilon\}$ has positive
$\mu_0$-measure.  Truncating as in the proof of Lemma~\ref{lem:monotonicity-full} (see \eqref{eq:upper--lower--restriction}), choose
$A\subset A_0$ of positive $\mu_0$-measure on which, in addition, $\rho_1$ and $\Phi_1$ are
bounded on $T(A)$; the lower bound \eqref{eq:upper-B-density} keeps $\rho_1$ away from zero
there.  Put $\mu_1^*=\mu_1|_{T(A)}$ and, for $0<t<1$,
\[
\mu_1^t=\mu_1-(1-t^2)\mu_1^*.
\]
We will show that, for \(t<1\) sufficiently close to \(1\), this competitor has strictly smaller JKO objective than \(\mu_1\), contradicting the minimality of \(\mu_1\).

To exploit this large-growth region, we scale the growth factor down on $A$ while leaving the map unchanged.  We now quantify the resulting decrease in the cone action, and hence in the squared $\HK$ term: 
Let $(T_t,q_t)$ agree with $(T,q)$ outside $A$ and equal $(T,tq)$ on $A$.  By
Lemma~\ref{lem:Monge-restriction-identity}, $(T_t)_\#(q_t^2\mu_0)=\mu_1-(1-t^2)\mu_1^*=\mu_1^t$,
so $(T_t,q_t)$ is admissible for $\mu_1^t$ and only its contribution on $A$ differs from
that of $(T,q)$.  Writing $r(x)={\sf d}_\Omega(x,T(x))\wedge\frac\pi2$, that difference is
\[
\int_A\bigl[(1+q^2-2q\cos r)-(1+t^2q^2-2tq\cos r)\bigr]d\mu_0
=(1-t)\int_A\bigl[(1+t)q^2-2q\cos r\bigr]d\mu_0 .
\]
Here $\int_Aq^2d\mu_0=\mu_1(T(A))=\mu_1^*(\Omega)$ by
Lemma~\ref{lem:Monge-restriction-identity}, while $\cos r\le1$ and $q\ge\Lambda_\varepsilon$
on $A$ give $\int_Aq\cos r\,d\mu_0\le\int_Aq\,d\mu_0\le\Lambda_\varepsilon^{-1}\int_Aq^2d\mu_0$.
Since $(T,q)$ is optimal and $(T_t,q_t)$ is merely admissible,
\eqref{eq:injective-growth-representation} and \eqref{eq:incremental-HK-upper-bound} yield
\begin{equation}\label{eq:upper-HK-gain}
\HK^2(\mu_0,\mu_1)-\HK^2(\mu_0,\mu_1^t)
\ge\mathcal A_{\mu_0}(T,q)-\mathcal A_{\mu_0}(T_t,q_t)
\ge(1-t)\left(1+t-\frac2{\Lambda_\varepsilon}\right)\mu_1^*(\Omega).
\end{equation}

We next compare the energy change for the same perturbation.  On $T(A)$ the chemical potential has a uniform lower bound, so the localized first variation gives a matching first-order energy change.

Set $a=U'(c_{\rm upp})+\essinf_\Omega P_1$.  Since $A\subset T^{-1}(B)$, one has
$T(A)\subset B$ and therefore, by \eqref{eq:upper-B-density},
$\Phi_1\ge U'((1+\varepsilon)^2c_{\rm upp})+\essinf_\Omega P_1\ge a$ on $T(A)$.  With
$s=1-t^2$ and $h=-\rho_1\mathbbm1_{T(A)}$, which satisfies $h_-\le\rho_1$,
Lemma~\ref{lem:localized-first-variation} gives
\begin{equation}\label{eq:upper-energy-gain}
\mathcal F(\mu_1)-\mathcal F(\mu_1^t)
=s\int_{T(A)}\Phi_1\,d\mu_1+o(s)
\ge(1-t)(1+t)a\,\mu_1^*(\Omega)+o(1-t).
\end{equation}
Combining \eqref{eq:upper-HK-gain} and \eqref{eq:upper-energy-gain}, dividing by $1-t$, we obtain
\[
\liminf_{t \uparrow 1}\frac{\mathcal J_\tau(\mu_1\mid\mu_0)-\mathcal J_\tau(\mu_1^t\mid\mu_0)}{1-t}
\ge
\left[
\frac1{2\tau}\left(2-\frac2{\Lambda_\varepsilon}\right)+2a
\right]\mu_1^*(\Omega)
=
\left[
\frac1{\tau}\left(1-\frac1{\Lambda_\varepsilon}\right)+2a
\right]\mu_1^*(\Omega).
\]
The right-hand side is positive;  When $a\ge0$ this is immediate, since $\Lambda=1$ and
$\Lambda_\varepsilon>1$.  When $a<0$ one has $\Lambda=(1+2\tau a)^{-1}$, so that
\[
\frac1\tau\left(1-\frac1{\Lambda}\right)+2a
=\frac1\tau\bigl(1-(1+2\tau a)\bigr)+2a
=-2a+2a=0,
\]
and the expression at $\Lambda_\varepsilon>\Lambda$ is strictly larger because
$\Lambda\mapsto1-\Lambda^{-1}$ is increasing.  
Hence
$$ \mathcal J_\tau(\mu_1^t\mid\mu_0) < \mathcal J_\tau(\mu_1\mid\mu_0) $$
for \(t<1\) sufficiently close to \(1\), contradicting the minimality of \(\mu_1\).
Therefore $\mu_1(B)=0$, and since $\rho_1$ is bounded below by a positive constant on $B$ by \eqref{eq:upper-B-density}, also $|B|=0$.

The contradiction shows that the high-potential set is negligible.  It remains only to translate this chemical-potential statement back into a pointwise upper bound for the density.

Let \(\xi_\varepsilon\) be defined by
\[
U'(\xi_\varepsilon)=\ell_\varepsilon+\essosc_\Omega P_1.
\]
If $\rho_1(x)\ge\xi_\varepsilon$, then
$\Phi_1(x)\ge\ell_\varepsilon+\essosc_\Omega P_1+\essinf_\Omega P_1=\ell_\varepsilon+\esssup_\Omega P_1$,
so \(\{\rho_1\ge\xi_\varepsilon\}\subset B\) and hence \(\rho_1<\xi_\varepsilon\) almost everywhere.  Letting \(\varepsilon\downarrow0\) and using the continuity of $U'$ proves \eqref{eq:upper-density-bound-final}.
\end{proof}

\begin{proposition}[Lower density bound]\label{prop:density-lower-full}
Choose \(c_{\rm low}>0\) such that
\[
U'(c_{\rm low})+\essinf_\Omega P_1<0,
\]
which is always possible under \(U'(0+)=-\infty\).  Set \(\bar c=\min\{c_{\min},c_{\rm low}\}\).  Then
\begin{equation}\label{eq:lower-density-bound-final}
\rho_1(x)
\ge
(U')^{-1}\!\left(
U'(\bar c)-\essosc_\Omega P_1
\right)
\end{equation}
for almost every \(x\in\Omega\).
\end{proposition}
\begin{proof}
The lower bound is proved by the analogous low-potential argument, but the competitor is different: increasing the target locally is achieved here by splitting source mass between the original transport branch and a diagonal branch.  We therefore first identify a region where the optimal growth factor is strictly smaller than one, and then build this split competitor explicitly.

Fix \(\varepsilon>0\), set \(c_\varepsilon=\bar c/(1+\varepsilon)\), and define
\[
B=\{x:\Phi_1(x)\le U'(c_\varepsilon)+\essinf_\Omega P_1\}.
\]
By \eqref{eq:chemical-potential-monotonicity}, \(T(B\cap G)\subset B\), and on \(B\) one has \(\rho_1\le c_\varepsilon\).

Assume that the low-potential set has positive measure.  Forward invariance implies that the average of $q^2$ on this set is strictly below one.

If \(|B|>0\), then $\mu_0(B\cap G)>0$ and Lemma~\ref{lem:Monge-restriction-identity} gives
\[
\int_{B\cap G}q^2\,d\mu_0
=\mu_1(T(B\cap G))
\le\mu_1(B)
\le c_\varepsilon|B|
\le\frac1{1+\varepsilon}\mu_0(B\cap G).
\]
Here $\mu_0(\Omega\setminus G)=0$ and the positive lower bound on $\rho_0$ imply
$|B\setminus G|=0$, so the last step uses
$c_\varepsilon|B|=\frac{\bar c}{1+\varepsilon}|B\cap G|\le\frac{c_{\min}}{1+\varepsilon}|B\cap G|\le\frac{1}{1+\varepsilon}\mu_0(B\cap G)$,
by $\bar c\le c_{\min}\le\rho_0$.

We sharpen this average estimate into a pointwise one on a positive-measure subset: there must be a region on which $q$ is uniformly bounded by some $k_\varepsilon<1$.

Choose
\[
k_\varepsilon^2=\frac{1+\varepsilon/2}{1+\varepsilon}<1.
\]
There is a measurable \(A_0\subset B\cap G\) of positive \(\mu_0\)-measure on which \(q\le k_\varepsilon\): otherwise \(q>k_\varepsilon\) \(\mu_0\)-almost everywhere on \(B\cap G\), contradicting the previous display because \(k_\varepsilon^2>(1+\varepsilon)^{-1}\).  Since \(\rho_1>0\) almost everywhere, there is \(\delta>0\) such that \(A:=A_0\cap\{\rho_1\ge\delta\}\) still has positive \(\mu_0\)-measure; on \(A\) one has \(\delta\le\rho_1\le c_\varepsilon\) and \(\rho_0\mathbbm1_A\in L^\infty\) is nonnegative, so Lemma~\ref{lem:localized-first-variation} applies with \(E=A\).

We now construct the competitor. 
For \(t<1\) close to \(1\), set
\[
\mu_1^t=\mu_1+(1-t^2)\mu_0|_A.
\]
To estimate $\HK^2(\mu_0,\mu_1^t)$, we construct an admissible
triple in the cone formulation \eqref{eq:HK-cone}, preserving
the original transported target mass and adding mass on $A$
through a diagonal branch. Admissibility supplies the required
cost upper bound.
Writing \(s_t:=1-t^2\), we reduce the source radial weight of the original transport branch from \(1\) to \(t\), while keeping its target radial weight \(q\) unchanged, and assign the missing source contribution \(s_t\mu_0|_A\) to a diagonal branch with radial weights \((\sqrt{s_t},\sqrt{s_t})\). Thus the original transported target mass is preserved, while the diagonal branch adds exactly \(s_t\mu_0|_A\) to the target. On the fixed-point set \(T(x)=x\), however, these two branches coincide and their radial weights must be merged.
Accordingly, set
\[
A_{\rm fix}:=A\cap\{x:T(x)=x\},
\qquad A_*:=A\setminus A_{\rm fix},
\]
and take
\[
\gamma
=(\mathrm{id},T)_\#\mu_0|_{A_*}
+(\mathrm{id},\mathrm{id})_\#\mu_0|_A
+(\mathrm{id},T)_\#\mu_0|_{\Omega\setminus A}.
\]
These three pieces are mutually singular.  Define $(r_0,r_1)$ by
\[
(r_0,r_1)=
\begin{cases}
(t,q),&\text{on }(\mathrm{id},T)(A_*),\\
(\sqrt{s_t},\sqrt{s_t}),&\text{on }(\mathrm{id},\mathrm{id})(A_*),\\
(1,\sqrt{q^2+s_t}),&\text{on }(\mathrm{id},\mathrm{id})(A_{\rm fix}),\\
(1,q),&\text{on }(\mathrm{id},T)(\Omega\setminus A).
\end{cases}
\]
Since $T=\mathrm{id}$ on $A_{\rm fix}$ and $T_\#(q^2\mu_0)=\mu_1$, the marginals are
\[
\pi^0_\#(r_0^2\gamma)
=t^2\mu_0|_{A_*}+s_t\mu_0|_{A_*}+\mu_0|_{A_{\rm fix}}
 +\mu_0|_{\Omega\setminus A}=\mu_0,
\]
\[
\begin{aligned}
\pi^1_\#(r_1^2\gamma)
&=T_\#(q^2\mu_0|_{A_*})+s_t\mu_0|_{A_*}
 +(q^2+s_t)\mu_0|_{A_{\rm fix}}
 +T_\#(q^2\mu_0|_{\Omega\setminus A})\\
&=\mu_1+s_t\mu_0|_A=\mu_1^t,
\end{aligned}
\]
so the triple is admissible.

We next estimate the metric effect of this split construction.  On both the moving and fixed-point parts of $A$, the new cone cost is strictly smaller to first order because $q<1$.

On $A_*$ its excess cost over the original optimal pair is the
negative of
\[
(1-t)\bigl(1+t-2q\cos r\bigr),
\qquad r(x):={\sf d}_\Omega(x,T(x))\wedge\frac\pi2.
\]
On $A_{\rm fix}$ the merged diagonal branch has cost
$\bigl(1-\sqrt{q^2+s_t}\bigr)^2$.  Since
$1=\sqrt{t^2+s_t}$ and $a\mapsto\sqrt{a^2+s_t}$ is $1$-Lipschitz on $[0,\infty)$,
\[
\bigl(1-\sqrt{q^2+s_t}\bigr)^2\le(t-q)^2.
\]
Thus on $A_{\rm fix}$ the decrease relative to the original cost $(1-q)^2$ is at least
$(1-t)(1+t-2q)$.  Using $q\le k_\varepsilon$ on $A$ and $\cos r\le1$ on $A_*$, the total
cost decrease is therefore at least
\[
(1-t)(1+t-2k_\varepsilon)\mu_0(A).
\]
Since the triple above is merely admissible while $(T,q)$ is optimal, this gives
\begin{equation*}
\HK^2(\mu_0,\mu_1)-\HK^2(\mu_0,\mu_1^t)
\ge(1-t)(1+t-2k_\varepsilon)\mu_0(A)>0
\end{equation*}
for \(t\) sufficiently close to \(1\), because $k_\varepsilon<1$.
It remains to compare the energy.  Since $A$ lies in the low-potential set, adding mass there also decreases the energy to first order, which contradicts minimality.

  With \(s=1-t^2\), the localized first variation gives
\[
\mathcal F(\mu_1^t)-\mathcal F(\mu_1)
=s\int_A\Phi_1\,d\mu_0+o(s).
\]
Since \(A\subset B\) and
\(
U'(c_\varepsilon)+\essinf_\Omega P_1<U'(c_{\rm low})+\essinf_\Omega P_1<0,
\)
the energy also decreases for \(t\) close to \(1\).  This contradicts minimality, so \(|B|=0\).

Thus the low-potential set is negligible.  We finally convert this information into the asserted lower density bound and then let $\varepsilon\downarrow0$.

Define \(\xi_\varepsilon\) by
\[
U'(\xi_\varepsilon)
=U'(c_\varepsilon)-\essosc_\Omega P_1.
\]
If $\rho_1(x)\le\xi_\varepsilon$, then
$\Phi_1(x)\le U'(c_\varepsilon)-\essosc_\Omega P_1+\esssup_\Omega P_1=U'(c_\varepsilon)+\essinf_\Omega P_1$,
so \(\{\rho_1\le\xi_\varepsilon\}\subset B\) and hence \(\rho_1>\xi_\varepsilon\) almost everywhere.  Let \(\varepsilon\downarrow0\), so that $c_\varepsilon\uparrow\bar c$, to obtain \eqref{eq:lower-density-bound-final}.
\end{proof}

\subsection{Dual potential and discrete Euler equation}\label{sec:dual-euler}

Here $(T,q)$ is always oriented from the first argument of $\HK$ to the second.  In the
application to Theorem~\ref{thm:full-step-regularity} the ordered pair is $(\mu_1,\mu_0)$,
so this $T$ is the reversed map, from the new state to the source.

\paragraph{Transformed dual potentials.}
The logarithmic dual variables of \eqref{eq:HK-dual} are convenient to rescale.  Writing
$K(x,y):=\cos^2({\sf d}_\Omega(x,y)\wedge\frac\pi2)=e^{-c_{\HK}(x,y)}$ and substituting
\begin{equation}\label{eq:log-to-transformed}
\psi:=1-e^{-z_0},
\qquad
\chi:=1-e^{-z_1},
\end{equation}
into \eqref{eq:HK-dual} turns the objective into a linear one, because
$-F^*(-z)=1-e^{-z}$ for $F^*(a)=e^a-1$, and turns the constraint $z_0+z_1\le c_{\HK}$ into
$e^{-z_0}e^{-z_1}\ge e^{-c_{\HK}}$.  Hence \eqref{eq:HK-dual} is equivalent to
\begin{equation}\label{eq:transformed-HK-dual}
\HK^2(\mu,\nu)
=
\sup\left\{\int_\Omega \psi\,d\mu+\int_\Omega \chi\,d\nu\right\},
\end{equation}
the supremum being over Borel $\psi,\chi\le1$ subject to
\begin{equation}\label{eq:transformed-HK-constraint}
(1-\psi(x))(1-\chi(y))
\ge K(x,y)
=\cos^2\!\left({\sf d}_\Omega(x,y)\wedge\frac\pi2\right)
\qquad\text{for all }x,y\in\Omega.
\end{equation}
We call $(\psi,\chi)$ a \emph{transformed} dual pair; $\psi$ is its \emph{source component}
and $\chi$ its \emph{target component}.  The pair is \emph{saturated} when each component is
the envelope of the other with respect to \eqref{eq:transformed-HK-constraint}, that is,
\begin{equation}\label{eq:tight-source-representative}
\psi(x)=1-\sup_{y\in\Omega}\frac{K(x,y)}{1-\chi(y)},
\qquad
\chi(y)=1-\sup_{x\in\Omega}\frac{K(x,y)}{1-\psi(x)} .
\end{equation}
Saturation fixes a canonical representative: replacing a feasible pair by the corresponding
envelopes preserves feasibility and does not decrease the objective.  Following \cite{LMS23},
we call the source component $\psi$ of this envelope-normalized pair the \emph{tight source
potential}; it is the component differentiated below.  
% The symbols $\psi,\chi$ avoid
% conflict with the density notation $u$.

\begin{lemma}[Regularity and uniqueness of transformed $\HK$ dual potentials]\label{lem:transformed-potential-Lipschitz}
Let $\Omega$ be either a compact convex subset of $\mathbb R^d$ with nonempty interior or the flat torus $\mathbb T^d$, and let
$\mu=\rho\,dx$, $\nu=\tilde\rho\,dx$ satisfy
\[
0<m\le\rho,\tilde\rho\le M<\infty\qquad\text{a.e.}
\]
There is a saturated optimal pair $(\psi,\chi)$ in \eqref{eq:transformed-HK-dual} with
$\psi,\chi\in W^{1,\infty}(\Omega)$.  Its saturated source component $\psi$ is unique $\mu$-almost everywhere.  Moreover, for some $\delta>0$,
\begin{equation}\label{eq:dual-away-from-one}
1-\psi\ge\delta,\qquad 1-\chi\ge\delta
\qquad\text{on }\Omega,
\end{equation}
and the associated Monge map stays uniformly below the cutoff $\pi/2$.
Moreover, $\delta$ and the $W^{1,\infty}$ norms of $\psi$ and $\chi$ may be chosen to depend
only on $m$, $M$, and $\Omega$.
\end{lemma}
\begin{proof}
We first establish existence, uniqueness, and Lipschitz regularity of the transformed dual potentials.  The argument is quoted separately in the Euclidean and torus settings because the available regularity results are formulated differently.

In the Euclidean case, both measures have compact full support, so each support lies in the
$\pi/2$-neighborhood of the other and the pair is strongly reduced in the sense of
\cite[Definition~2.6]{LMS23}.  The uniqueness of the optimal marginal densities from
\cite[Theorem~2.14]{LMS23}, together with the tight-potential regularity and Monge results in
\cite[Theorems~3.3 and~3.6 and Corollary~3.5]{LMS23}, yields a saturated Lipschitz
transformed pair and the associated Monge map.  In the present normalization the optimal
marginal densities are
\(1-\psi=\sigma_0\) and \(1-\chi=\sigma_1,\)
so their uniqueness gives uniqueness of the saturated source component $\mu$-almost everywhere.
Moreover, the factors $1-\psi$ and $1-\chi$ are positive and continuous on the compact full
supports and hence are bounded away from zero.

We now treat the torus using the logarithmic potentials and then transform back to the variables
$(\psi,\chi)$.  On the flat torus both measures have full support.  In particular they form an
admissible pair for the $\HK$ cost, since for each $x$ one may choose $y=x$, for which
$c_{\HK}(x,x)=0$.  The Lipschitz-potential and Monge results of
\cite[Sections~2.2 and~3.4.1]{GGV25}
yield a unique Lipschitz logarithmic optimal pair $(z_0,z_1)$ and the associated Monge map.
The change of variables \eqref{eq:log-to-transformed} gives the corresponding saturated pair in
\eqref{eq:transformed-HK-dual}; uniqueness of $z_0$ gives uniqueness of the source component.
Since $z_0,z_1$ are finite and Lipschitz on the compact torus,
$e^{-z_0}$ and $e^{-z_1}$ have positive lower bounds.

Thus, in either geometry, after decreasing $\delta>0$ if necessary,
\eqref{eq:dual-away-from-one} holds.  It remains to verify that the associated Monge map stays
uniformly below the cutoff.  On the contact set,
\[
(1-\psi(x))(1-\chi(T(x)))
=\cos^2\!\left({\sf d}_\Omega(x,T(x))\wedge\frac\pi2\right).
\]
Hence \eqref{eq:dual-away-from-one} gives
\[
\cos^2\!\left({\sf d}_\Omega(x,T(x))\wedge\frac\pi2\right)\ge\delta^2,
\]
and therefore
\[
{\sf d}_\Omega(x,T(x))
\le \arccos(\delta)<\frac\pi2
\]
for $\mu$-almost every $x$ (after replacing $\delta$ by $\min\{\delta,1\}$ if necessary).

For the later Euler argument we need more than pairwise regularity: the bounds must be uniform over all density pairs satisfying the same two-sided estimates.  We obtain this by compactness and the saturated-envelope representation.

To make this quantitative, consider all density pairs with the same $m$ and $M$.  Set
\[
a:=1-\psi \ge 0,\qquad b:=1-\chi \ge0.
\]
The preceding argument gives positive continuous $a,b$ on the compact space $\Omega$.
Since $0\le K\le1$, $K(x,x)=1$, and the pair is saturated,
\eqref{eq:tight-source-representative} implies
\[
\max_\Omega a=\frac1{\min_\Omega b},
\qquad
\max_\Omega b=\frac1{\min_\Omega a}.
\]
Let $a_{\max}:=\max_\Omega a$.  The saturation formula and $\Lip(K)\le1$ give
\[
|a(x)-a(x')|
\le \frac{\Lip(K)}{\min_\Omega b}\,{\sf d}_\Omega(x,x')
\le a_{\max}\,{\sf d}_\Omega(x,x').
\]
If $x_0$ is a maximizer of $a$, then $a\ge a_{\max}/2$ on
$B_{1/2}(x_0)\cap\Omega$.  Since $\Omega$ is fixed, compact, and is either convex with
nonempty interior or a flat torus,
\[
v_\Omega:=\inf_{x\in\Omega}|B_{1/2}(x)\cap\Omega|>0,
\]
where the balls are Riemannian in the torus case.  The feasible pair $(0,0)$ in
\eqref{eq:transformed-HK-dual} shows that the optimal value is nonnegative, and hence
\[
\int_\Omega a\,d\mu+\int_\Omega b\,d\nu
=\mu(\Omega)+\nu(\Omega)-\HK^2(\mu,\nu)
\le2M|\Omega|.
\]
Using $\rho\ge m$ therefore yields
\[
\frac{m a_{\max}v_\Omega}{2}
\le\int_\Omega a\,d\mu
\le2M|\Omega|,
\qquad
a_{\max}\le\frac{4M|\Omega|}{m v_\Omega}.
\]
The same argument for $b$ gives the same bound for $\max b$.  Consequently
\[
a,b\ge\delta_0:=\frac{m v_\Omega}{4M|\Omega|}>0,
\]
uniformly over the stated class of densities.  The saturation formulas then also give
$a,b\le\delta_0^{-1}$ and
\[
\Lip(\psi),\Lip(\chi)\le\frac{\Lip(K)}{\delta_0},
\]
where we used the elementary inequality
\(\left|\sup_y f(y)-\sup_y g(y)\right|
\le \sup_y |f(y)-g(y)|.\)
So the full $W^{1,\infty}$ bounds depend only on $m$, $M$, and $\Omega$.  This is the
uniformity used in Lemma~\ref{lem:jko-euler-source-potential-full} for the family of perturbed
densities.
\end{proof}

\begin{corollary}[Identities for an optimal Monge-growth pair]\label{cor:HK-growth-couple-relations-full}
Let \((T,q)\) be an essentially injective optimal Monge-growth pair from \(\mu\) to \(\nu\), so that \(T_\#(q^2\mu)=\nu\), and let $T$ also be the Monge map of Lemma~\ref{lem:transformed-potential-Lipschitz}.  The two coincide by uniqueness of the entropy--transport optimizer in the two cases covered by that lemma.  For the saturated representatives in Lemma~\ref{lem:transformed-potential-Lipschitz},
\begin{equation}\label{eq:dual-growth-relations}
1-\psi(x)=q(x)\cos r(x),
\qquad
1-\chi(T(x))=\frac{\cos r(x)}{q(x)},
\qquad
r(x):={\sf d}_\Omega(x,T(x)),
\end{equation}
for \(\mu\)-a.e. \(x\) on the transported part.  In particular, for some \(\varepsilon>0\),
\[
r(x)\le\frac\pi2-\varepsilon
\qquad\text{for \(\mu\)-a.e. }x.
\]
\end{corollary}
\begin{proof}
We identify the density of the optimal entropy--transport plan relative to the source in two different ways.  The primal minimization gives it in terms of $(T,q)$, while dual optimality gives it in terms of the source potential; equating the two descriptions yields the claimed identities.

The entropy--transport optimizer is concentrated on the graph of $T$, so it has the form
\(\gamma=(\mathrm{id},T)_\#(h\mu)\) for a Borel $h\ge0$, and it remains to identify $h$.  Its
first marginal is $\gamma_0=h\mu$, so $d\gamma_0/d\mu=h$.  Its second marginal is
$\gamma_1=T_\#(h\mu)$, while $\nu=T_\#(q^2\mu)$; since $T$ is essentially injective, the two
densities compare at the transported point, namely
\[
\frac{d\gamma_1}{d\nu}\circ T=\frac h{q^2}
\qquad\mu\text{-a.e.},
\qquad\text{hence}\qquad
\int_\Omega F\!\left(\frac{d\gamma_1}{d\nu}\right)d\nu
=\int_\Omega F\!\left(\frac h{q^2}\right)q^2\,d\mu .
\]
Substituting this and $\int c_{\HK}\,d\gamma=\int c_{\HK}(x,T(x))h\,d\mu$ into
\eqref{eq:HK-entropy-transport}, the entropy--transport cost on the right-hand side of \eqref{eq:HK-entropy-transport} becomes $\int_\Omega\Theta\,d\mu$
with pointwise integrand
\[
\Theta=F(h)+q^2F\!\left(\frac h{q^2}\right)+c_{\HK}(x,T(x))\,h .
\]
Being optimal among all plans, $\gamma$ is in particular optimal within the class of plans of
this form, in which $h$ is an unconstrained nonnegative Borel function.  Since
$F'(\sigma)=\log\sigma$, minimizing $\Theta$ in $h>0$ pointwise therefore gives
\[
\log h+\log\frac h{q^2}+c_{\HK}(x,T(x))=0,
\qquad\text{that is}\qquad
\frac{h^2}{q^2}=e^{-c_{\HK}(x,T(x))}=\cos^2r,
\]
so \(h=q\cos r\).

Let $(z_0,z_1)$ be the optimal logarithmic dual pair corresponding
to $(\psi,\chi)$ through \eqref{eq:log-to-transformed}.
We apply Fenchel's inequality
\[
F(s)+F^*(-z)+zs\ge0,
\qquad
\text{with equality if and only if }s=e^{-z},
\]
to $(s,z)=(h,z_0)$ and
$(s,z)=(d\gamma_1/d\nu,z_1)$.
Using the identities
$\int z_0h\,d\mu=\int z_0(x)\,d\gamma(x,y)$ and
$\int_\Omega z_1(y)\,\frac{d\gamma_1}{d\nu}(y)\,d\nu(y)=\int z_1(y)\,d\gamma(x,y)$,
the entropy--transport objective \eqref{eq:HK-entropy-transport}
minus the dual objective \eqref{eq:HK-dual} equals the sum of
the corresponding integrals against $\mu$ and $\nu$, together with
\[
\int_{\Omega\times\Omega}
\bigl(c_{\HK}(x,y)-z_0(x)-z_1(y)\bigr)\,d\gamma(x,y).
\]
All three terms are nonnegative, and their sum is zero at optimality.
Hence equality holds in both Fenchel inequalities, giving
\[
h=e^{-z_0}=1-\psi,
\qquad
\frac{d\gamma_1}{d\nu}=e^{-z_1}=1-\chi,
\]
$\mu$- and $\nu$-almost everywhere, respectively, by
\eqref{eq:log-to-transformed}.
Combining these identities with $h=q\cos r$ and the previously
established relation $(d\gamma_1/d\nu)\circ T=h/q^2$ yields
\[
1-\psi=q\cos r,
\qquad
1-\chi(T(x))
=\frac{h(x)}{q(x)^2}
=\frac{\cos r(x)}{q(x)}
\quad\text{for $\mu$-a.e. }x.
\]
Finally, the lower bounds in \eqref{eq:dual-away-from-one} give
\[
\cos^2r(x)
=(1-\psi(x))(1-\chi(T(x)))
\ge\delta^2
\]
for some $0<\delta<1$.
Thus $r(x)\le\arccos\delta<\pi/2$ for $\mu$-a.e. $x$,
which proves the uniform bound below the cutoff.
\end{proof}

\begin{corollary}[Discrete $\HK$ identity for the source-side dual potential]\label{cor:HK-energy-identity-full}
Assume the hypotheses of Lemma~\ref{lem:transformed-potential-Lipschitz}.  Let $\psi$ be its unique saturated source component, let $(T,q)$ be the associated Monge pair, and set $\phi=\psi/2$.  Then
\begin{equation}\label{eq:HK-source-energy-identity}
\int_\Omega\left(|\nabla\phi|^2+4|\phi|^2\right)d\mu
=\HK^2(\mu,\nu).
\end{equation}
\end{corollary}
\begin{proof}
We prove the identity by first expressing the gradient of the source dual potential in terms of the Monge pair.  The derivation is slightly different in Euclidean space and on the torus, but both give the same pointwise formula.

Let \(r={\sf d}_\Omega(x,T(x))<\pi/2\), and let \(e(x)\) be the initial unit tangent \emph{at $x$ of the minimizing geodesic running from $x$ to $T(x)$}, with \(e=0\) when \(r=0\); in the Euclidean case $e(x)=(T(x)-x)/|T(x)-x|$.  Fix a point at which the Lipschitz function $\psi$ is differentiable.

At an interior differentiability point $x$, fix $y=T(x)$.
By \eqref{eq:tight-source-representative}, the function
\[
x'\longmapsto 1-\frac{K(x',y)}{1-\chi(y)}-\psi(x')
\]
is nonnegative and vanishes at $x'=x$.
Its gradient therefore vanishes there, giving
\[
\nabla \psi(x)
=-\frac{\nabla_xK(x,T(x))}{1-\chi(T(x))} .
\]
For $r>0$, we have $K(x,y)=\cos^2 r$ and
$\nabla_x{\sf d}_\Omega(x,y)=-e(x)$ at $y=T(x)$.
The chain rule therefore gives
\[
\nabla_xK(x,T(x))
=-2\cos r\sin r\,\nabla_x{\sf d}_\Omega(x,T(x))
=2\cos r\,\sin r\,e(x).
\]
The resulting formula also holds when $r=0$, since
$K(\cdot,y)$ is differentiable at $x=y$ with zero gradient.
Inserting this together with $1-\chi(T(x))=\cos r/q$ from \eqref{eq:dual-growth-relations},
\[
\nabla\psi(x)
=-\frac{2\cos r\,\sin r\;e(x)}{\cos r/q}
=-2q(x)\sin r(x)\,e(x).
\]

For $\Omega=\mathbb T^d$, set $z=-\log(1-\psi)$, which is the logarithmic potential $z_0$ of \eqref{eq:log-to-transformed}.  The Monge formula in \cite[Section~3.4.1]{GGV25} gives, at almost every differentiability point of $z$,
\[
T(x)=\exp_x\!\left(-r(x)\frac{\nabla z(x)}{|\nabla z(x)|}\right),
\quad
r(x)={\sf d}_\Omega(x,T(x))=\arctan\!\left(\frac{|\nabla z(x)|}{2}\right),
\quad
q(x)=\frac{e^{-z(x)}}{\cos r(x)},
\]
with the evident convention when $\nabla z(x)=0$.  Thus $e=-\nabla z/|\nabla z|$ and, from
$|\nabla z|=2\tan r$ and $e^{-z}=q\cos r$,
\[
\nabla \psi=e^{-z}\nabla z=-q\cos r\,|\nabla z|\,e=-2q\sin r\,e.
\]

Hence the same gradient identity holds on both state spaces, with the Riemannian gradient on the torus.  By \eqref{eq:dual-growth-relations}, \(\psi=1-q\cos r\).  Therefore
\[
|\nabla\phi|^2=q^2\sin^2r,
\qquad
4|\phi|^2=(1-q\cos r)^2,
\]
and hence
\[
|\nabla\phi|^2+4|\phi|^2
=1+q^2-2q\cos r.
\]
Integrating and using \eqref{eq:injective-growth-representation} proves \eqref{eq:HK-source-energy-identity}.
\end{proof}

\begin{lemma}[Euler--Lagrange identity for the fully implicit $\HK$-JKO step]\label{lem:jko-euler-source-potential-full}
Let $\mu=\rho\,dx$ minimize
\[
\mathcal F(\sigma)+\frac1{2\tau}\HK^2(\sigma,\nu),
\]
assume Assumption~\ref{ass:standing-entropy}, $V\in C(\Omega)$, \textbf{\emph{(W0)}}, and $\nu=\tilde\rho\,dx$ with $0<m\le\rho,\tilde\rho\le M$.  Let $\phi$ be the source-side dual potential from Corollary~\ref{cor:HK-energy-identity-full}.  Then
\begin{equation}\label{eq:full-Euler-Lagrange}
U'(\rho)+V+W[\rho]=-\frac\phi\tau
\qquad\text{a.e. on }\Omega.
\end{equation}
In particular, the full chemical potential $g_\rho$ has a globally Lipschitz representative.
\end{lemma}
\begin{proof}
We first compute the directional derivative of the squared $\HK$ term with respect to its first argument.  The key point is that the saturated source dual potentials for the perturbed measures converge uniformly to the one at the minimizer.

For $h\in L^\infty(\Omega)$, the positive lower and upper density bounds permit the two-sided perturbation $\rho_\varepsilon=\rho+\varepsilon h$ for sufficiently small $|\varepsilon|$.  Set
\(
G(\sigma)=\HK^2(\sigma,\nu).
\)
There are then $0<m'\le M'<\infty$, independent of $\varepsilon$, with $m'\le\rho_\varepsilon\le M'$ for all small $|\varepsilon|$.  Let $(\psi_\varepsilon,\chi_\varepsilon)$ be saturated optimal transformed dual pairs for $(\rho_\varepsilon dx,\nu)$, and let $(\psi,\chi)$ be one for $(\mu,\nu)$.  For $\varepsilon>0$, using $(\psi,\chi)$ as a competitor at $\rho_\varepsilon dx$ and $(\psi_\varepsilon,\chi_\varepsilon)$ as a competitor at $\mu$ in \eqref{eq:transformed-HK-dual} gives the two envelope bounds
\[
\varepsilon\int_\Omega \psi\, h\,dx
\le
G(\rho_\varepsilon dx)-G(\mu)
\le
\varepsilon\int_\Omega \psi_\varepsilon h\,dx,
\]
because the $\nu$-integrals cancel in each comparison.  
By Lemma~\ref{lem:transformed-potential-Lipschitz}, both components
$(\psi_\varepsilon,\chi_\varepsilon)$ are uniformly bounded in
$W^{1,\infty}(\Omega)$. By Arzel\`a--Ascoli on the compact space
$\Omega$, every sequence $\varepsilon_n\downarrow0$ admits a
subsequence along which the pair converges uniformly to some
$(\psi_*,\chi_*)$.

The envelope bounds above and the uniform sup-norm bounds imply
$G(\rho_{\varepsilon_n}dx)\to G(\mu)$.
Together with $\rho_{\varepsilon_n}\to\rho$ in $L^\infty(\Omega)$
and the uniform convergence of
$(\psi_{\varepsilon_n},\chi_{\varepsilon_n})$,
this allows us to pass to the limit and obtain
\[
G(\mu)
=
\int_\Omega \psi_*\,d\mu+\int_\Omega \chi_*\,d\nu .
\]

The dual constraint also passes to the limit, so
\((\psi_*,\chi_*)\) is feasible for \((\mu,\nu)\); hence the equality
above shows that it is optimal. Finally, the uniform positive lower
bounds on \(1-\psi_{\varepsilon_n}\) and
\(1-\chi_{\varepsilon_n}\), together with uniform convergence, allow
passage to the saturation identities
\eqref{eq:tight-source-representative}.

% The envelope bounds above and the uniform sup-norm bounds imply
% $G(\rho_{\varepsilon_n}dx)\to G(\mu)$. Passing to the limit in the
% dual constraint and the optimal objective therefore shows that
% $(\psi_*,\chi_*)$ is optimal for $(\mu,\nu)$.
% The uniform positive lower bounds for $1-\psi_{\varepsilon_n}$
% and $1-\chi_{\varepsilon_n}$ also allow passage to the saturation
% identities \eqref{eq:tight-source-representative}.

By Lemma~\ref{lem:transformed-potential-Lipschitz},
uniqueness of the saturated optimal source component gives
$\psi_*=\psi$ $\mu$-almost everywhere.
Since both functions are continuous and $\mu$ has full support,
they agree everywhere.
Every subsequential source limit is thus $\psi$, proving
$\psi_\varepsilon\to\psi$ uniformly on $\Omega$
as $\varepsilon\downarrow0$.
  Therefore
\[
\lim_{\varepsilon\downarrow0}
\frac{G(\rho_\varepsilon dx)-G(\mu)}{\varepsilon}
=
\int_\Omega \psi\, h\,dx.
\]
Applying the same argument to $-h$ yields the left derivative, so that
\[
\left.\frac d{d\varepsilon}\right|_{\varepsilon=0}
G(\rho_\varepsilon dx)
=\int_\Omega \psi\,h\,dx
=2\int_\Omega\phi h\,dx.
\]
Therefore
\[
\left.\frac d{d\varepsilon}\right|_{\varepsilon=0}
\frac1{2\tau}\HK^2(\rho_\varepsilon\,dx,\nu)
=\frac1\tau\int_\Omega\phi h\,dx.
\]

With the proximal derivative identified, we now differentiate the full energy.  Minimality for arbitrary bounded perturbations then yields the Euler--Lagrange identity.

The perturbed densities remain in a compact subset of $(0,\infty)$, while $V$ and $W$ are bounded.  Thus differentiation under the integral and symmetry of $W$ give
\[
\left.\frac d{d\varepsilon}\right|_{\varepsilon=0}
\mathcal F(\rho_\varepsilon\,dx)
=\int_\Omega\bigl(U'(\rho)+V+W[\rho]\bigr)h\,dx.
\]
Minimality for all bounded $h$ yields \eqref{eq:full-Euler-Lagrange}.
\end{proof}

\begin{corollary}[Discrete metric--dissipation identity]\label{cor:exact-step-dissipation-full}
Under the hypotheses of Corollary~\ref{cor:HK-energy-identity-full} and Lemma~\ref{lem:jko-euler-source-potential-full},
\begin{equation*}
\HK^2(\rho\,dx,\nu)=\tau^2 I_g(\rho).
\end{equation*}
\end{corollary}
\begin{proof}
By \eqref{eq:full-Euler-Lagrange}, $\phi=-\tau g_\rho$.  Substitute into \eqref{eq:HK-source-energy-identity}.
\end{proof}

\begin{proof}[Proof of Theorem~\ref{thm:full-step-regularity}]
We now assemble the preceding results.  We first obtain the forward Monge statement and the two-sided density bounds; these are the inputs needed to invoke the dual regularity theory for the reversed pair.

Proposition~\ref{prop:HK-Monge-representation} and Lemma~\ref{lem:monotonicity-full} give part (i).  By Theorem~\ref{thm:main-fully-implicit-HK-JKO-step}, $\mu_1(\Omega)\le M_1$, so
\[
\|P_1\|_\infty\le B_1,\qquad \essosc P_1\le\omega_1.
\]
In Proposition~\ref{prop:density-upper-full}, choose $c_{\rm upp}=(U')^{-1}(B_1)$.  Then $U'(c_{\rm upp})+\essinf_\Omega P_1=B_1+\essinf_\Omega P_1\ge0$ because $\essinf_\Omega P_1\ge-B_1$, so the minimum in \eqref{eq:upper-threshold-constant} vanishes and $\Lambda=1$.  With $c_{\max}=C_0$, \eqref{eq:upper-density-bound-final} then reads
$\rho_1\le(U')^{-1}(\max\{B_1,U'(C_0)\}+\essosc_\Omega P_1)$, which is the upper bound in \eqref{eq:canonical-two-sided-density-bound}.  In Proposition~\ref{prop:density-lower-full}, choose $c_{\rm low}=c_*=(U')^{-1}(-B_1-1)$.  Since $\essinf_\Omega P_1\le B_1$, one has $U'(c_*)+\essinf_\Omega P_1\le-1<0$, so $c_*$ is admissible there, and $\bar c=\min\{c_0,c_*\}$ gives the lower bound.

We next work with the reversed ordering $(\mu_1,\mu_0)$.  The two-sided bounds just proved allow us to identify the discrete Euler equation and a Lipschitz representative of the chemical potential.

Apply Lemma~\ref{lem:transformed-potential-Lipschitz}, Corollary~\ref{cor:HK-energy-identity-full}, and Lemma~\ref{lem:jko-euler-source-potential-full} to the ordered pair $(\mu_1,\mu_0)$, whose two densities are two-sided bounded by part~(ii) and by hypothesis, respectively.  Note that the Monge pair used in this second half runs from $\mu_1$ to $\mu_0$ and is therefore not the pair of part~(i).  Equation \eqref{eq:full-Euler-Lagrange} gives \eqref{eq:full-Euler-Lagrange-main}, so that $\Phi_1=-\phi_1/\tau$ has a Lipschitz representative.  Its weak gradient defines $I_g(\rho_1)$ without requiring separate spatial derivatives of $V$ or $W[\rho_1]$.

We combine the elementary one-step energy inequality with the exact metric identity to obtain the dissipation statement.
Since $\mathcal F(\bar\mu)<\infty$, the minimality gives
\begin{equation}\label{eq:one-step-full-energy-dissipation}
\mathcal F(\mu_1)+\frac1{2\tau}\HK^2(\mu_1,\bar\mu)
=\mathcal J_\tau(\mu_1\mid\bar\mu)
\le\mathcal J_\tau(\bar\mu\mid\bar\mu)
=\mathcal F(\bar\mu).    
\end{equation}
 Corollary~\ref{cor:exact-step-dissipation-full} gives the metric identity, and
combining it with \eqref{eq:one-step-full-energy-dissipation} gives the last inequality.

Finally, if $V\in W^{1,\infty}(\Omega)$ and \textbf{\emph{(W1)}} holds, then $P_1$ is Lipschitz, since $\Lip(W[\rho_1])\le\Lip(W)\,\mu_1(\Omega)$.  By Lemma~\ref{lem:transformed-potential-Lipschitz}, $\phi_1$ is Lipschitz, so $U'(\rho_1)=-\phi_1/\tau-P_1$ is Lipschitz.  The continuous positive function $U''$ has a positive minimum on $[\underline\rho_1,\overline\rho_1]$, hence $(U')^{-1}$ is Lipschitz on $U'([\underline\rho_1,\overline\rho_1])$.  Solving \eqref{eq:full-Euler-Lagrange-main} gives $\rho_1\in W^{1,\infty}(\Omega)$ under these additional assumptions.
\end{proof}

\section{Proof of Theorem~\ref{thm:surrogate-one-step-consistency} and Corollary~\ref{cor:recursive-finite-time-tracking} (Theorem~C)}\label{sec:proof-third-theorem}\label{sec:neural-consistency}

We first record the elementary fact about inverses that both proofs of this appendix use.

\begin{lemma}[Uniform convergence of inverses]\label{lem:inverse-uniform-convergence}
Let $\Omega$ be compact, let $S,S_n:\Omega\to\Omega$ be continuous bijections of $\Omega$ onto
itself, and suppose $S_n\to S$ uniformly.  Then $S_n^{-1}\to S^{-1}$ uniformly.
\end{lemma}
\begin{proof}
We prove the claim by contradiction.  A failure of uniform convergence of the inverse maps would produce convergent preimages whose limit violates injectivity of $S$.

Both $S$ and the $S_n$ are homeomorphisms, being continuous bijections of a compact Hausdorff
space. Otherwise, some subsequence satisfies ${\sf d}_\Omega(S_n^{-1}(y_n),S^{-1}(y_n))\ge\delta$ for a fixed $\delta>0$ and points $y_n\in\Omega$.  By compactness we may
assume $y_n\to y$ and $x_n:=S_n^{-1}(y_n)\to x$.  Then
${\sf d}_\Omega(S(x_n),y_n)={\sf d}_\Omega(S(x_n),S_n(x_n))\le\|S-S_n\|_\infty\to0$, so
$S(x)=y$ by continuity, that is, $x=S^{-1}(y)$.  Since $S^{-1}$ is continuous,
$S^{-1}(y_n)\to S^{-1}(y)=x$ as well, contradicting
${\sf d}_\Omega(x_n,S^{-1}(y_n))\ge\delta$.
\end{proof}

\subsection{\texorpdfstring{{One step}}{One step}}
\begin{proof}[Proof of Theorem~\ref{thm:surrogate-one-step-consistency}]
We first show that neural approximation of the regular optimal pair implies uniform recovery of the induced density.  We then pass to the pair objective, and finally use primal majorization to recover the original JKO objective and the cone action.

Let $(T_*,q_*)$ be the regular optimal Monge-growth pair.  By Assumption~\ref{ass:incremental-neural-recovery}, there are neural pairs $(T_{\theta_n},q_{\theta_n})$ such that
\[
T_{\theta_n}\to T_*
\quad\text{in }C^1(\Omega),
\qquad
q_{\theta_n}\to q_*
\quad\text{uniformly}.
\]
Since $J_{T_*}$ is bounded away from zero and $DT_{\theta_n}\to DT_*$ uniformly, the same is true for $J_{T_{\theta_n}}$ for all sufficiently large $n$.  Assumption~\ref{ass:incremental-neural-recovery} makes $T_*$ and every $T_{\theta_n}$ a bijection of $\Omega$ onto itself, so Lemma~\ref{lem:inverse-uniform-convergence} applies and $T_{\theta_n}^{-1}\to T_*^{-1}$ uniformly; in particular $u_{\theta_n}$ below is defined at every point of $\Omega$.  If
\[
(T_{\theta_n})_\#(q_{\theta_n}^2\bar\rho\,dx)=u_{\theta_n}\,dx,
\]
then \eqref{eq:incremental-induced-density} gives, at a common point $y$,
\[
u_{\theta_n}(y)
=
\frac{q_{\theta_n}(T_{\theta_n}^{-1}(y))^2
\bar\rho(T_{\theta_n}^{-1}(y))}
{J_{T_{\theta_n}}(T_{\theta_n}^{-1}(y))}.
\]
The same formula holds for $\rho_*$ with $(T_*,q_*)$.  Hence
\[
\|u_{\theta_n}-\rho_*\|_{L^\infty(\Omega)}\longrightarrow0.
\]
This proves \eqref{eq:one-step-uniform-density-recovery} for all sufficiently large $n$.

Having recovered the state uniformly, we now pass to the surrogate objective.  Uniform convergence of maps, amplitudes, densities, and Jacobians is enough to pass to every term in $\mathcal I_\tau$.

The same change-of-variables formula and the positive continuous source put $\rho_*$ and the approximating densities in a fixed positive compact interval.  Finiteness and convexity of $U$ imply its continuity there.  Continuity of $U$, $V$, and $W$ on the relevant compact sets, uniform convergence of the maps and amplitudes, and convergence of the Jacobians imply convergence of the three energy terms in \eqref{eq:incremental-current-objective}.  The integrand in \eqref{eq:incremental-HK-action} also converges uniformly.  Therefore
\[
\mathcal I_\tau(T_{\theta_n},q_{\theta_n}\mid\bar\mu)
\longrightarrow
\mathcal I_\tau(T_*,q_*\mid\bar\mu).
\]
By Proposition~\ref{prop:incremental-HK-JKO-reparam},
\[
\mathcal I_\tau(T_*,q_*\mid\bar\mu)
=
\mathcal J_\tau(\mu_*\mid\bar\mu),
\]
and every admissible competitor has objective at least this minimum.  Thus the difference in \eqref{eq:one-step-incremental-objective-recovery} is nonnegative and converges to zero.

{
We finally compare the surrogate value with the original JKO value.  The primal upper bound sandwiches the exact JKO objective between the optimum and the surrogate objective, and the same comparison controls the excess cone action.

Since $\nu_{\theta_n}=u_{\theta_n}\,dx$, the minimality and the primal majorization give
\[
\mathcal J_\tau(\mu_*\mid\bar\mu)
\le
\mathcal J_\tau(\nu_{\theta_n}\mid\bar\mu)
\le
\mathcal I_\tau(T_{\theta_n},q_{\theta_n}\mid\bar\mu).
\]
Hence the original JKO gap is bounded by the pair-objective gap.  Moreover,
\[
\begin{aligned}
0
&\le
\mathcal A_{\bar\mu}(T_{\theta_n},q_{\theta_n})
-\HK^2(\bar\mu,\nu_{\theta_n})\\
&=2\tau\left[
\mathcal I_\tau(T_{\theta_n},q_{\theta_n}\mid\bar\mu)
-\mathcal J_\tau(\nu_{\theta_n}\mid\bar\mu)
\right]\\
&\le
2\tau\left[
\mathcal I_\tau(T_{\theta_n},q_{\theta_n}\mid\bar\mu)-
\mathcal J_\tau(\mu_*\mid\bar\mu)
\right].
\end{aligned}
\]
Choosing $n$ sufficiently large proves \eqref{eq:one-step-uniform-density-recovery}--\eqref{eq:one-step-cone-tightness}.}
\end{proof}

\subsection{\texorpdfstring{{A fixed finite horizon}}{A fixed finite horizon}}
\begingroup
\begin{corollary}[Finite-horizon composition of the one-step recovery]\label{cor:recursive-finite-time-tracking}
Fix a horizon $\mathcal T>0$ and $0<\tau\le\mathcal T$, let $N_\tau=\lfloor\mathcal T/\tau\rfloor$, and let
$\mu_{k+1}^\tau=\mathcal S_\tau(\mu_k^\tau)=\rho_{k+1}^\tau dx$, $0\le k<N_\tau$, be exact
$\HK$-JKO iterates satisfying Theorem~\ref{thm:full-step-regularity}, with $\rho_0^\tau$
positive and continuous.  For each $k$ choose an optimal Monge-growth pair
$(T_{k+1}^*,q_{k+1}^*)$ from $\mu_k^\tau$ to $\mu_{k+1}^\tau$, and suppose that every chosen
pair {satisfies all of Assumption~\ref{ass:incremental-neural-recovery}, including the network recovery property \eqref{eq:incremental-network-recovery}}.

Then there exist neural pairs $(T_{\theta_{n,k+1}},q_{\theta_{n,k+1}})$ and positive
continuous densities generated recursively from the current neural measure, with
$u_{n,0}^\tau:=\rho_0^\tau$,
\begin{equation}\label{eq:finite-horizon-action-faithful-recursion}
\widehat\mu_{n,0}^\tau=\mu_0^\tau=u_{n,0}^\tau dx,
\qquad
\widehat\mu_{n,k+1}^\tau
=(T_{\theta_{n,k+1}})_\#\bigl(q_{\theta_{n,k+1}}^2\widehat\mu_{n,k}^\tau\bigr)
=u_{n,k+1}^\tau dx,
\end{equation}
such that the pairs and the represented states converge,
\begin{equation}\label{eq:finite-horizon-pair-state-recovery}
\max_{0\le k\le N_\tau}
\|u_{n,k}^\tau-\rho_k^\tau\|_{L^\infty(\Omega)}
+
\max_{0\le k<N_\tau}
\left(
\|T_{\theta_{n,k+1}}-T_{k+1}^*\|_{C^1(\Omega)}
+
\|q_{\theta_{n,k+1}}-q_{k+1}^*\|_{L^\infty(\Omega)}
\right)
\longrightarrow0,
\end{equation}
and both variational objectives are recovered uniformly over the horizon:
\begin{equation}\label{eq:finite-horizon-variational-action-recovery}
\begin{aligned}
&\max_{0\le k<N_\tau}
\left|
\mathcal I_\tau(T_{\theta_{n,k+1}},q_{\theta_{n,k+1}}\mid\widehat\mu_{n,k}^\tau)
-\mathcal J_\tau(\mu_{k+1}^\tau\mid\mu_k^\tau)
\right|\longrightarrow0,\\
&\max_{0\le k<N_\tau}
\left|
\mathcal J_\tau(\widehat\mu_{n,k+1}^\tau\mid\widehat\mu_{n,k}^\tau)
-\mathcal J_\tau(\mu_{k+1}^\tau\mid\mu_k^\tau)
\right|\longrightarrow0 .
\end{aligned}
\end{equation}
The cone action and its transport and reaction components converge likewise, uniformly in $k$, to
$\tau^2I_g(\rho_{k+1}^\tau)$ and to $\mathcal A^{\sharp}_{\mu_k^\tau}(T_{k+1}^*,q_{k+1}^*)$ for
$\sharp\in\{{\rm tr},{\rm re}\}$.
\end{corollary}

\begin{proof}[Proof of Corollary~\ref{cor:recursive-finite-time-tracking}]
The finite-horizon statement is obtained by synchronizing the one-step approximations over the finitely many steps and then propagating the state error recursively.  After the states are controlled, the objective and action limits follow from the same continuity arguments as in the one-step result.

For each $0\le k<N_\tau$, Assumption~\ref{ass:incremental-neural-recovery} supplies neural pairs such that
\[
T_{\theta_{n,k+1}}\longrightarrow T_{k+1}^*
\quad\text{in }C^1(\Omega),
\qquad
q_{\theta_{n,k+1}}\longrightarrow q_{k+1}^*
\quad\text{uniformly}.
\]
Since the horizon contains finitely many steps, the sequences may be indexed by one common $n$ so that these convergences hold simultaneously for every $k$. 

We first propagate the density approximation.  The induction uses the change-of-variables formula together with uniform convergence of the inverse maps and Jacobians.

Set $\widehat\mu_{n,0}^\tau=\mu_0^\tau$ and define the recursion in \eqref{eq:finite-horizon-action-faithful-recursion}.  We prove the density convergence by induction.  Suppose that $u_{n,k}^\tau\to\rho_k^\tau$ uniformly.  The change-of-variables formula gives, at $y\in\Omega$,
\[
u_{n,k+1}^\tau(y)
=
\frac{
 q_{\theta_{n,k+1}}(T_{\theta_{n,k+1}}^{-1}(y))^2
 u_{n,k}^\tau(T_{\theta_{n,k+1}}^{-1}(y))
}{
 J_{T_{\theta_{n,k+1}}}(T_{\theta_{n,k+1}}^{-1}(y))
}.
\]
The same identity holds for $\rho_{k+1}^\tau$ with
$(u_{n,k}^\tau,T_{\theta_{n,k+1}},q_{\theta_{n,k+1}})$ replaced by
$(\rho_k^\tau,T_{k+1}^*,q_{k+1}^*)$.  The Jacobians stay uniformly away from zero and the inverse maps converge uniformly.  Hence $u_{n,k+1}^\tau\to\rho_{k+1}^\tau$ uniformly.  The induction starts at $k=0$, and finiteness of the horizon yields \eqref{eq:finite-horizon-pair-state-recovery}.  In particular, for all sufficiently large $n$, all source and induced densities lie in one positive compact interval.

With uniform state convergence established at every step, we can pass to the pair objective and the cone action step by step; finiteness of the horizon makes the convergence uniform in $k$.

Uniform convergence of the densities and amplitudes, $C^1$ convergence of the maps, and the formulas \eqref{eq:incremental-current-objective} and \eqref{eq:incremental-HK-action} imply
\[
\begin{aligned}
&\mathcal I_\tau
(T_{\theta_{n,k+1}},q_{\theta_{n,k+1}}\mid\widehat\mu_{n,k}^\tau)
\longrightarrow
\mathcal I_\tau(T_{k+1}^*,q_{k+1}^*\mid\mu_k^\tau),\\
&\mathcal A_{\widehat\mu_{n,k}^\tau}
(T_{\theta_{n,k+1}},q_{\theta_{n,k+1}})
\longrightarrow
\mathcal A_{\mu_k^\tau}(T_{k+1}^*,q_{k+1}^*).
\end{aligned}
\]
The selected exact pair is optimal, so Proposition~\ref{prop:incremental-HK-JKO-reparam} and \eqref{eq:exact-step-dissipation-main} give
\[
\begin{aligned}
\mathcal I_\tau(T_{k+1}^*,q_{k+1}^*\mid\mu_k^\tau)
&=\mathcal J_\tau(\mu_{k+1}^\tau\mid\mu_k^\tau),\\
\mathcal A_{\mu_k^\tau}(T_{k+1}^*,q_{k+1}^*)
&=\HK^2(\mu_k^\tau,\mu_{k+1}^\tau)
=\tau^2I_g(\rho_{k+1}^\tau).
\end{aligned}
\]
This proves the first limit in \eqref{eq:finite-horizon-variational-action-recovery} together with the cone-action convergence asserted at the end of the corollary.

We next recover the original JKO objective.  Uniform density convergence gives convergence in $\HK$, and the triangle inequality transfers this to the distance between consecutive approximate states.

By \eqref{eq:HK2-FR2}, uniform density convergence implies
$\HK(\widehat\mu_{n,k}^\tau,\mu_k^\tau)\to0$ for every $k$.  Since $\HK$ is a metric, the triangle inequality gives
\[
\left|\HK(\widehat\mu_{n,k+1}^\tau,\widehat\mu_{n,k}^\tau)-\HK(\mu_{k+1}^\tau,\mu_k^\tau)\right|
\le
\HK(\widehat\mu_{n,k+1}^\tau,\mu_{k+1}^\tau)+\HK(\widehat\mu_{n,k}^\tau,\mu_k^\tau)
\longrightarrow0 .
\]
The energy terms converge by uniform density convergence, so the original JKO objectives converge as well.  This proves the second limit in \eqref{eq:finite-horizon-variational-action-recovery}.

It remains only to separate the limiting cone action into its transport and reaction components.  This follows directly from uniform convergence of the two split integrands and of the source densities.

The two integrands in the splitting \eqref{eq:cone-action-splitting} converge uniformly, while the source densities converge uniformly; hence the transport and reaction components converge as asserted at the end of the corollary.
\end{proof}

% \begin{remark}[Three consequences of Corollary~\ref{cor:recursive-finite-time-tracking}]\label{rem:finite-horizon-consequences}
% Under the hypotheses of Corollary~\ref{cor:recursive-finite-time-tracking}, three statements
% follow at once from its proof.

% First, for every $n$ and $k$ the cone action majorizes the
% squared $\HK$ distance between its own endpoints, and both quantities converge to the same exact value
% $\tau^2I_g(\rho_{k+1}^\tau)$, so the primal majorization is asymptotically tight uniformly over the
% horizon.

% Second, combining this with the metric-dissipation identity
% \eqref{eq:exact-step-dissipation-main} recovers the exact discrete slopes,
% \[
% \max_{0\le k<N_\tau}
% \left|\tau^{-1}\HK(\widehat\mu_{n,k+1}^\tau,\widehat\mu_{n,k}^\tau)-\sqrt{I_g(\rho_{k+1}^\tau)}\right|
% \longrightarrow0 .
% \]

% Third, the horizon is finite, so the action convergence may be summed; together with energy convergence
% at the terminal state and the exact one-step inequality this gives the represented counterpart of the
% discrete energy balance,
% \[
% \mathcal F(\widehat\mu_{n,N_\tau}^\tau)
% +\frac1{2\tau}\sum_{k=0}^{N_\tau-1}
% \mathcal A_{\widehat\mu_{n,k}^\tau}(T_{\theta_{n,k+1}},q_{\theta_{n,k+1}})
% \longrightarrow
% \mathcal F(\mu_{N_\tau}^\tau)+\frac\tau2\sum_{k=0}^{N_\tau-1}I_g(\rho_{k+1}^\tau)
% \le \mathcal F(\mu_0^\tau).
% \]
% All three are statements about a fixed horizon and carry no rate and no constant uniform in
% the number of phases.
% \end{remark}

\endgroup

\section{Proof of Theorem~\ref{thm:a-posteriori} and Corollary~\ref{cor:ap-boltzmann-L1} (Theorem~D)}\label{sec:proof-a-posteriori}

\begin{proof}[Proof of Theorem~\ref{thm:a-posteriori}]
We prove the certificate in two stages.  First we construct a computable lower bound for the exact JKO minimum from a feasible dual pair and the tangent inequality for the positive-semidefinite interaction.  We then compare this lower bound with the primal cone-action value and convert the resulting objective gap into a state-error bound.

For $\sigma=s\,dx$, every pair satisfying \eqref{eq:ap-feasible} is a competitor in
\eqref{eq:HK-dual}, and therefore
\begin{equation}\label{eq:ap-minorant}
\HK^2(\sigma,\bar\mu)\ge
\int_\Omega\bigl(1-e^{-\zeta}\bigr)s\,dx
+\int_\Omega\bigl(1-e^{-\xi}\bigr)d\bar\mu .
\end{equation}
In particular, $\sigma\mapsto\HK^2(\sigma,\bar\mu)$ is convex in the flat structure of
$\mathcal M_+(\Omega)$, since it is a supremum of affine functionals.  Positive
semidefiniteness makes $\mathcal W$ convex, and $U''>0$ makes $\int U$ strictly convex on
its finite domain.  Existence from Theorem~\ref{thm:main-fully-implicit-HK-JKO-step} thus
gives uniqueness of $\mu_1$.

We now linearize the interaction energy at the candidate measure $\nu$.  Positive semidefiniteness is precisely what makes the quadratic interaction convex and validates this tangent lower bound.

The tangent inequality for the convex quadratic form at $\nu$ is
\[
\mathcal W(\sigma)\ge
\int_\Omega W[u]s\,dx-\mathcal W(\nu).
\]
Indeed, setting $
f(t):=\mathcal W\bigl((1-t)\nu+t\sigma\bigr),$
the positive semidefiniteness of $W$ makes $f$ convex, so that $f(1)\ge f(0)+f'(0)$. 
Since
$$
f'(0)=\int_\Omega W[u](s-u)\,dx
\quad\text{and}\quad
\int_\Omega W[u]u\,dx=2\mathcal W(\nu),
$$
this gives precisely the displayed inequality.

Combining the metric minorant with the interaction tangent bound gives a pointwise-separable lower bound for every competitor.  Minimizing that lower bound through the convex conjugate of $U$ produces the dual certificate.

Combining the interaction tangent inequality with \eqref{eq:ap-minorant} and using the
definition of $a$ in \eqref{eq:ap-field} gives
\[
\mathcal J_\tau(\sigma\mid\bar\mu)
\ge
\int_\Omega[U(s)+a s] \,dx-\mathcal W(\nu)
+\frac1{2\tau}\int_\Omega(1-e^{-\xi})\,d\bar\mu.
\]

We now optimize the preceding lower bound in the density variable and compare it with the primal upper certificate supplied by $(T,q)$.

Since $\inf_{s\ge0}\{U(s)+as\}=-U^*(-a)$ pointwise,
\[
\min_\sigma\mathcal J_\tau(\sigma\mid\bar\mu)=\mathcal J_\tau(\mu_1\mid\bar\mu)
\ge-\int_\Omega U^*(-a)dx-\mathcal W(\nu)
+\frac1{2\tau}\int_\Omega(1-e^{-\xi})d\bar\mu.
\]
On the primal side, \eqref{eq:primal-upper-certificate} gives
$\mathcal J_\tau(\nu\mid\bar\mu)\le\mathcal I_\tau(T,q\mid\bar\mu)$.
Subtracting the two bounds proves \eqref{eq:ap-objective-gap} and $\mathcal R\ge0$.
(If $\int_\Omega U^*(-a)\,dx=+\infty$, then $\mathcal R=+\infty$ and the conclusions hold
vacuously; the estimate is informative only when the dual choice makes this integral finite.)

We next convert the objective gap into a state error.  For this purpose we use the exact step's own normalized source dual potential and the Euler--Lagrange identity.  Let $\phi_1$ be the normalized
source-side transformed dual potential in Theorem~\ref{thm:full-step-regularity}(iii), so
that the coefficient of the first argument in \eqref{eq:transformed-HK-dual} is $2\phi_1$; write $\chi_1$ for the companion target component.  Since $(2\phi_1,\chi_1)$ is feasible in \eqref{eq:transformed-HK-dual} for every first argument and optimal at $\mu_1$,
\[
\HK^2(\sigma,\bar\mu)\ge2\int_\Omega\phi_1 s\,dx+\int_\Omega\chi_1\,d\bar\mu,
\qquad
\HK^2(\mu_1,\bar\mu)=2\int_\Omega\phi_1\rho_1\,dx+\int_\Omega\chi_1\,d\bar\mu,
\]
and subtracting the second from the first gives
\[
\HK^2(\sigma,\bar\mu)-\HK^2(\mu_1,\bar\mu)
\ge2\int_\Omega\phi_1(s-\rho_1)dx.
\]
Positive semidefiniteness also gives, by the tangent inequality applied at $\mu_1$ instead of $\nu$,
\[
\mathcal W(\sigma)-\mathcal W(\mu_1)
\ge\int_\Omega W[\rho_1](s-\rho_1)dx.
\]
Adding these to the exact identity $\int V\,d\sigma-\int V\,d\mu_1=\int V(s-\rho_1)dx$ and to
$\int_\Omega[U(s)-U(\rho_1)]dx$ yields
\[
\mathcal J_\tau(\sigma\mid\bar\mu)-\mathcal J_\tau(\mu_1\mid\bar\mu)
\ge
\int_\Omega\bigl[U(s)-U(\rho_1)\bigr]dx
+\int_\Omega\Bigl(V+W[\rho_1]+\frac{\phi_1}{\tau}\Bigr)(s-\rho_1)\,dx,
\]
and the Euler--Lagrange equation \eqref{eq:full-Euler-Lagrange-main} identifies the bracket in the
second integral as $-U'(\rho_1)$, so that
\[
\mathcal J_\tau(\sigma\mid\bar\mu)-\mathcal J_\tau(\mu_1\mid\bar\mu)
\ge
\int_\Omega[U(s)-U(\rho_1)-U'(\rho_1)(s-\rho_1)]dx.
\]
Taking $\sigma=\nu$ and using \eqref{eq:ap-objective-gap} proves
\eqref{eq:ap-estimate}.
\end{proof}

\begin{proof}[Proof of Corollary~\ref{cor:ap-boltzmann-L1}]
For Boltzmann entropy, the convexity remainder becomes the relative-entropy integrand.  We first compare it pointwise with squared Hellinger distance and then convert the latter into an $L^1$ estimate by Cauchy--Schwarz.

For $U(s)=s\log s-s$ one has $U'(s)=\log s$, so the integrand in \eqref{eq:ap-estimate} is
\[
U(u)-U(\rho_1)-U'(\rho_1)(u-\rho_1)
=u\log\frac{u}{\rho_1}-u+\rho_1 .
\]
It dominates the squared Hellinger integrand.  With
$t=\sqrt{u/\rho_1}$, the difference after division by $\rho_1$ is
\[
2t^2\log t-t^2+1-(t-1)^2
=2t\,(t\log t-t+1)\ge0,
\]
since $t\log t-t+1\ge0$ for $t\ge0$, with the continuous convention $0\log0=0$.  Hence
\[
u\log\frac{u}{\rho_1}-u+\rho_1
\ge(\sqrt u-\sqrt{\rho_1})^2 .
\]
The pointwise estimate gives Hellinger control.  We now pass from Hellinger distance to $L^1$ using the factorization $u-\rho_1=(\sqrt u-\sqrt{\rho_1})(\sqrt u+\sqrt{\rho_1})$.

Cauchy--Schwarz then gives
\[
\|u-\rho_1\|_{L^1}^2
\le
\left(\int_\Omega(\sqrt u-\sqrt{\rho_1})^2dx\right)
\left(\int_\Omega(\sqrt u+\sqrt{\rho_1})^2dx\right)
\le2\bigl(\nu(\Omega)+\mu_1(\Omega)\bigr)\mathcal R.
\]
The remaining inequality in \eqref{eq:ap-boltzmann-L1} is the mass bound
\eqref{eq:full-mass-sublevel-bound}.
\end{proof}

{
\begin{proposition}[Decomposition of the gap]\label{prop:ap-gap-decomposition}
Under the hypotheses of Theorem~\ref{thm:a-posteriori},
\begin{equation}\label{eq:ap-gap-decomposition}
\begin{aligned}
\mathcal R
={}&\underbrace{\int_\Omega\bigl[U(u)+a\,u+U^*(-a)\bigr]dx}_{=:~\mathcal R_{\rm EL}\ge0}\\
&+\underbrace{\frac1{2\tau}\left[\mathcal A_{\bar\mu}(T,q)
-\int_\Omega\bigl(1-e^{-\zeta}\bigr)u\,dx
-\int_\Omega\bigl(1-e^{-\xi}\bigr)d\bar\mu\right]}_{=:~\mathcal R_{\rm met}\ge0}.
\end{aligned}
\end{equation}
\end{proposition}
\begin{proof}
We verify the decomposition by direct algebra.  After substitution, the candidate-dependent terms cancel exactly; the two remaining pieces are nonnegative for independent reasons, namely Fenchel--Young and the primal--dual metric bounds.

Substituting $a$ from \eqref{eq:ap-field}, using
$\int_\Omega W[u]u\,dx=2\mathcal W(\nu)$, and adding the two terms in
\eqref{eq:ap-gap-decomposition}, the contributions involving $1-e^{-\zeta}$ cancel and the
remaining terms are exactly $\mathcal R$ in \eqref{eq:ap-field}.  Moreover,
$\mathcal R_{\rm EL}\ge0$ by Fenchel--Young, while
\[
\mathcal A_{\bar\mu}(T,q)
\ge\HK^2(\bar\mu,\nu)
\ge\int_\Omega(1-e^{-\zeta})u\,dx
+\int_\Omega(1-e^{-\xi})\,d\bar\mu
\]
by \eqref{eq:incremental-HK-upper-bound} and \eqref{eq:ap-minorant}; hence
$\mathcal R_{\rm met}\ge0$.
\end{proof}

For Boltzmann entropy and positive $u$ one has $U^*(p)=e^p$, so with $r=U'(u)+a=\log u+a$ the
identity $e^{-a}=ue^{-r}$ gives $\mathcal R_{\rm EL}=\int_\Omega u(e^{-r}-1+r)dx$, which is
$\frac12\int_\Omega r^2u\,dx$ to leading order.  The decomposition therefore separates a
discrete Euler residual from the primal--dual metric gap.
}

\section{Design and implementation notes}\label{sec:design-notes}
This appendix records the potential-generated parametrization and its admissibility
conditions, step-scaling limitations, and the implementation details for
Algorithm~\ref{alg:composite-HK-JKO} and Section~\ref{sec:experiments}.

\subsection{Potential-generated pairs and admissibility}\label{sec:design-notes-Zadm}

{
\begin{remark}[Potential-generated and step-scaled pairs]\label{rem:potential-generated-pair}\label{rem:tau-scaled-ansatz}
Proposition~\ref{prop:HK-Monge-representation} suggests parametrizing an optimal pair by one scalar network $z_\theta$:
\begin{equation}\label{eq:potential-generated-pair}
T_{z_\theta}(x)=\exp_x\!\left(-\arctan\!\left(\frac{|\nabla z_\theta|}{2}\right)\frac{\nabla z_\theta}{|\nabla z_\theta|}\right),
\qquad
q_{z_\theta}(x)=e^{-z_\theta}\left(1+\frac{|\nabla z_\theta|^2}{4}\right)^{1/2},
\end{equation}
with $T_{z_\theta}(x)=x$ when $\nabla z_\theta(x)=0$.  For $z_\theta$ in the admissible class $\mathcal Z_{\rm adm}$ defined below, the cutoff is inactive and
\begin{equation*}
\mathcal A_\mu(T_{z_\theta},q_{z_\theta})
=\int_\Omega\left[1-2e^{-z_\theta}+e^{-2z_\theta}\left(1+\frac{|\nabla z_\theta|^2}{4}\right)\right]d\mu
=\int_\Omega\left(z_\theta^2+\frac{|\nabla z_\theta|^2}{4}\right)d\mu+O\bigl(\|z_\theta\|_{C^1}^3\bigr).
\end{equation*}
The reversed and forward orientations carry opposite leading-order signs: $z_{\rm rev}\approx-2\tau g_{\rho_1}$ and $z_{\rm fwd}\approx2\tau g_{\rho_1}$.  This motivates the scaled ansatz
\begin{equation}\label{eq:tau-scaled-ansatz}
z_\theta=\tau\,w_\theta .
\end{equation}
For a smooth positive current density $\rho$ with smooth $g_\rho$, suppose that $z_\tau:=2\tau g_\rho$ belongs to $\mathcal Z_{\rm adm}$ for all sufficiently small $\tau$.  Then the explicit forward predictor gives, to first order,
\[
T_{z_\tau}=\mathrm{id}-\tau\nabla g_\rho+O(\tau^2),
\qquad
q_{z_\tau}^2=1-4\tau g_\rho+O(\tau^2),
\]
and therefore
\[
(T_{z_\tau})_\#(q_{z_\tau}^2\rho\,dx)
=\bigl[\rho+\tau\{\operatorname{div}(\rho\nabla g_\rho)-4\rho g_\rho\}\bigr]dx+o(\tau)
\]
in the distributional sense.  This is a first-order predictor calculation, not a consistency estimate for the exact
minimizer.  The admissibility and approximation requirements are stated below;
Section~\ref{sec:design-notes-scaling} discusses step scaling and step uniformity.
\end{remark}
}

The parametrization \eqref{eq:potential-generated-pair} of
Remark~\ref{rem:potential-generated-pair} defines an admissible pair only after domain,
injectivity, and minimizing-branch conditions are imposed.  We therefore restrict it to
\[
\mathcal Z_{\rm adm}:=
\left\{
 z\in C^2(\Omega)\ \middle|\
 \begin{aligned}
 &T_z:\Omega\to\Omega\text{ is an orientation-preserving}\\[-1mm]
 &\qquad C^1\text{ diffeomorphism},\qquad \inf_\Omega J_{T_z}>0,\\
 &{\sf d}_\Omega(x,T_z(x))
 =\arctan\!\left(\frac{|\nabla z(x)|}{2}\right)
 \quad\text{for every }x\in\Omega
 \end{aligned}
\right\}.
\]
On a bounded convex Euclidean domain the last identity follows once
$T_z(\Omega)\subset\Omega$; on the torus it additionally requires the exponential
displacement to remain on the minimizing geodesic branch.
The trade-off is the following.  By Proposition~\ref{prop:HK-Monge-representation} every optimal pair between positive densities has the form \eqref{eq:potential-generated-pair} for a Lipschitz $z$, so this parametrization matches the structure of an exact optimal pair, removes the cutoff, and eliminates the directions along which $T$ and $q$ trade off at fixed pushforward.  {It does not, however, guarantee admissibility for an arbitrary network output, and standard uniform universal approximation \cite{Leshno93} does not by itself yield Assumption~\ref{ass:incremental-neural-recovery} for this class: one must separately require a regular potential $z_*\in C^2(\Omega)$ and a network sequence with
\[
\|z_{\theta_n}-z_*\|_{C^2(\Omega)}\longrightarrow0
\]
preserving the conditions defining $\mathcal Z_{\rm adm}$, which then gives $T_{z_{\theta_n}}\to T_{z_*}$ in $C^1$ and $q_{z_{\theta_n}}\to q_{z_*}$ uniformly.}

Lemma~\ref{lem:transformed-potential-Lipschitz} relates $z_*$ to the transformed dual potential; the exact-step theory supplies only Lipschitz regularity. The Euler--Lagrange equation \eqref{eq:full-Euler-Lagrange-main} is algebraic in the density and potential and gives no elliptic regularity gain. Even with smoother $V$ and $W$, the $C^2$ control needed here would require additional regularity theory for unbalanced optimal transport. The density bounds \eqref{eq:canonical-two-sided-density-bound} and the uniform sub-cutoff bound of Corollary~\ref{cor:HK-growth-couple-relations-full} identify the relevant non-degenerate regime; see also \cite{GGV25}. Establishing this additional regularity remains open here.

{Under the step-scaled ansatz $z_\theta=\tau w_\theta$, invertibility near the identity follows from a quantitative Hessian bound such as $\tau\|D^2w_\theta\|_\infty\ll1$ together with the boundary or minimizing-branch condition; small $\tau$ alone is not sufficient.  Finally, $J_{T_{z_\theta}}$ involves second derivatives, so the stochastic trace estimators available for \eqref{eq:neural-ode-log-jacobian} do not apply directly.}  The two parametrizations are complementary: \eqref{eq:neural-ode-map} guarantees admissibility by construction, whereas \eqref{eq:potential-generated-pair} matches the exact optimal-pair structure under the additional restrictions above.  The main representability theorem is formulated for a generic admissible pair class and applies to either choice.

\subsection{Step scaling and uniform estimates}\label{sec:design-notes-scaling}
The scaling $z_\theta=\tau w_\theta$ in Remark~\ref{rem:tau-scaled-ansatz} separates the
$O(\tau)$ potential from its $O(1)$ profile.  For admissible potentials $z_*=\tau w_*$ and
$z_\theta$ with uniformly bounded $C^1$ profiles, an $O(\tau)$ approximation of $w_*$ in $C^1$
yields $O(\tau^2)$ errors in the map
and growth factor of \eqref{eq:potential-generated-pair}, which enter the one-step stability
estimate \eqref{eq:eta-to-tau-squared}.

The bounds of Theorem~\ref{thm:full-step-regularity} do not, however, give step-uniform density
control by direct iteration.  For Boltzmann entropy, \eqref{eq:canonical-two-sided-density-bound}
incurs a factor $e^{\omega_1}$ per step, where $\omega_1$ is the oscillation of the self-consistent
field; without additional $\tau$-scaled control, this yields no $\tau$-independent bound over
$N_\tau=\lfloor\mathcal T/\tau\rfloor$ steps.  Sharper estimates or a separate argument are needed.
The results in the main text require neither this scaling nor step-uniform bounds.

\input{sections/experiment_details}

\bibliographystyle{unsrtnat}
\bibliography{references}
\end{document}

%% file: sections/experiments.tex
% HKFLOW_SOLVER_PROFILE: neural
\section{Numerical experiments}\label{sec:experiments}\label{sec:numerical-protocol}
\begingroup
\setlength{\textfloatsep}{8pt plus 2pt minus 2pt}
\setlength{\floatsep}{7pt plus 2pt minus 2pt}
\setlength{\intextsep}{8pt plus 2pt minus 2pt}
\renewcommand{\topfraction}{0.95}
\renewcommand{\bottomfraction}{0.90}
\renewcommand{\textfraction}{0.05}
\renewcommand{\floatpagefraction}{0.72}
\captionsetup[figure]{font=small,skip=3pt}
\setcounter{topnumber}{4}
\setcounter{bottomnumber}{3}
\setcounter{totalnumber}{6}
\makeatletter
\setlength{\@fptop}{0pt}
\setlength{\@fpsep}{10pt}
\setlength{\@fpbot}{0pt plus 1fil}
\makeatother

\input{generated/trajectory_overview.tex}
The code for the numerical experiments is available at \url{https://github.com/tyty4/hkflow}.

Trajectory states are fitted with the positive neural density family in Table~\ref{tab:density-family}. Masses, moments, and reference errors are quadrature estimates; Gaussian references are used only for evaluation; the trainable density family also contains nonquadratic log densities. Finite-box errors include truncation. Comparisons use a common seed and are descriptive, not statistical significance claims.

Time-dependent error curves are shown for $t\geq5\tau$, omitting the first five time levels; other diagnostics and tabulated values use the full trajectories.

Trajectory pairs are checked against the identity on held-out samples or enriched quadrature. {The common neural solver terminates if its prescribed attempts fail validation; no fallback is propagated. Rejected attempts do not advance physical time.} Ablation endpoints use the latest time attained by every compared trajectory; curves retain their own accepted clocks. Tables report these times, null/fallback counts, empirical guards, action shares, and map contraction bounds.

The ``reference transport share'' is a leading-order diagnostic at the numerical reference state, when available, or the projected neural endpoint. It is sensitive to differentiation, quadrature, and small near-equilibrium action; it is not a finite-step transport plan. The optional transport-share acceptance criterion is not imposed. Implementation details are in Appendix~\ref{sec:design-notes-experiments}.

\begin{table}[!htbp]
\centering\small
\caption{Positive neural state families and the actual pair-training rules. Importance counts are draws before rejection outside the box.}
\label{tab:density-family}
\begingroup
\setlength{\tabcolsep}{4pt}
\renewcommand{\arraystretch}{1.08}
\let\densitytabular\tabular
\renewcommand{\tabular}[1]{\densitytabular{@{}>{\raggedright\arraybackslash}p{1.6in}c>{\raggedright\arraybackslash}p{1.8in}>{\raggedright\arraybackslash}p{1.8in}@{}}}
\selectedtable{generated/density_family_table.tex}{1,2,4,5,7}
\endgroup
\par\smallskip
\parbox{\linewidth}{\footnotesize Here, ``mixed-activation FCNN'' denotes a fully connected neural network combining quadratic and $\tanh$ activations; see Appendix~\ref{sec:design-notes-experiments}.}
\end{table}

\begin{table}[!htbp]
\centering\scriptsize
\resizebox{\linewidth}{!}{\selectedtable{generated/diagnostics_table.tex}{1,2,4,5,7}}
\caption{Trajectory diagnostics. Action shares are cone-action-weighted; $\max_\ell\kappa_{\theta,\ell}<1$ certifies the two blocks of each accepted map.}
\label{tab:diagnostics}
\end{table}

\FloatBarrier
\subsection{One-step PDE consistency and small-time scaling}\label{sec:one-step-numerics}
For $g(\rho)=\log\rho+V$, the one-dimensional $\HK$ flow is
\begin{equation*}
\partial_t\rho=\partial_x(\rho\,\partial_x g)-4\rho g.
\end{equation*}
We compare the backward time increment and the endpoint PDE operator pointwise:
\[
L_\tau:=\frac{\rho_\tau-\rho_0}{\tau},
\qquad
F_\tau:=\partial_x(\rho_\tau\partial_xg_\tau)-4\rho_\tau g_\tau,
\qquad g_\tau=\log\rho_\tau+V.
\]
Each $\tau$ uses an independent one-step update from
\[
V(x)=\tfrac12\bigl(1-\cos(\pi x)\bigr),
\qquad
\rho_0(x)=I_0(0.2)^{-1}\exp\!\left(0.2\cos(\pi x)\right),
\]
where $I_0$ is the modified Bessel function.  Since $\partial_xg_0=0.3\pi\sin(\pi x)$, the smooth positive source satisfies no flux at $x=\pm1$.  The fitted chemical potential uses the common trainable FCNN with cosine inputs enforcing zero normal derivative; architecture, budget, seed, and grid are fixed across $\tau$.

Spatial derivatives are evaluated by automatic differentiation in CPU, using fixed FCNN weights learned on MPS. We report the relative strong residual over the full interval,
\[
\mathcal R^{\rm str}_\tau
=\frac{\|L_\tau-F_\tau\|_{L^2(-1,1)}}
{\|L_\tau\|_{L^2(-1,1)}+\|F_\tau\|_{L^2(-1,1)}}.
\]
It includes truncation, projection, and training errors, so Figure~\ref{fig:hk-euler-check} is an agreement diagnostic rather than a convergence-rate estimate.

\begin{figure}[!htbp]
    \centering
    \safeincludegraphics[width=0.98\linewidth]{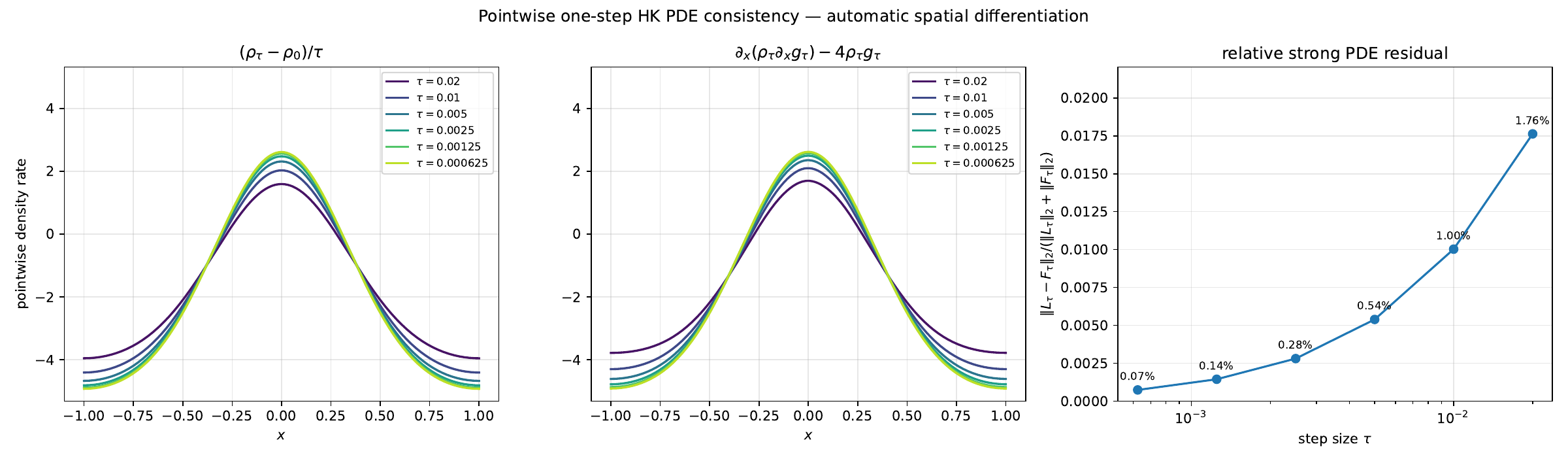}
    \caption{Pointwise one-step PDE consistency under $\tau$ refinement. Left and center: $L_\tau$ and $F_\tau$ on a common scale. Right: $\mathcal R^{\rm str}_\tau$.}
    \label{fig:hk-euler-check}
\end{figure}

The explicit forward predictor in Remark~\ref{rem:tau-scaled-ansatz} has the expansion
\[
T_{\rm pred}=\mathrm{id}-\tau\nabla g_\rho+O(\tau^2),
\qquad
q_{\rm pred}^2=1-4\tau g_\rho+O(\tau^2).
\]
We test the optimized pairs against these fields using the same Neumann-compatible $\rho_0,V$ as above. Each $\tau$ uses fresh dense map and growth networks, initialized at $(\mathrm{id},1)$ and optimized in the same pair objective. The predictor fields are used to evaluate the resulting pairs, without a predictor-fitting loss. The entropy quadrature uses the current-source score through {the integration-by-parts correction \eqref{eq:jcp-ibp-correction}}. Architecture, quadrature, and optimizer budget are fixed across $\tau$ (Appendix~\ref{sec:design-notes-experiments}). Figure~\ref{fig:tau-scaling} and Table~\ref{tab:tau-scaling} report errors over the full interval. Separate spatial and optimizer refinements check that the observed field errors are resolved; this is an empirical one-step consistency test.

\begin{figure}[!htbp]
    \centering
    \safeincludegraphics[width=0.96\linewidth]{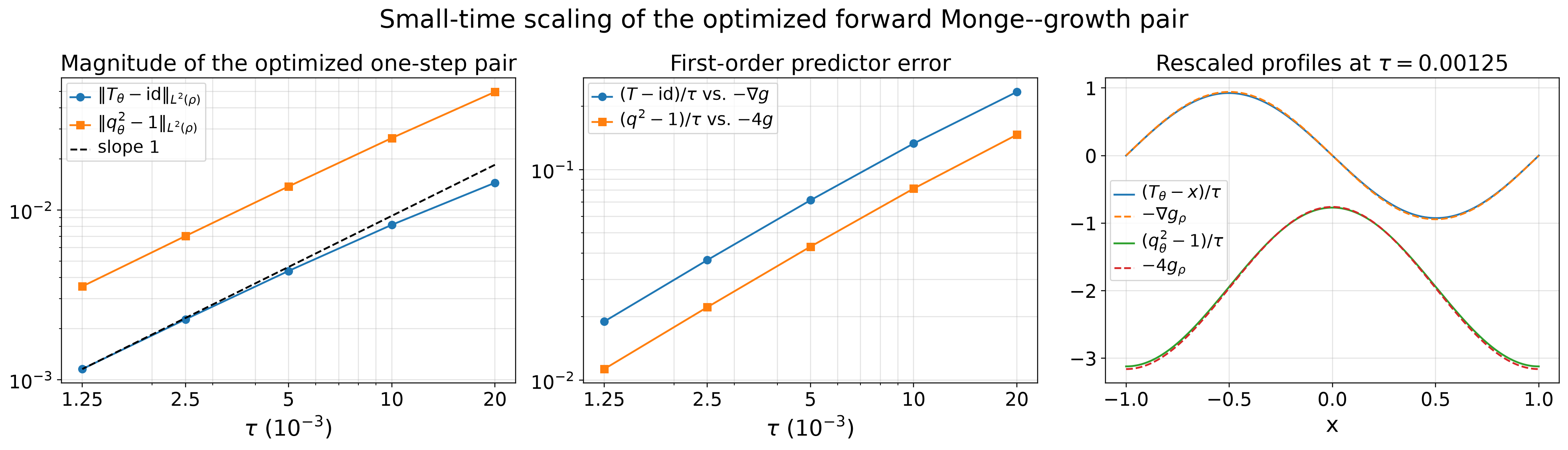}
    \caption{Small-time pair scaling.  Left: pair deviations; center: rescaled-field errors; right: profiles at the smallest $\tau$.}
    \label{fig:tau-scaling}
\end{figure}

\begin{table}[!htbp]
\centering\scriptsize
\resizebox{\linewidth}{!}{\input{generated/tau_scaling_table.tex}}
\caption{Small-time pair diagnostics on $[-1,1]$. Velocity and reaction errors are relative $L^2(\rho_0\,dx)$ errors of $(T-\mathrm{id})/\tau$ and $(q^2-1)/\tau$ against $-g_0'$ and $-4g_0$.}
\label{tab:tau-scaling}
\end{table}

\input{generated/one_step_observations.tex}

\FloatBarrier
\subsection{Entropy--potential relaxation and the role of both pair components}\label{sec:entropy-potential-experiments}
On $[-1,1]^d$ we take
\[
\begin{gathered}
U(s)=s\log s-s,
\qquad
V(x)=20|x|^2,\\
\rho_0(x)=\prod_{j=1}^d C_\beta\exp\!\left(
-20x_j^2+\beta(x_j-x_j^3/3)\right),
\qquad \beta=4,\\
C_\beta^{-1}=\frac12\int_{-1}^{1}\exp\!\left(
-20s^2+\beta(s-s^3/3)\right)\,ds.
\end{gathered}
\]
The source has mass $2^d$ and satisfies the no-flux condition
$\partial_j(\log\rho_0+V)=\beta(1-x_j^2)=0$ on each face.
The asymmetric tilt produces directed transport with bounded chemical-potential variation.
We compare with a finite-volume trajectory and the equilibrium
$\rho_\infty=e^{-20|x|^2}$ at accepted physical times.

\begin{figure}[!htbp]
    \centering
    \safeincludegraphics[width=0.88\linewidth]{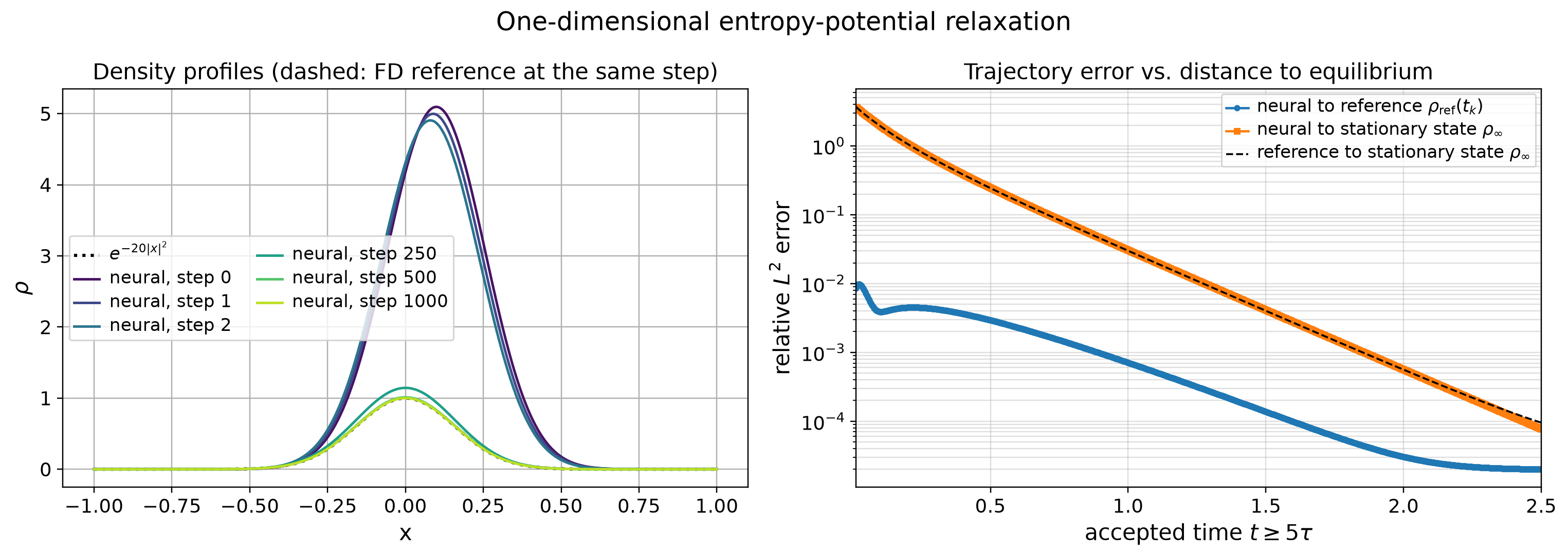}
    \caption{One-dimensional relaxation. Step labels count attempts; finite-volume references and diagnostics use the corresponding accepted physical time.}
    \label{fig:hk-stationary-1d}
\end{figure}

To assess both $\HK$ mechanisms, we compare the full pair with pure growth ($T_\theta=\mathrm{id}$) and transport only ($q_\theta=1$), reporting trajectory error and transport share alongside energy.

\begin{figure}[!htbp]
    \centering
    \safeincludegraphics[width=0.98\linewidth]{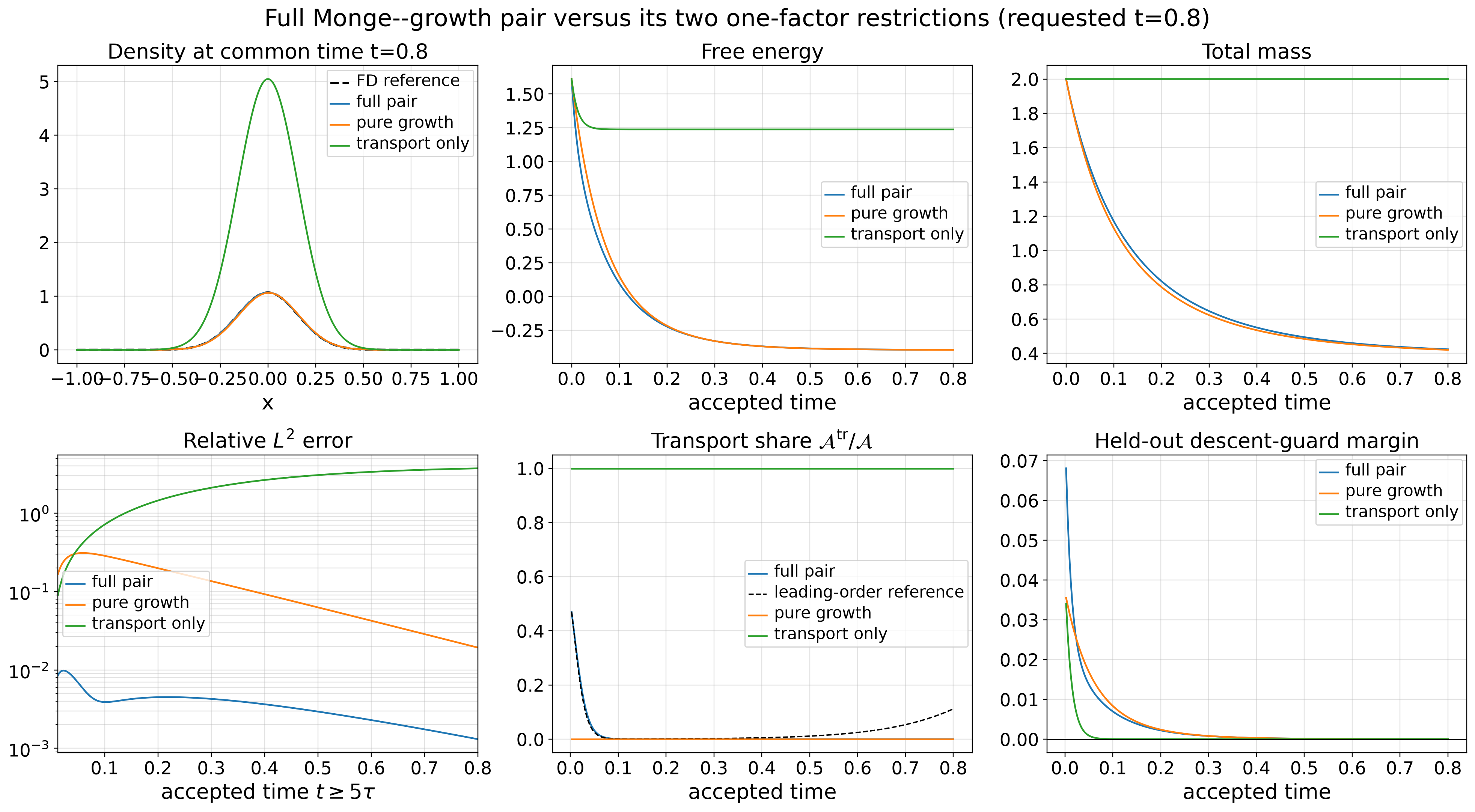}
    \caption{Full pair versus pure-growth and transport-only restrictions.}
    \label{fig:restriction-ablation}
\end{figure}

\begin{table}[!htbp]
\centering\scriptsize
\resizebox{\linewidth}{!}{\input{generated/restriction_table.tex}}
\caption{Pair-restriction diagnostics at a common attained physical time.  Transport shares use the prefix ending at that time.  The energy-monotonicity flag uses $E_{k+1}-E_k\leq 10^{-7}(1+|E_k|)$, a tighter tolerance than the acceptance guard.}
\label{tab:pure-growth-ablation}
\end{table}

\begin{figure}[!t]
    \centering
    \safeincludegraphics[width=0.90\linewidth]{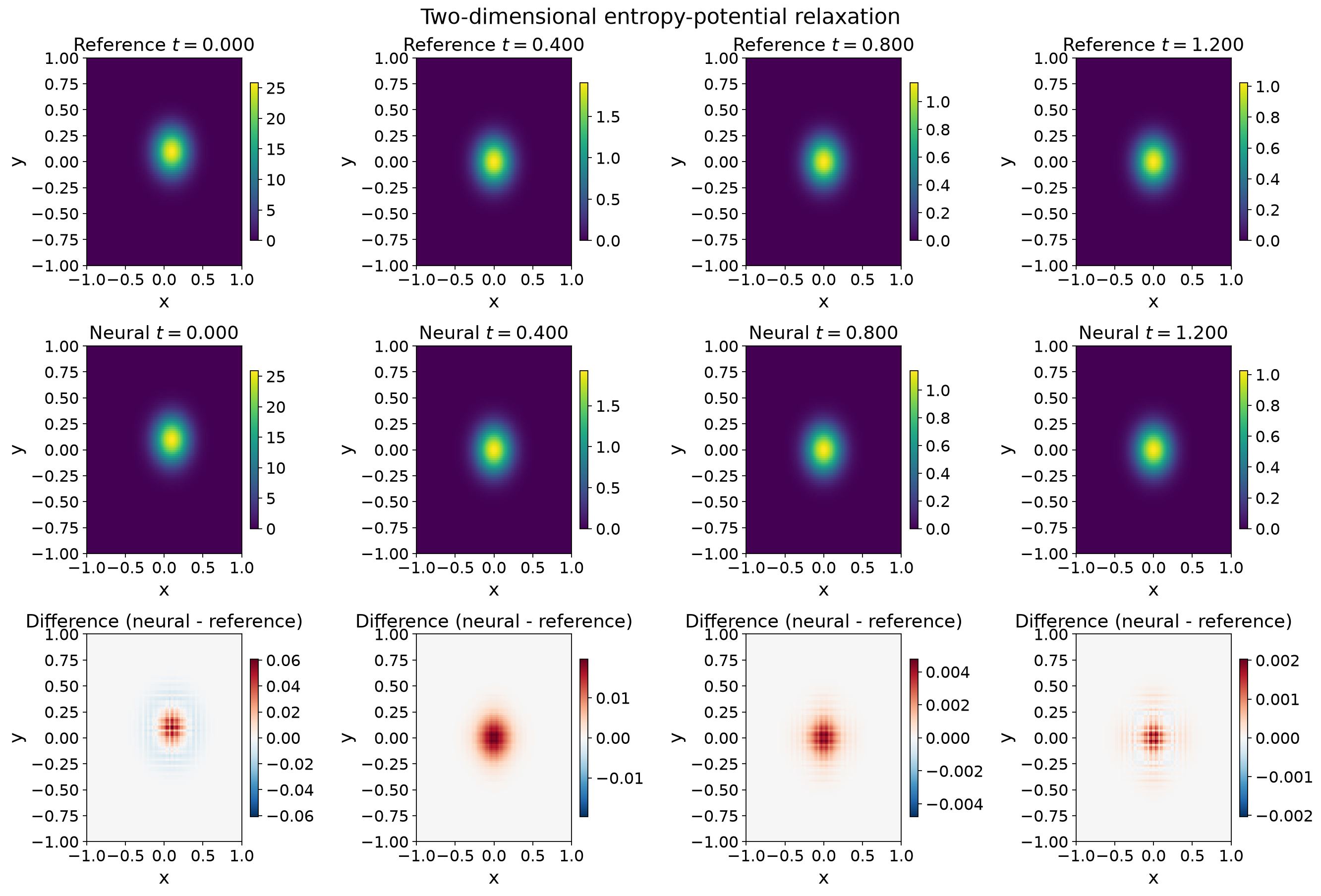}
    \par\vspace{2pt}
    \safeincludegraphics[width=0.74\linewidth]{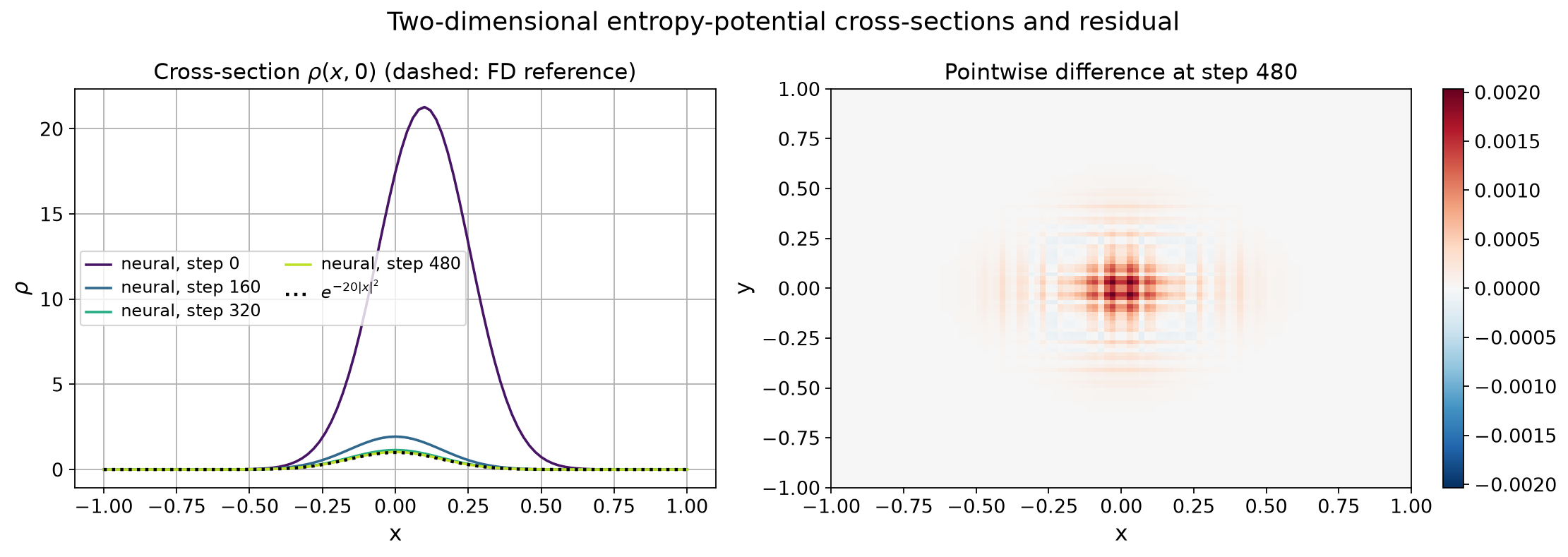}
    \par\vspace{2pt}
    \safeincludegraphics[width=0.88\linewidth]{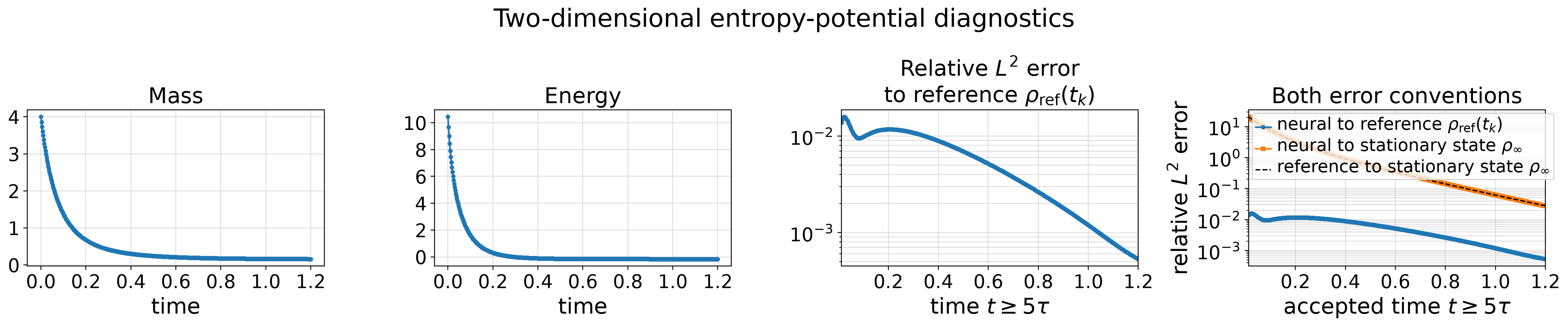}
    \caption{Two-dimensional relaxation.  Top: reference, neural, and signed error; middle: axial sections; bottom: trajectory diagnostics.}
    \label{fig:hk-stationary-2d}
\end{figure}

\input{generated/relaxation_observations.tex}

\subsection{Gaussian entropy and confinement flows}\label{sec:gaussian-experiments}\label{sec:entropy-highdim}\label{sec:fokker-planck-experiments}
For Boltzmann entropy, the $\HK$ equation is
\begin{equation*}
\partial_t\rho
=\Delta\rho+\nabla\cdot(\rho\nabla V)-4\rho(\log\rho+V).
\end{equation*}
We consider $V=0$ in eight dimensions and $V(x)=\tfrac12(x-\mu)^\top\Sigma^{-1}(x-\mu)$ in two dimensions. Gaussian initial data admit whole-space Gaussian-ansatz references for density, mass, mean, and covariance. Both learned states use the same positive FCNN as the non-Gaussian and interaction tests. Its latent log density combines quadratic and $\tanh$ features, so the approximation family is not restricted to Gaussians. These cases reproduce Gaussian dynamics within this larger approximation family; finite-box errors also include truncation.

Figure~\ref{fig:hk-entropy-flow-8d} shows entropy flow on $[-3.5,3.5]^8$, with unshown coordinates fixed at zero. It uses full-input dense maps, learned neural growth, and Gaussian importance proposals whose moments are estimated from the current neural density; Table~\ref{tab:trajectory-refinement} checks density and mass errors at fixed $T=0.15$ under time refinement. Figure~\ref{fig:hk-fp-2d} shows confinement and mean motion. Initial data and numerical settings are specified in Appendix~\ref{sec:design-notes-experiments}. FP2D uses distinct deterministic Gauss--Legendre rules; enriched checks assess integration error but do not establish optimizer convergence.

\begin{figure}[!htbp]
    \centering
    \safeincludegraphics[width=0.92\linewidth]{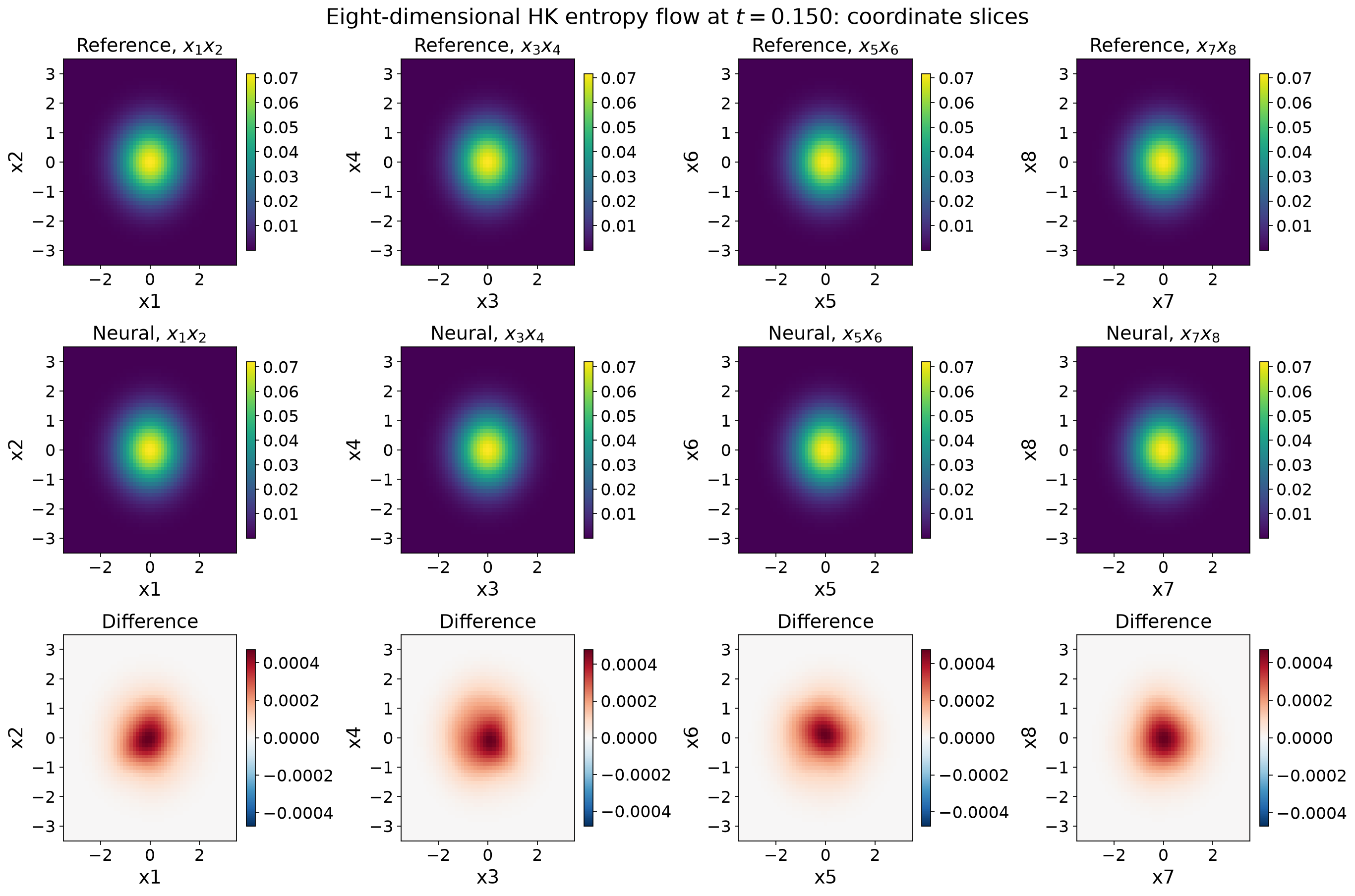}
    \par\vspace{2pt}
    \safeincludegraphics[width=0.68\linewidth]{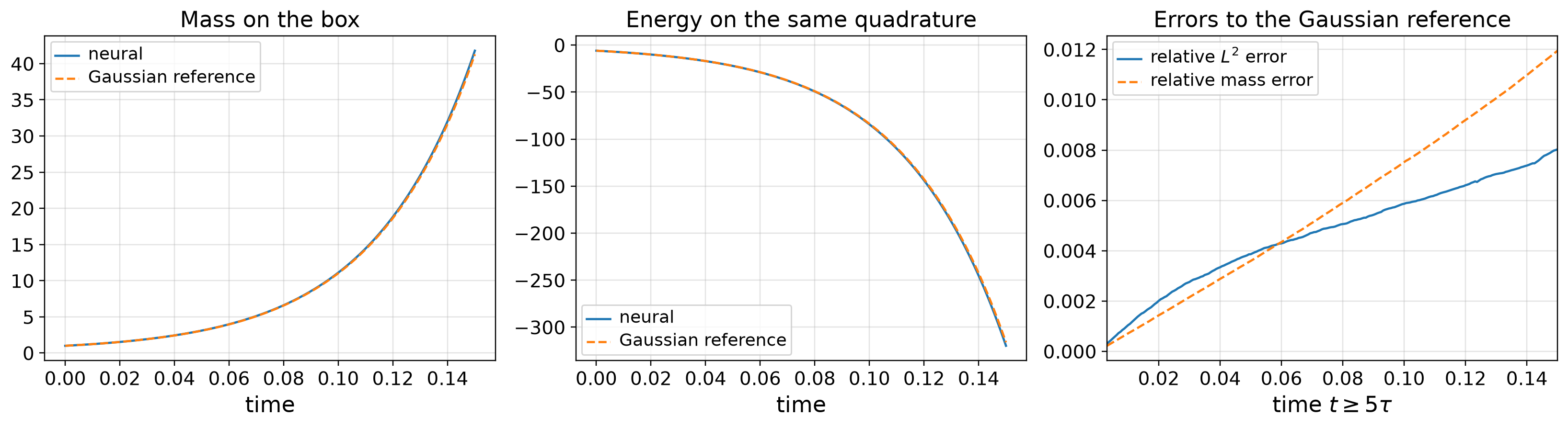}
    \caption{Eight-dimensional Gaussian entropy flow.  Top: final coordinate sections; bottom: mass, energy, and error to the whole-space Gaussian reference.}
    \label{fig:hk-entropy-flow-8d}
\end{figure}

\begin{figure}[!htbp]
    \centering
    \safeincludegraphics[width=0.86\linewidth]{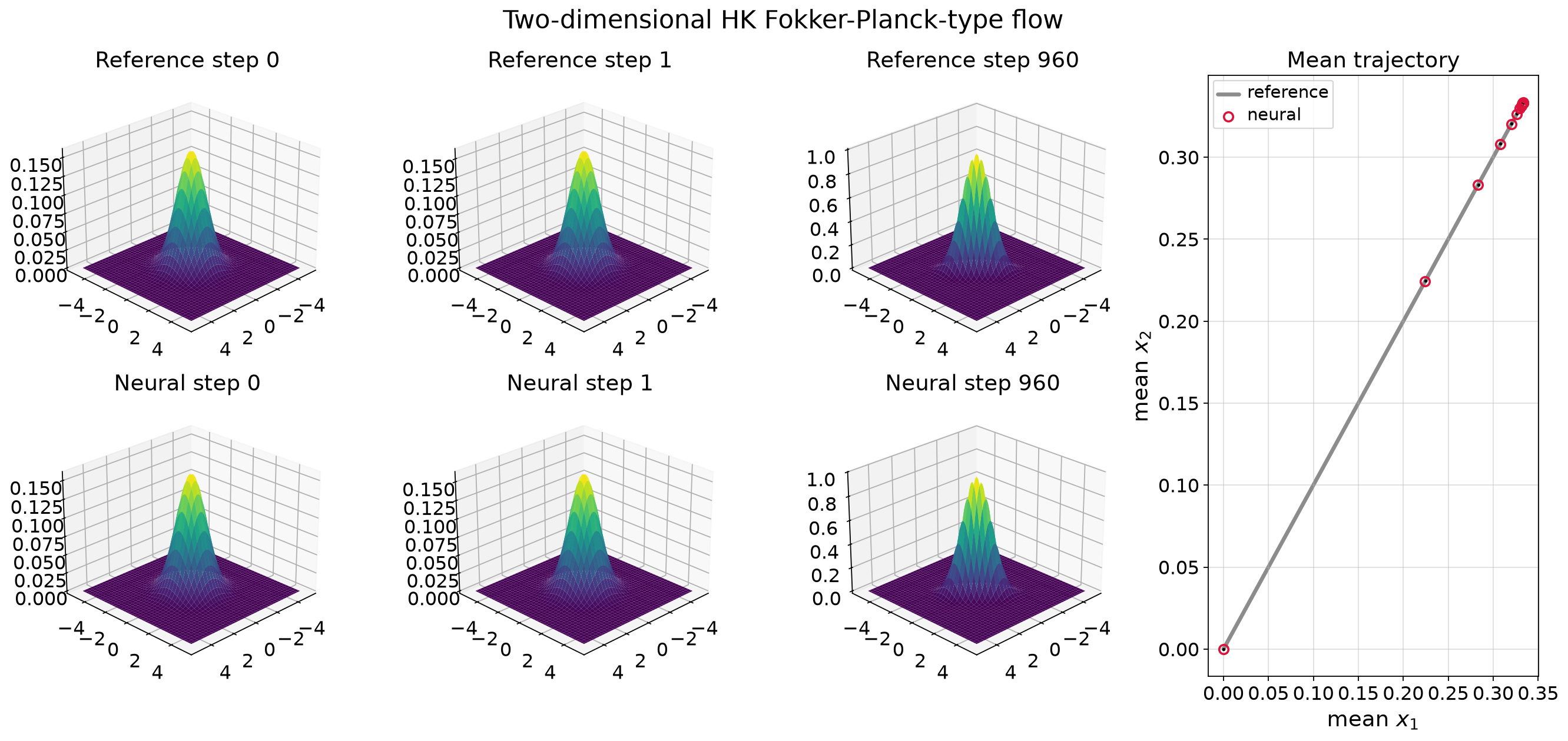}
    \par\vspace{2pt}
    \safeincludegraphics[width=0.78\linewidth]{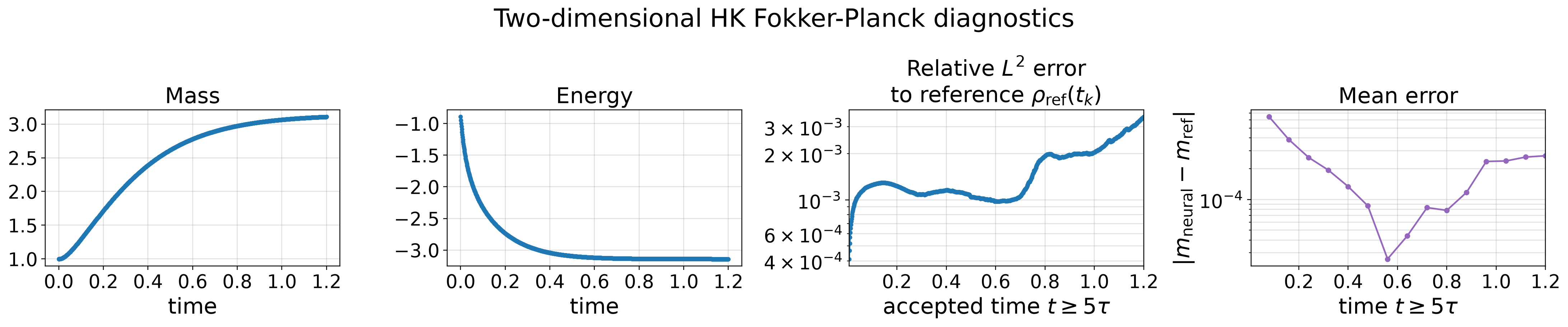}
    \caption{Two-dimensional confinement test: density and mean-path comparisons with mass, energy, density error, and mean error.}
    \label{fig:hk-fp-2d}
\end{figure}

\input{generated/gaussian_observations.tex}

\FloatBarrier
\subsection{Coercive fully implicit interaction benchmark}\label{sec:interaction-experiment}
On the flat two-dimensional torus we use
\[
W(x,y)=\frac12-\frac12\exp\!\left(-\frac{|x-y|^2}{2(0.35)^2}\right),
\]
with $V(x)=0.4+0.5(2-\cos(\pi x_1)-\cos(\pi x_2))$. Each coordinate of $x-y$ is reduced to $[-1,1)$, giving the same periodic kernel in the neural objective and the FFT reference. The shift preserves $\nabla_xW$ and ensures coercivity, but changes the reaction term through $W[\rho]$. We compare the fully implicit field $W[\rho]$ with a frozen preceding-state field, using full deterministic pair quadrature and enriched validation grids. The common initial density is specified in Appendix~\ref{sec:design-notes-experiments}.

Table~\ref{tab:interaction-ablation} compares endpoint density errors and full-periodic PDE residuals, applying the same operator to the independent reference. Fixed-horizon time refinement and separate spatial and optimizer checks are given in Appendix~\ref{sec:design-notes-experiments}. This comparison does not establish a uniform accuracy advantage for either treatment. The strong-form residual is a grid-dependent finite-difference diagnostic; the refinement checks do not establish spatial convergence.

\begin{figure}[!htbp]
    \centering
    \safeincludegraphics[width=0.98\linewidth]{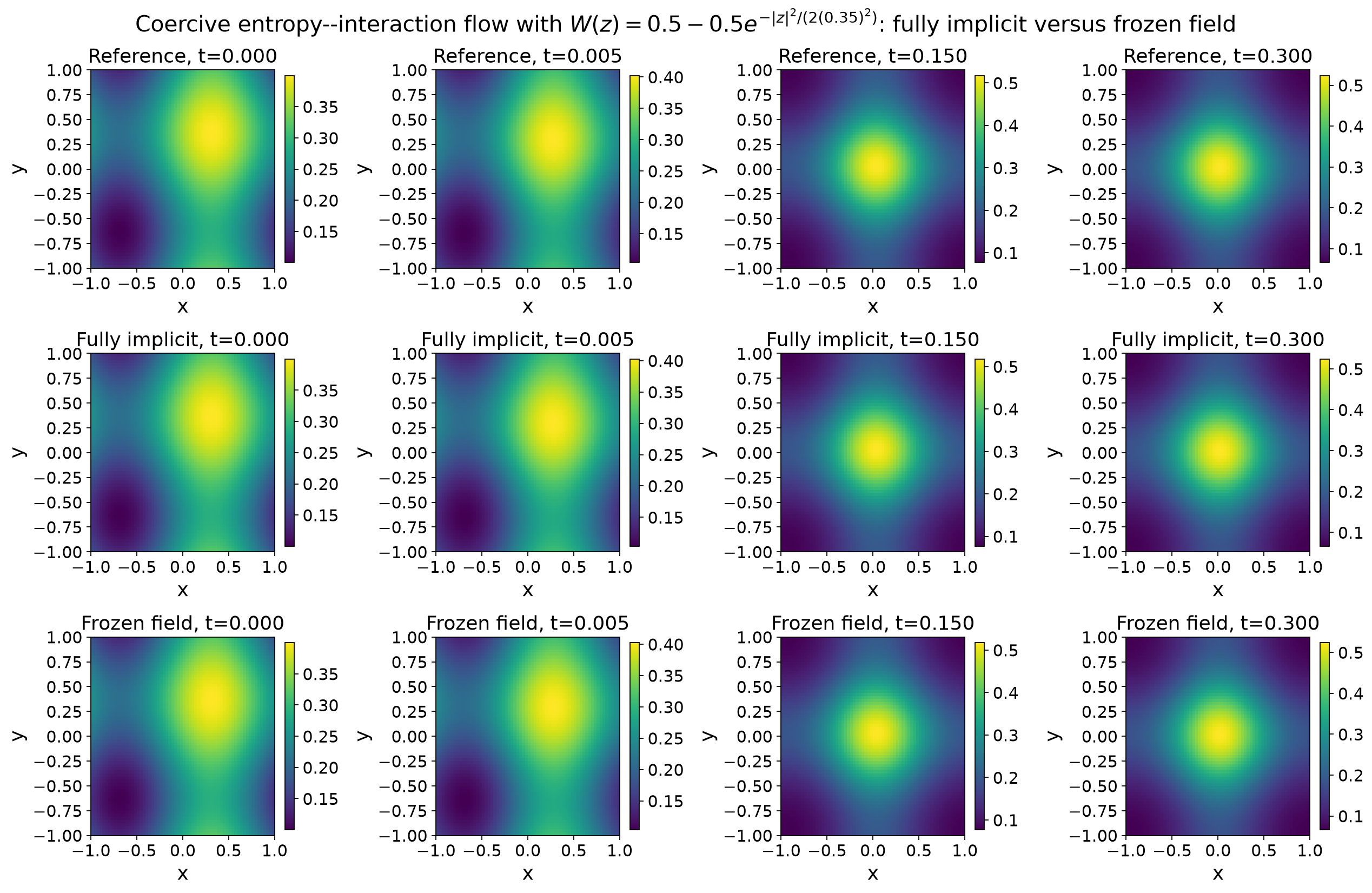}
    \caption{Interaction benchmark: periodic reference, fully implicit update, and frozen-field ablation at matched times.}
    \label{fig:hk-interaction-fields}
\end{figure}

\begin{figure}[!htbp]
    \centering
    \safeincludegraphics[width=0.96\linewidth]{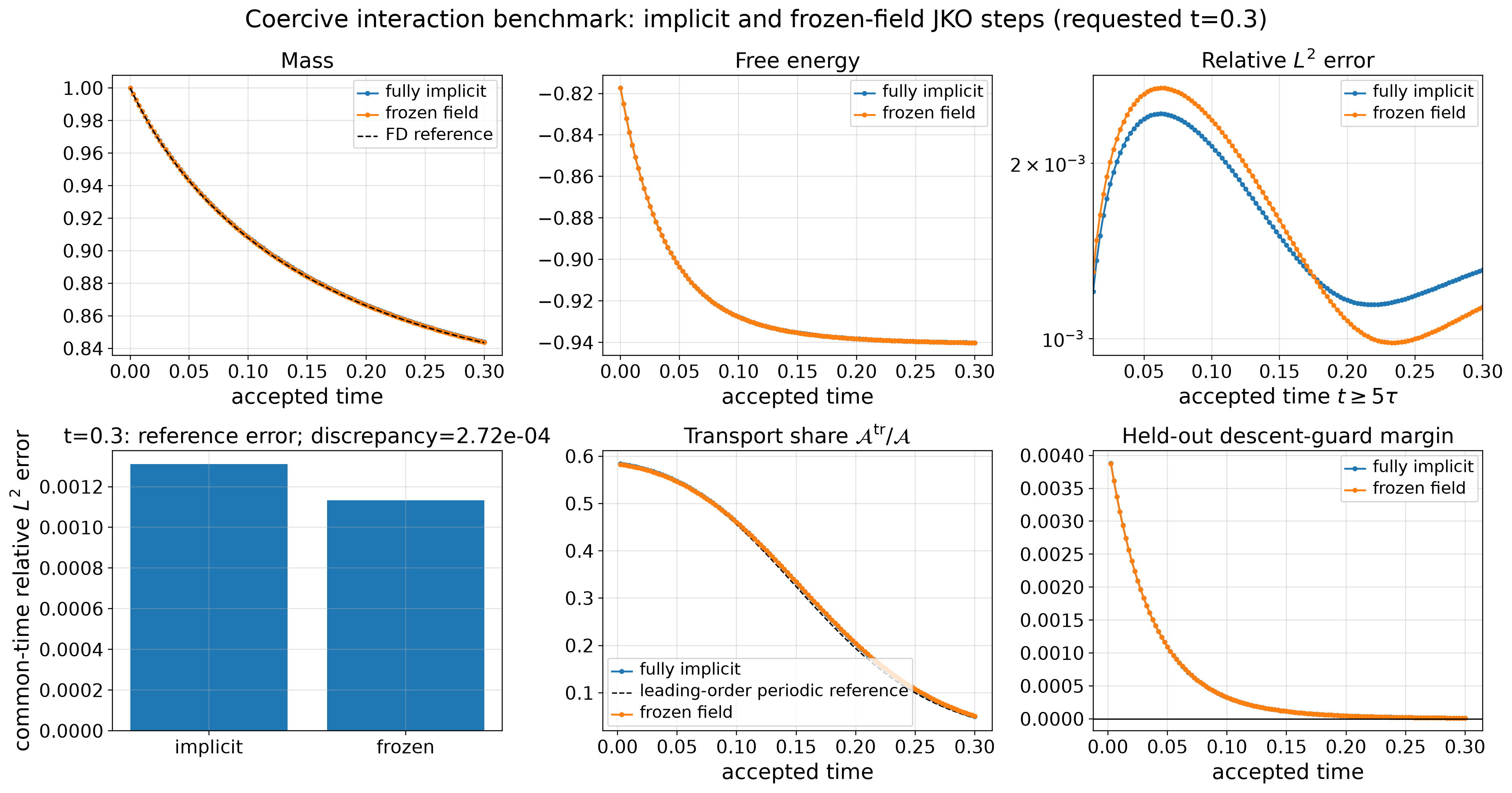}
    \caption{Interaction diagnostics: mass, common fully implicit energy, trajectory error, transport share, and guard margin.}
    \label{fig:hk-interaction-diagnostics}
\end{figure}

\begin{table}[!htbp]
\centering\scriptsize
\resizebox{\linewidth}{!}{\input{generated/interaction_table.tex}}
\caption{Implicit/frozen comparison at the common attained physical time. Density errors use the bilinearly interpolated $64^2$ reference; residuals use the same full-periodic operator for both trajectories on the refined $128^2$ grid. Because the reference is advanced using the same spatial operator, its residual primarily reflects temporal discretization. Null and fallback counts cover all attempted steps; the transport share is that of the last update reaching the common time.}
\label{tab:interaction-ablation}
\end{table}

\FloatBarrier
\subsection{Finite-grid primal--dual certificate}\label{sec:certificate-experiment}
For a one-dimensional Boltzmann step with $W=0$, we evaluate the optimized pair's density $u_\theta$ by exact map inversion and change of variables. The discrete candidate consists of its values on the fixed 65-node grid; source, potential, and pair solver are specified in Appendix~\ref{sec:design-notes-experiments}. A nonnegative entropy--transport coupling gives a primal upper value for this same candidate. A grid dual pair is shifted until
\[
\zeta_i+\xi_j\le c_{\HK}(x_i,x_j)
\quad\hbox{for every grid pair }(i,j),
\]
and hence gives a certified lower value.  The gap $R_h$ and a rigorous mass upper bound yield the discrete analogue of Corollary~\ref{cor:ap-boltzmann-L1}.  The radius applies to the exact finite-grid minimizer; the comparison with a numerical grid iterate is diagnostic only, and its own primal--dual gap is reported separately.  No continuum certificate is claimed.  The discrete bounds are evaluated in floating-point arithmetic with a numerical feasibility margin; they are not interval-arithmetic proofs of the reported decimal values.

\begin{figure}[!htbp]
    \centering
    \safeincludegraphics[width=0.98\linewidth]{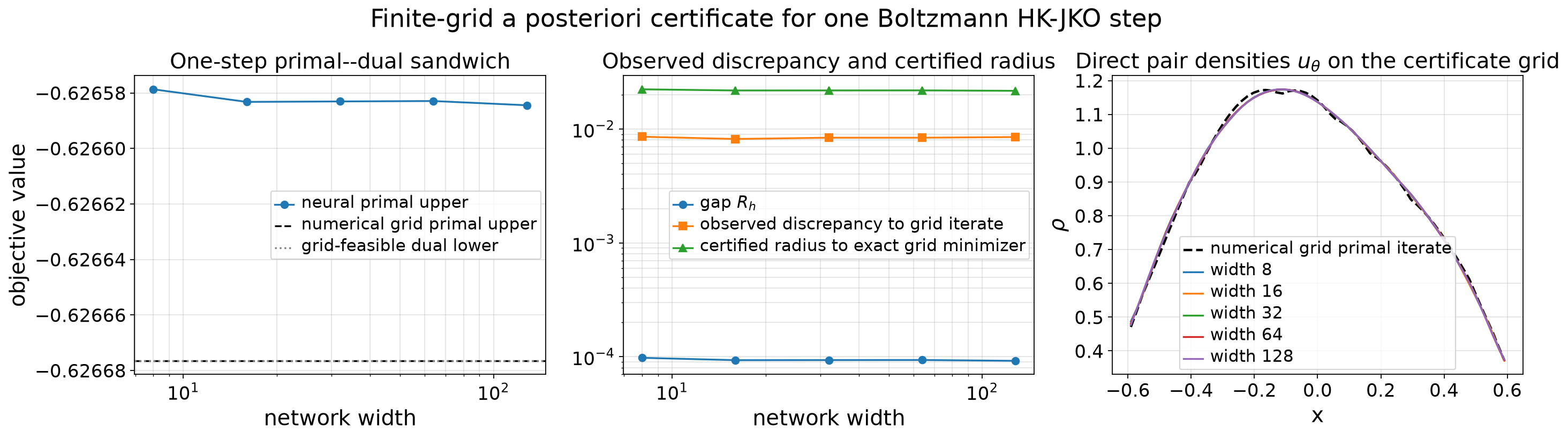}
    \caption{Finite-grid certificate across network widths, with depth fixed. Left: primal and dual values; center: gap, observed error, and certified radius; right: directly evaluated $u_\theta$ and the numerical grid reference.}
    \label{fig:certificate-1d}
\end{figure}

\begin{table}[!htbp]
\centering\scriptsize
\resizebox{\linewidth}{!}{\input{generated/certificate_table.tex}}
\caption{Finite-grid certificate. Feasibility is checked over all grid pairs. The scaled metric gap measures coupling-solve accuracy with the candidate density fixed.}
\label{tab:certificate-1d}
\end{table}

\input{generated/certificate_observations.tex}

\input{generated/certificate_reference_results.tex}

\FloatBarrier
\endgroup

%% file: generated/trajectory_overview.tex
All main experiments use the common neural solver. Time steps, physical horizons and optimizer budgets are reported in Appendix~\ref{sec:design-notes-experiments}.

%% file: generated/tau_scaling_table.tex
\begin{tabular}{@{}cccccc@{}}
\toprule
$\tau$ & $\|T-\mathrm{id}\|_{L^2(\rho)}$ & $\|q^2-1\|_{L^2(\rho)}$ & velocity error & reaction error & $\min J_T$ \\
\midrule
0.020 & 0.014 & 0.050 & 0.234 & 0.146 & 0.958 \\
0.010 & 8.169e-03 & 0.027 & 0.133 & 0.081 & 0.975 \\
5.000e-03 & 4.370e-03 & 0.014 & 0.071 & 0.043 & 0.987 \\
2.500e-03 & 2.265e-03 & 7.014e-03 & 0.037 & 0.022 & 0.993 \\
1.250e-03 & 1.153e-03 & 3.542e-03 & 0.019 & 0.011 & 0.996 \\
\bottomrule
\end{tabular}

%% file: generated/one_step_observations.tex
The relative strong PDE residual decreases from $0.0176$ at $\tau=0.02$ to $0.000743$ at $\tau=0.000625$. In the independent pair-scaling test, the smallest step $\tau=0.00125$ gives velocity and reaction field errors $0.019$ and $0.0113$, respectively. These finite-resolution observations do not by themselves establish an asymptotic rate.

%% file: generated/restriction_table.tex
\begin{tabular}{@{}lccccccc@{}}
\toprule
Configuration & $\mathcal F$ & energy monotone? & rel. $L^2$ & $\overline{\mathcal A^{\rm tr}/\mathcal A}$ & reference & common $t$ & reached $t$ \\
\midrule
full pair $(T_\theta,q_\theta)$ & -0.395 & yes & 1.308e-03 & 0.150 & 0.145 & 0.800 & 0.800 \\
pure growth, $T_\theta=\mathrm{id}$ & -0.395 & yes & 0.019 & 0 & 0.109 & 0.800 & 0.800 \\
transport only, $q_\theta=1$ & 1.237 & yes & 3.726 & 1.000 & 0.316 & 0.800 & 0.800 \\
\bottomrule
\end{tabular}

%% file: generated/relaxation_observations.tex
In one dimension, the final relative $L^2$ error to the finite-volume trajectory is $2.01\times10^{-5}$ at $t=2.5$. In two dimensions, the final relative $L^2$ error to the finite-volume trajectory is $0.000532$ at $t=1.2$. At the common ablation time, full-pair, pure-growth and transport-only errors are $0.00131$, $0.0193$, $3.73$, respectively. The full pair improves trajectory agreement even though endpoint energy alone does not rank these constrained dynamics by accuracy.

%% file: generated/gaussian_observations.tex
For eight-dimensional entropy flow, the recorded endpoint relative $L^2$ density error is $0.00803$ at $t=0.15$. For FP2D, the recorded endpoint relative $L^2$ density error is $0.00348$ at $t=1.2$. These values use the reported error quadratures and whole-space Gaussian references; they include time discretization, finite-box, learned-pair and density-fit errors.

%% file: generated/interaction_table.tex
\begin{tabular}{@{}lccccccccc@{}}
\toprule
interaction treatment & rel. $L^2$ & mass & true energy & transport share & residual (neural/ref.) & common $t$ & reached $t$ & null & fallback \\
\midrule
fully implicit & 1.312e-03 & 0.844 & -0.940 & 0.050 & 0.021/4.787e-03 & 0.300 & 0.300 & 0 & 0 \\
frozen field & 1.133e-03 & 0.844 & -0.940 & 0.051 & 0.021/4.787e-03 & 0.300 & 0.300 & 0 & 0 \\
\bottomrule
\end{tabular}

%% file: generated/certificate_table.tex
\begin{tabular}{@{}cccccccc@{}}
\toprule
width & primal upper & dual lower & gap $R_h$ & observed $L^1$ & certified radius & metric gap$/2\tau$ & feasibility violation \\
\midrule
8 & -0.627 & -0.627 & 9.783e-05 & 8.534e-03 & 0.022 & 1.261e-09 & -1.000e-10 \\
16 & -0.627 & -0.627 & 9.330e-05 & 8.159e-03 & 0.022 & 1.378e-10 & -1.000e-10 \\
32 & -0.627 & -0.627 & 9.346e-05 & 8.371e-03 & 0.022 & 5.154e-10 & -1.000e-10 \\
64 & -0.627 & -0.627 & 9.360e-05 & 8.370e-03 & 0.022 & 1.779e-10 & -1.000e-10 \\
128 & -0.627 & -0.627 & 9.208e-05 & 8.485e-03 & 0.022 & 2.226e-10 & -1.000e-10 \\
\bottomrule
\end{tabular}

%% file: generated/certificate_observations.tex
Across the tested widths, the largest relative growth-equation RMS and maximum residuals are $9.02\times10^{-5}$ and $0.000583$. The observed $L^1$ distance to the numerical grid iterate is at most $0.00853$, while the certified radii range from $0.0216$ to $0.0223$. The radii remain valid across all widths; the measured errors and gaps need not decrease monotonically with width.

%% file: generated/certificate_reference_results.tex
For the numerical grid reference itself, the primal upper value is $-0.6267$, the feasible dual lower value is $-0.6267$, and the gap is $2.746\times10^{-9}$. The same discrete error estimate gives it an $L^1$ radius $0.000118$. These computed bounds, rather than the optimizer iteration budget alone, describe reference accuracy. Dual feasibility is checked with an outward numerical margin in floating-point arithmetic.

%% file: sections/experiment_details.tex
% HKFLOW_PRESENTATION: compact-v1
% HKFLOW_JCP_REVISION: stored-diagnostics-v1
\subsection{Experimental implementation and refinement checks}\label{sec:design-notes-experiments}
\paragraph{Parametrization and optimization.}
Each map composes two full-input dense residual FCNN blocks with certified contraction bounds. Scalar fields use trainable mixed features
\[
F_\eta(z)=a_0+a\cdot z+\sum_{j=1}^{m}
\bigl[b_j(w_j\cdot z+c_j)^2+d_j\tanh(w_j\cdot z+c_j)\bigr].
\]
The density is represented as $\exp(\operatorname{clip}(F_\eta,-300,30))$, replacing $F_\eta$ by $F_\eta-V$ for no-flux potential problems, with a fitted mass shift; a separate growth network uses $q_\eta=\exp(4\tanh(\tau F_\eta/4))$. Cosine inputs enforce the no-flux chemical-potential condition. Interaction uses four periodic inputs (one sine/cosine pair per coordinate), width 256, and width 512 in the capacity check.

Trajectory pairs, including the one-step PDE test, are initialized by fitting output coefficients to the source predictor $T\approx\mathrm{id}-\tau\nabla g_\rho$, $\log q\approx-2\tau g_\rho$ of Remark~\ref{rem:tau-scaled-ansatz} before joint training; independent pair-scaling and certificate tests start from $(\mathrm{id},1)$.
Joint L-BFGS with base budget $N$ is followed by a conditioned stage of budget $\max(20,\min(400,\lfloor N/2\rfloor))$; when growth is learned, each stage includes growth polishing of the same objective with budget $\max(50,\min(400,\lfloor N/3\rfloor))$.
Recursive propagation fits the log density with up to 40 L-BFGS iterations for hidden parameters, variable projection for output coefficients, and sampled-mass restoration; the certificate uses the raw inverse-map pushforward.
Across the six trajectories in Table~\ref{tab:jcp-projection-work},
the per-fit Frobenius change in the density hidden-input weight matrix
was at most $2\times10^{-5}$, and was zero at every interaction step.
This diagnostic excludes hidden biases and output coefficients.
Tables~\ref{tab:numerical-experiment-protocol} and~\ref{tab:density-family} give the base settings and pair-training rules.

The entropy objective uses the integration-by-parts identity
\begin{equation}\label{eq:jcp-ibp-correction}
C_\rho(T):=\int_\Omega\rho\bigl(\operatorname{div}v+\nabla\log\rho\cdot v\bigr)\,dx=0,
\qquad v=T-\mathrm{id},
\end{equation}
for a smooth positive source $\rho$, with $v\cdot n=0$ or a periodic lift. Its quadrature estimate is added with coefficient $+1$ in training and validation, using the current source's score. This changes finite-quadrature values and gradients, but not the continuum objective.

\input{generated/execution_protocol.tex}

\begin{table}[H]
\centering\scriptsize
\setlength{\tabcolsep}{3pt}
\resizebox{\linewidth}{!}{\selectedtable{generated/experiment_protocol_table.tex}{1,2,4,5,6,8,9,11,12}}
\caption{Executed settings. $N$ excludes conditioned and growth-polishing stages. Certificate map widths are $8,16,32,64,128$; a missing density width denotes a raw pushforward.}
\label{tab:numerical-experiment-protocol}
\end{table}

\paragraph{Data and references.}
Gaussian sources are $\mathcal N(0,0.45^2I_8)$ for entropy flow and $\mathcal N(0,I_2)$ for FP2D, restricted to their boxes. For FP2D, $\mu=(1/3,1/3)$ and $\Sigma=\frac18\left(\begin{smallmatrix}5&-3\\-3&5\end{smallmatrix}\right)$.
Projection/error evaluation uses $46^2/90^2$ Gauss--Legendre nodes for FP2D and $4096/32768$ importance draws for entropy 8D; 2D relaxation uses $91^2$ projection nodes.

The certificate uses $\rho_0=0.75+0.35e^{-8(x+0.18)^2}$, $V=5x^2+0.25x$, 512 training nodes (also used in the small-time pair tests) and a 65-node certificate grid. Relaxation references use 600 cells in 1D and $128^2$ in 2D, with $\Delta t=10^{-4}$. The interaction source is $\rho_0=\tfrac14[1+0.38\cos(\pi(x_1-0.30))+0.22\sin(\pi(x_2+0.15))]$; its reference includes potential and interaction drift.

The one-step PDE diagnostic uses 129 initial-fit nodes, 257 pair nodes, 1024 projection nodes and 601 uniform evaluation points, checked against 1201 points.

\paragraph{Refinement and propagation checks.}
Held-out rules check descent, density fitting and map admissibility; the one-step growth-stationarity tolerances are $10^{-4}$ in relative RMS and $10^{-3}$ in maximum residual. Table~\ref{tab:trajectory-refinement} reports independently optimized time refinements at fixed horizons.

\begin{table}[H]
\centering\scriptsize
\resizebox{\linewidth}{!}{\input{generated/trajectory_refinement_table.tex}}
\caption{Fixed-horizon time refinement. Interaction errors and residuals use the refined reference grid and the same periodic operator for both trajectories.}
\label{tab:trajectory-refinement}
\end{table}

Table~\ref{tab:jcp-projection-work} reports held-out fit errors relative to the unprojected endpoint and optimizer work in closure evaluations, not wall-clock time. All six stored trajectories satisfy $E_{k+1}-E_k\le10^{-7}(1+|E_k|)$ for the propagated states; these checks do not bound accumulated projection error.

\begin{table}[H]
\centering\small
\setlength{\tabcolsep}{3pt}
\input{generated/jcp_projection_work_table.tex}
\caption{Propagation and work diagnostics. Fit and mass errors are maxima over steps; work columns are mean objective-closure evaluations per step, including pair stages and density fitting. Pair and fit closures have different costs.}
\label{tab:jcp-projection-work}
\end{table}

Interaction residuals use the same periodic finite-difference operator
for learned and reference states, normalized by the sum of time-increment
and PDE-activity norms. Because the reference is advanced using this
spatial operator, its residual primarily reflects temporal discretization.
Applying the operator to the learned continuum states also introduces
spatial truncation error. At $T=0.3$, for the implicit state on the
$128^2$ grid, diffusion and drift each have $L^2$ norm about $2.6$,
whereas their sum has norm about $0.03$; the residual numerator has
$L^2$ norm about $0.14\%$ of that of either individual transport term.
The nearest-image kernel has derivative kinks, so smooth-kernel
conclusions are not invoked.
\begin{table}[H]
\centering\small
\input{generated/jcp_residual_grid_table.tex}
\caption{Interaction strong residuals at $T=0.3$, $\tau=0.0025$. All grids evaluate the same final two learned states with the full-periodic operator, without retraining.}
\label{tab:jcp-residual-grid}
\end{table}
On the $64^2$ grid the same operator gives a residual of $0.048$ for
the implicit state, so the residual first decreases and then increases
under grid refinement. This behavior is consistent with $O(h^2)$
truncation on coarse grids and with small pointwise evaluation errors
amplified by second differences in proportion to $h^{-2}$ on fine grids;
no grid-converged strong-form residual is established.
The $64^2$ reference differs from the $128^2$ reference at quarter
time step by $6.957\times10^{-4}$ in relative $L^2$ on the coarse grid,
which limits the density comparison at this level.

\FloatBarrier

%% file: generated/execution_protocol.tex
Trainable networks use \texttt{mps} \texttt{torch.float32}; quadrature and linear solves use CPU float64. The seed is 1234.

%% file: generated/trajectory_refinement_table.tex
\begin{tabular}{@{}lcccccc@{}}
\toprule
Experiment & $\tau$ & attained $T$ & rel. $L^2$ & rel. mass error & flow residual & reference residual \\
\midrule
Entropy $8$d & 1.250e-03 & 0.150 & 0.025 & 0.022 & -- & -- \\
Entropy $8$d & 6.250e-04 & 0.150 & 8.043e-03 & 0.012 & -- & -- \\
Implicit interaction & 2.500e-03 & 0.300 & 5.212e-04 & -- & 0.021 & 4.787e-03 \\
Frozen interaction & 2.500e-03 & 0.300 & 4.211e-04 & -- & 0.021 & 4.787e-03 \\
Implicit interaction & 5.000e-03 & 0.300 & 1.051e-03 & -- & 0.021 & 9.533e-03 \\
Frozen interaction & 5.000e-03 & 0.300 & 9.422e-04 & -- & 0.020 & 9.533e-03 \\
\bottomrule
\end{tabular}

%% file: generated/jcp_projection_work_table.tex
\begin{tabular}{@{}lrrrrr@{}}
\toprule
Case & steps & fit rel.\ $L^2$ & mass error & pair evals./step & fit evals./step \\
\midrule
Relaxation 1D & 1000 & 1.72e-05 & 5.29e-06 & 381.5 & 2.1 \\
Relaxation 2D & 480 & 4.78e-05 & 7.79e-06 & 693.3 & 2.8 \\
Entropy 8D & 240 & 1.92e-04 & 3.91e-05 & 499.5 & 5.0 \\
FP2D & 960 & 1.89e-04 & 1.22e-04 & 635.9 & 7.2 \\
Implicit interaction & 120 & 1.26e-05 & 7.12e-07 & 800.0 & 2.0 \\
Frozen interaction & 120 & 1.25e-05 & 5.96e-07 & 823.7 & 2.0 \\
\bottomrule
\end{tabular}

%% file: generated/jcp_residual_grid_table.tex
\begin{tabular}{@{}lrrr@{}}
\toprule
Treatment & $128^2$ & $256^2$ & $512^2$ \\
\midrule
Implicit & 0.02069 & 0.02370 & 0.06650 \\
Frozen & 0.02123 & 0.02446 & 0.06700 \\
\bottomrule
\end{tabular}

%% file: main.bbl
\begin{thebibliography}{31}
\providecommand{\natexlab}[1]{#1}
\providecommand{\url}[1]{\texttt{#1}}
\expandafter\ifx\csname urlstyle\endcsname\relax
  \providecommand{\doi}[1]{doi: #1}\else
  \providecommand{\doi}{doi: \begingroup \urlstyle{rm}\Url}\fi

\bibitem[Jordan et~al.(1998)Jordan, Kinderlehrer, and Otto]{JKO98}
R.~Jordan, D.~Kinderlehrer, and F.~Otto.
\newblock {The variational formulation of the Fokker--Planck equation}.
\newblock \emph{SIAM J. Math. Anal.}, 29:\penalty0 1--17, 1998.

\bibitem[Ambrosio et~al.(2008)Ambrosio, Gigli, and Savar\'e]{AGS08}
L.~Ambrosio, N.~Gigli, and G.~Savar\'e.
\newblock \emph{{Gradient Flows in Metric Spaces and in the Space of
  Probability Measures}}.
\newblock Birkh\"auser, Basel, 2nd edition, 2008.

\bibitem[Kondratyev et~al.(2016)Kondratyev, Monsaingeon, and Vorotnikov]{KMV16}
S.~Kondratyev, L.~Monsaingeon, and D.~Vorotnikov.
\newblock {A new optimal transport distance on the space of finite Radon
  measures}.
\newblock \emph{Adv. Differential Equations}, 21:\penalty0 1117--1164, 2016.

\bibitem[Chizat et~al.(2018{\natexlab{a}})Chizat, Peyr\'e, Schmitzer, and
  Vialard]{CPSV2018}
L.~Chizat, G.~Peyr\'e, B.~Schmitzer, and F.-X. Vialard.
\newblock {Unbalanced optimal transport: dynamic and Kantorovich formulations}.
\newblock \emph{J. Funct. Anal.}, 274:\penalty0 3090--3123, 2018{\natexlab{a}}.

\bibitem[Chizat et~al.(2018{\natexlab{b}})Chizat, Peyr\'e, Schmitzer, and
  Vialard]{CPSV18}
L.~Chizat, G.~Peyr\'e, B.~Schmitzer, and F.-X. Vialard.
\newblock {An interpolating distance between optimal transport and Fisher--Rao
  metrics}.
\newblock \emph{Found. Comput. Math.}, 18:\penalty0 1--44, 2018{\natexlab{b}}.

\bibitem[Liero et~al.(2016)Liero, Mielke, and Savar\'e]{LMS16}
M.~Liero, A.~Mielke, and G.~Savar\'e.
\newblock {Optimal transport in competition with reaction: the
  Hellinger--Kantorovich distance and geodesic curves}.
\newblock \emph{SIAM J. Math. Anal.}, 48:\penalty0 2869--2911, 2016.

\bibitem[Liero et~al.(2018)Liero, Mielke, and Savar\'e]{LMS18}
M.~Liero, A.~Mielke, and G.~Savar\'e.
\newblock {Optimal entropy--transport problems and a new Hellinger--Kantorovich
  distance between positive measures}.
\newblock \emph{Invent. Math.}, 211:\penalty0 969--1117, 2018.

\bibitem[Gallou\"et and Monsaingeon(2017)]{GM17}
T.~O. Gallou\"et and L.~Monsaingeon.
\newblock {A JKO splitting scheme for Kantorovich--Fisher--Rao gradient flows}.
\newblock \emph{SIAM J. Math. Anal.}, 49:\penalty0 1100--1130, 2017.

\bibitem[Flei{\ss}ner(2021)]{Fleissner21}
F.~Flei{\ss}ner.
\newblock {A minimizing movement approach to a class of scalar
  reaction--diffusion equations}.
\newblock \emph{ESAIM Control Optim. Calc. Var.}, 27:\penalty0 18, 2021.
\newblock \doi{10.1051/cocv/2020090}.

\bibitem[Di~Marino and Chizat(2020)]{DiMarinoChizat20}
S.~Di~Marino and L.~Chizat.
\newblock {A tumor growth model of Hele--Shaw type as a gradient flow}.
\newblock \emph{ESAIM Control Optim. Calc. Var.}, 26:\penalty0 103, 2020.
\newblock \doi{10.1051/cocv/2020019}.

\bibitem[Liero et~al.(2023)Liero, Mielke, and Savar\'e]{LMS23}
M.~Liero, A.~Mielke, and G.~Savar\'e.
\newblock {Fine properties of geodesics and geodesic $\lambda$-convexity for
  the Hellinger--Kantorovich distance}.
\newblock \emph{Arch. Ration. Mech. Anal.}, 247:\penalty0 112, 2023.
\newblock \doi{10.1007/s00205-023-01941-1}.

\bibitem[Gallou\"et et~al.(2025)Gallou\"et, Ghezzi, and Vialard]{GGV25}
T.~Gallou\"et, R.~Ghezzi, and F.-X. Vialard.
\newblock {Regularity theory and geometry of unbalanced optimal transport}.
\newblock \emph{J. Funct. Anal.}, 289\penalty0 (7):\penalty0 111042, 2025.
\newblock \doi{10.1016/j.jfa.2025.111042}.

\bibitem[Laschos and Mielke(2026)]{LM26}
V.~Laschos and A.~Mielke.
\newblock {Evolutionary variational inequalities on the Hellinger--Kantorovich
  and spherical Hellinger--Kantorovich spaces}.
\newblock \emph{Commun. Partial Differential Equations}, 51\penalty0
  (2--3):\penalty0 193--235, 2026.
\newblock \doi{10.1080/03605302.2026.2650715}.

\bibitem[Mokrov et~al.(2021)Mokrov, Korotin, Li, Genevay, Solomon, and
  Burnaev]{Mokrov21}
P.~Mokrov, A.~Korotin, L.~Li, A.~Genevay, J.~M. Solomon, and E.~Burnaev.
\newblock {Large-scale Wasserstein gradient flows}.
\newblock In \emph{Advances in Neural Information Processing Systems 34}, 2021.

\bibitem[Park et~al.(2023)Park, Kim, Son, and Hwang]{PKSH23}
M.~S. Park, C.~Kim, H.~Son, and H.~J. Hwang.
\newblock {The deep minimizing movement scheme}.
\newblock \emph{J. Comput. Phys.}, 494:\penalty0 112518, 2023.
\newblock \doi{10.1016/j.jcp.2023.112518}.

\bibitem[Chizat et~al.(2018{\natexlab{c}})Chizat, Peyr\'e, Schmitzer, and
  Vialard]{ChizatScaling18}
L.~Chizat, G.~Peyr\'e, B.~Schmitzer, and F.-X. Vialard.
\newblock {Scaling algorithms for unbalanced optimal transport problems}.
\newblock \emph{Math. Comp.}, 87:\penalty0 2563--2609, 2018{\natexlab{c}}.
\newblock \doi{10.1090/mcom/3303}.

\bibitem[Carrillo et~al.(2022)Carrillo, Craig, Wang, and Wei]{CarrilloCraig22}
J.~A. Carrillo, K.~Craig, L.~Wang, and C.~Wei.
\newblock {Primal dual methods for Wasserstein gradient flows}.
\newblock \emph{Found. Comput. Math.}, 22:\penalty0 389--443, 2022.
\newblock \doi{10.1007/s10208-021-09503-1}.

\bibitem[Ponnoprat et~al.(2026)Ponnoprat, Isobe, and Imaizumi]{PII26}
D.~Ponnoprat, N.~Isobe, and M.~Imaizumi.
\newblock {Minimax optimal estimation of transport-growth pairs in unbalanced
  optimal transport}.
\newblock arXiv:2605.08705, 2026.
\newblock \doi{10.48550/arXiv.2605.08705}.

\bibitem[Fan et~al.(2022)Fan, Zhang, Taghvaei, and Chen]{Fan22}
J.~Fan, Q.~Zhang, A.~Taghvaei, and Y.~Chen.
\newblock {Variational Wasserstein gradient flow}.
\newblock In \emph{Proceedings of the 39th International Conference on Machine
  Learning}, volume 162 of \emph{PMLR}, pages 6185--6215, 2022.

\bibitem[Xu et~al.(2023)Xu, Cheng, and Xie]{Xu23}
C.~Xu, X.~Cheng, and Y.~Xie.
\newblock {Normalizing flow neural networks by JKO scheme}.
\newblock In \emph{Advances in Neural Information Processing Systems 36}, pages
  47379--47405, 2023.

\bibitem[Lee et~al.(2024)Lee, Wang, and Li]{Lee24}
W.~Lee, L.~Wang, and W.~Li.
\newblock {Deep JKO: Time-implicit particle methods for general nonlinear
  gradient flows}.
\newblock \emph{J. Comput. Phys.}, 514:\penalty0 113187, 2024.
\newblock \doi{10.1016/j.jcp.2024.113187}.

\bibitem[Yang and Uhler(2019)]{YangUhler19}
K.~D. Yang and C.~Uhler.
\newblock {Scalable unbalanced optimal transport using generative adversarial
  networks}.
\newblock In \emph{International Conference on Learning Representations}, 2019.

\bibitem[L\"ubeck et~al.(2022)L\"ubeck, Bunne, Gut, {Sarabia del Castillo},
  Pelkmans, and Alvarez-Melis]{Lubeck22}
F.~L\"ubeck, C.~Bunne, G.~Gut, J.~{Sarabia del Castillo}, L.~Pelkmans, and
  D.~Alvarez-Melis.
\newblock {Neural unbalanced optimal transport via cycle-consistent
  semi-couplings}.
\newblock arXiv:2209.15621, 2022.
\newblock \doi{10.48550/arXiv.2209.15621}.

\bibitem[Eyring et~al.(2024)Eyring, Klein, Uscidda, Palla, Kilbertus, Akata,
  and Theis]{Eyring24}
L.~Eyring, D.~Klein, T.~Uscidda, G.~Palla, N.~Kilbertus, Z.~Akata, and
  F.~Theis.
\newblock {Unbalancedness in neural Monge maps improves unpaired domain
  translation}.
\newblock In \emph{International Conference on Learning Representations}, 2024.

\bibitem[Sha et~al.(2024)Sha, Qiu, Zhou, and Nie]{TIGON24}
Y.~Sha, Y.~Qiu, P.~Zhou, and Q.~Nie.
\newblock {Reconstructing growth and dynamic trajectories from single-cell
  transcriptomics data}.
\newblock \emph{Nat. Mach. Intell.}, 6:\penalty0 25--39, 2024.
\newblock \doi{10.1038/s42256-023-00763-w}.

\bibitem[Choi et~al.(2024)Choi, Choi, and Kang]{CCK24}
J.~Choi, J.~Choi, and M.~Kang.
\newblock {Scalable Wasserstein gradient flow for generative modeling through
  unbalanced optimal transport}.
\newblock In \emph{Proceedings of the 41st International Conference on Machine
  Learning}, volume 235 of \emph{PMLR}, pages 8629--8650, 2024.

\bibitem[Neklyudov et~al.(2024)Neklyudov, Brekelmans, Tong, Atanackovic, Liu,
  and Makhzani]{Neklyudov24}
K.~Neklyudov, R.~Brekelmans, A.~Tong, L.~Atanackovic, Q.~Liu, and A.~Makhzani.
\newblock {A computational framework for solving Wasserstein Lagrangian flows}.
\newblock In \emph{Proceedings of the 41st International Conference on Machine
  Learning}, volume 235 of \emph{PMLR}, pages 37461--37485, 2024.

\bibitem[Hertrich and Gruhlke(2025)]{Hertrich25}
J.~Hertrich and R.~Gruhlke.
\newblock {Importance corrected neural JKO sampling}.
\newblock In \emph{Proceedings of the 42nd International Conference on Machine
  Learning}, volume 267 of \emph{PMLR}, pages 23083--23119, 2025.

\bibitem[Chen et~al.(2018)Chen, Rubanova, Bettencourt, and Duvenaud]{Chen18}
R.~T.~Q. Chen, Y.~Rubanova, J.~Bettencourt, and D.~Duvenaud.
\newblock {Neural ordinary differential equations}.
\newblock In \emph{Advances in Neural Information Processing Systems 31}, 2018.

\bibitem[Grathwohl et~al.(2019)Grathwohl, Chen, Bettencourt, Sutskever, and
  Duvenaud]{Grathwohl19}
W.~Grathwohl, R.~T.~Q. Chen, J.~Bettencourt, I.~Sutskever, and D.~Duvenaud.
\newblock {FFJORD: Free-form continuous dynamics for scalable reversible
  generative models}.
\newblock In \emph{International Conference on Learning Representations}, 2019.

\bibitem[Leshno et~al.(1993)Leshno, Lin, Pinkus, and Schocken]{Leshno93}
M.~Leshno, V.~Ya. Lin, A.~Pinkus, and S.~Schocken.
\newblock {Multilayer feedforward networks with a nonpolynomial activation
  function can approximate any function}.
\newblock \emph{Neural Networks}, 6:\penalty0 861--867, 1993.
\newblock \doi{10.1016/S0893-6080(05)80131-5}.

\end{thebibliography}
